\documentclass[10pt]{article}

\usepackage{xspace}
\usepackage{url}
\usepackage{mathtools}
\usepackage{amssymb}
\usepackage{amsthm}
\usepackage{empheq}
\usepackage{latexsym}
\usepackage{enumitem}
\usepackage{eurosym}
\usepackage{dsfont}
\usepackage{appendix}
\usepackage{color} 
\usepackage[unicode]{hyperref}
\usepackage{frcursive}
\usepackage[utf8]{inputenc}
\usepackage[T1]{fontenc}
\usepackage{geometry}
\usepackage{multirow}
\usepackage{todonotes}
\usepackage{lmodern}
\usepackage{anyfontsize}
\usepackage{stmaryrd}
\usepackage{natbib}
\usepackage{cleveref}
\usepackage[english]{babel}
\usepackage[english=british]{csquotes}
\usepackage{mathtools}
\usepackage{accents}
\usepackage{color}
\usepackage{bm}

\setcitestyle{numbers,open={[},close={]}}

\definecolor{red}{rgb}{0.7,0.15,0.15}
\definecolor{green}{rgb}{0,0.5,0}
\definecolor{blue}{rgb}{0,0,0.7}
\hypersetup{colorlinks, linkcolor={red}, citecolor={green}, urlcolor={blue}}
			
\makeatletter \@addtoreset{equation}{section}

\newtheorem{theorem}{Theorem}[section]
\newtheorem{assumption}[theorem]{Assumption}
\newtheorem{corollary}[theorem]{Corollary}

\newtheorem{lemma}[theorem]{Lemma}
\newtheorem{proposition}[theorem]{Proposition}

\newtheorem{definition}[theorem]{Definition}
\newtheorem{remark}[theorem]{Remark}

\DeclareUnicodeCharacter{014D}{\=o}
\newcommand{\diff}{\mathrm{d}}

\newcommand\bx{\mathbf{x}}

\def\balpha{\bm{\alpha}}

\newcommand\cA{\mathcal A}

\newcommand\cC{\mathcal C}

\newcommand\cE{\mathcal E}
\newcommand\cF{\mathcal F}
\newcommand\cG{\mathcal G}

\newcommand\cL{\mathcal L}

\newcommand\cO{\mathcal O}
\newcommand\cP{\mathcal P}

\newcommand\cW{\mathcal W}

\newcommand\EE{\mathbb E}
\newcommand\FF{\mathbb F}

\newcommand\PP{\mathbb P}
\newcommand\NN{\mathbb N}
\newcommand\RR{\mathbb R}
\def\SS{\mathbb S}

\newcommand\XX{\mathbb X}

\newcommand\ZZ{\mathbb Z}

\def \A{\mathbb{A}}

\def \E{\mathbb{E}}
\def \F{\mathbb{F}}
\def \G{\mathbb{G}}
\def \H{\mathbb{H}}

\def \L{\mathbb{L}}

\def \N{\mathbb{N}}

\def \P{\mathbb{P}}
\def \Q{\mathbb{Q}}
\def \R{\mathbb{R}}
\def \S{\mathbb{S}}

\def \X{\mathbb{X}}

\def \Z{\mathbb{Z}}

\def\Ac{\mathcal{A}}
\def\Bc{\mathcal{B}}
\def\Cc{\mathcal{C}}

\def\Fc{\mathcal{F}}
\def\Gc{\mathcal{G}}
\def\Hc{\mathcal{H}}

\def\Lc{\mathcal{L}}
\def\Mc{\mathcal{M}}

\def\Oc{\mathcal{O}}
\def\Pc{\mathcal{P}}

\def\Tc{\mathcal{T}}

\def\Wc{\mathcal{W}}

\def\Yc{\mathcal{Y}}
\def\Zc{\mathcal{Z}}

\def \tX{\widetilde{X}}
\def \tY{\widetilde{Y}}
\def \tZ{\widetilde{Z}}

\def \barX{\overline{X}}
\def \barY{\overline{Y}}
\def \barZ{\overline{Z}}

\newcommand{\x}{\mathbf{x}}

\newcommand{\xdim}{m}
\newcommand{\bmdim}{d}
\newcommand{\np}{N}

\newcommand{\e}{\mathrm{e}}

\newcommand{\smalltext}[1]{\text{\fontsize{4}{4}\selectfont$#1$}}
\newcommand{\tinytext}[1]{\text{\fontsize{3}{3}\selectfont$#1$}}

\def\eps{\varepsilon}
\def\d{\mathrm{d}}
\DeclareMathOperator*{\argmax}{arg\,max}

\DeclareMathOperator*{\esssup}{ess\,sup}

\begin{document}
\title{Non-asymptotic convergence rates for mean-field games: weak formulation and McKean--Vlasov BSDEs\footnote{Dylan Possama\"i gratefully acknowledges support from ANR project PACMAN ANR-16-CE05-0027. Ludovic Tangpi gratefully acknowledges support from the NSF grant DMS-2005832 and the NSF CAREER award DMS-2143861. Both authors thank Daniel Lacker and Mathieu Lauri\`ere for fruitful discussions.}}
\author{Dylan Possama\"i\footnote{ETH Z\"urich, Mathematics department, Switzerland, dylan.possamai@math.ethz.ch} $\qquad \qquad$  Ludovic Tangpi\footnote{Princeton University, ORFE, ludovic.tangpi@princeton.edu}}

\date{\today}

\maketitle


\abstract{
	This work is mainly concerned with the so-called limit theory for mean-field games. Adopting the weak formulation paradigm put forward by \citeauthor*{carmona2015probabilistic} \cite{carmona2015probabilistic}, we consider a fully non-Markovian setting allowing for drift control and interactions through the joint distribution of players' states and controls. We provide first a characterisation of mean-field equilibria as arising from solutions to a novel kind of McKean--Vlasov backward stochastic differential equations, for which we provide a well-posedness theory. We incidentally obtain there unusual existence and uniqueness results for mean-field equilibria, which do not require short-time horizon, separability assumptions on the coefficients, nor Lasry and Lions's monotonicity conditions, but rather smallness---or alternatively regularity---conditions on the terminal reward and a dissipativity condition on the drift. We then take advantage of this characterisation to provide non-asymptotic rates of convergence for the value functions and the Nash-equilibria of the $N$-player version to their mean-field counterparts, for general open-loop equilibria. An appropriate reformulation of our approach also allows us to treat closed-loop equilibria, and to obtain convergence results for the master equation associated to the problem.
}
\setlength{\parindent}{0pt}

\tableofcontents


\allowdisplaybreaks

\section{Introduction}\label{sec:intro}

This paper is concerned with the limit theory for so-called mean-field games with interactions through the controls, also sometimes refereed to as extended mean-field games in the literature. The classical theory of mean-field games dates back to the early 2000s, when they were independently introduced by \citeauthor*{lasry2006jeux} \cite{lasry2006jeux,lasry2006jeux2,lasry2007mean} and \citeauthor*{huang2003individual} \cite{huang2003individual,huang2006large,huang2007invariance,huang2007large,huang2007nash}, as a tractable alternative to studying symmetric Nash equilibria in non--zero-sum stochastic differential games involving a large number of players. The crux of their approach was to realise that for such $N$-player games where the state variables controlled by homogeneous players depended on their own state, and on the other players' states only through the latter's empirical distribution, a version of the game with infinitely many players would not only be more tractable theoretically, but would in turn provide `good' approximations for the original $N$-player game. We cannot provide an in-depth bibliography for mean-field games here, and instead urge our readers to go through the illuminating monographs by \citeauthor*{carmona2018probabilisticI} \cite{carmona2018probabilisticI,carmona2018probabilisticII} for additional background and references.

\medskip
Quantifying properly and rigorously what was meant by such a `good' approximation has been one of the most challenging problems in the early days of the theory. Proving on the one hand that the equilibria stemming from the mean-field game---the so-called mean-field equilibria---were actually $\eps$-Nash equilibria\footnote{Roughly speaking, a player deviating from an $\eps$-Nash equilibrium can at most increase their criterion by $\eps$.} for the $N$-player game was already achieved in the seminal papers mentioned above (see more precisely \cite{huang2006large}, as well as the more recent contributions by \citeauthor*{lacker2016general} \cite{lacker2016general} for general games of control, \citeauthor*{carmona2017mean} \cite{carmona2017mean} for games of timing, or \citeauthor*{cecchin2020probabilistic} \cite{cecchin2020probabilistic} for games with finitely many states). However, the converse direction, at least in a relatively general form, remained open for a while. There, the question becomes to understand in which sense the mean-field game and its equilibria arise as limits, in an appropriate sense, of the $N$-player games, as $N$ goes to infinity. Early results in that direction were obtained by \citeauthor*{lasry2006jeux} \cite{lasry2006jeux,lasry2007mean}, \citeauthor*{feleqi2013derivation} \cite{feleqi2013derivation}, \citeauthor*{gomes2013continuous} \cite{gomes2013continuous}, and \citeauthor*{bardi2014linear} \cite{bardi2014linear}, albeit by imposing relatively strong restrictions on the controls allowed for the players. The first comprehensive results for general open-loop controls were then obtained by \citeauthor*{fischer2017connection} \cite{fischer2017connection}, and especially \citeauthor*{lacker2016general} \cite{lacker2016general} who showed, in a nutshell, that all accumulation points of $N$-player's Nash equilibria were so-called weak mean-field equilibria, and conversely that any such weak mean-field equilibrium could be obtained as a limit of $\eps$-Nash equilibria. 

\medskip
The same question when considering closed-loop controls instead of open-loop ones turned out to be much more challenging. The first breakthrough came from \citeauthor*{cardaliaguet2019master} \cite{cardaliaguet2019master}, who showed that one could use smooth solutions to the so-called \emph{master equation}---a partial differential equation on the Wasserstein space characterising the value function of the mean-field game---in order to prove convergence in this case. Their approach was subsequently extended by \citeauthor*{cardaliaguet2017convergence} \cite{cardaliaguet2017convergence} for problems with local coupling, and by \citeauthor*{delarue2020master} \cite{delarue2019master,delarue2020master}, who managed to derive not only a central limit theorem, but also large deviation principles, as well as non-asymptotic bounds on various distances between a Nash equilibrium and its limit. A more general result with a probabilistic flavour and allowing for non-unique mean-field equilibria in the analysis was then obtained by \citeauthor*{lacker2020convergence} \cite{lacker2020convergence}, who related limits of Nash equilibria to what he coined \emph{weak semi-Markov mean-field equilibria}. We emphasise that such a limit theory is not always available for variants of the problem at hand. Hence, \citeauthor*{campi2018n} \cite{campi2018n} gave a counter-example in degenerate game with absorption to the fact that mean-field equilibria provided $\eps$-Nash equilibria. Similarly, if one is interested in knowing whether mean-field equilibria arise as limits of Nash equilibria (and not just of $\eps$-Nash equilibria), \citeauthor*{nutz2020convergence} \cite{nutz2020convergence} showed in an optimal stopping game that this was not true in general (see also \citeauthor*{cecchin2019convergence} \cite{cecchin2019convergence} and \citeauthor*{delarue2020selection} \cite{delarue2020selection} for related results).

\medskip
Almost all the aforementioned references consider only what we already referred to as classical mean-field games. Mean-field games with interaction through the controls are the ones for which the dynamics of the states of each player not only depend on the distribution of other players' states, but also on the distribution of other players' controls. Such games were introduced by \citeauthor*{gomes2014existence} \cite{gomes2014existence} and \citeauthor*{gomes2016extended} \cite{gomes2016extended} (see also \citeauthor*{graber2016linear} \cite{graber2016linear}, \citeauthor*{elie2019tale} \cite{elie2019tale}, or \citeauthor*{alasseur2020extended} \cite{alasseur2020extended} for results in specific models). 
The first associated general study is due to \citeauthor*{carmona2015probabilistic} \cite{carmona2015probabilistic}, see also \citeauthor*{bertucci2019some} \cite{bertucci2019some}, \citeauthor*{kobeissi2022classical} \cite{kobeissi2022classical} and \citeauthor*{djete2023mean} \cite{djete2023mean}, the latter being the first general treatment in the literature of a limit theory for extended mean-field games with common noise.

\medskip
A specific feature of all the previously mentioned results, with the notable exception of \cite{delarue2020master}, is that they all provide convergence or compactness results, but do not quantify any non-asymptotic error estimates between Nash and mean-field equilibria. This gap in the literature motivated a recent take on the problem by \citeauthor*{lauriere2022convergence} \cite{lauriere2022convergence}. They considered Markovian symmetric stochastic differential games where players' states were controlled only through their drift, but were allowed to depend on the joint (empirical) distribution of the other players' states and controls. They then built up a three-step approach allowing them to obtain explicit convergence rates in $\L^2$-norm as well as concentration inequalities between Nash and mean-field equilibria. Roughly speaking, their approach proceeds as follows
\begin{itemize}
\item[$(i)$] first use Pontryagin's maximum principle to characterise Nash equilibria in the $N$-player game by a fully coupled system of forward--backward SDE (FBSDE for short);
\item[$(ii)$] second use again the maximum principle to characterise mean-field equilibria through FBSDEs of McKean--Vlasov type;
\item[$(iii)$] argue using techniques from \emph{backward propagation of chaos}, developed by \citeauthor*{lauriere2022backward} \cite{lauriere2022backward}, that as $N$ goes to $\infty$, the FBSDE derived in $(i)$ converges appropriately to the McKean--Vlasov FBSDE in $(ii)$.
\end{itemize}
Using the stochastic maximum principle in mean-field game theory is the heart of the probabilistic approach developed by \citeauthor*{carmona2013probabilistic} \cite{carmona2013mean,carmona2013probabilistic,carmona2018probabilisticI,carmona2018probabilisticII}, and is the one typically followed in the literature not using analytical tools, such as the master equation mentioned above. Readers familiar with stochastic control theory will however recall that there is an alternative approach for these problems, namely Bellman's optimality principle, also often referred to as the dynamic programming principle (DPP for short), see \citeauthor*{yong1999stochastic} \cite{yong1999stochastic} for a classical take on these two complementary approaches. Unlike Pontryagin's maximum principle which aims at characterising optimal controls (or in our case equilibria), the DPP characterises value functions directly, and allows to characterise optimal controls only incidentally. Though the range of problems in which the maximum principle can be applied is typically larger, since it allows to tackle time-inconsistent optimisation problems for which the DPP is not satisfied, in settings where the DPP also holds, one can generally use it under much weaker assumptions. It is therefore somewhat surprising that the literature on mean-field games relied only very rarely on such a DPP approach. As far as we know, the major exception is the paper by \citeauthor*{carmona2015probabilistic} \cite{carmona2015probabilistic}, which developed a weak formulation approach for extended mean-field games with drift control by linking them to backward SDEs (BSDEs for short), as well as the recent extension to volatility control by \citeauthor*{barrasso2022controlled} \cite{barrasso2022controlled}; see also \citeauthor*{elie2019tale} \cite{elie2019tale}, \citeauthor*{elie2021mean} \cite{elie2021mean}, and \citeauthor*{carmona2016finite} \cite{carmona2016finite} for similar takes on mean-field games within the context of contract theory.

\medskip
Our goal in this paper is to address the problem of proving convergence of value functions and Nash equilibria for $N$-player stochastic differential games to their mean-field game counterparts, and to obtain quantitative rates of convergence in a general non-Markovian setting. The problem and the idea to tackle it is in spirit close to  \cite{lauriere2022convergence}. Indeed, our approach relies on the following steps
\begin{itemize}
\item[$(i)$] first use the DPP to characterise value functions in the $N$-player game by a multi-dimensional system of BSDEs, and Nash equilibria as `fixed-points' of the corresponding vector-valued Hamiltonian;
\item[$(ii)$] second use again the DPP to characterise the value function of the mean-field game by \emph{a new type} of McKean--Vlasov BSDE, and mean-field equilibria as maximisers of the corresponding Hamiltonian;
\item[$(iii)$] use general \emph{backward propagation of chaos} arguments to prove that, on a suitable probability space, as $N$ goes to $\infty$, the BSDE derived in $(i)$ converges appropriately to the McKean--Vlasov BSDE in $(ii)$.
\end{itemize}

Despite the seemingly similar approach, the techniques we use are fundamentally different in nature, since we favour the DPP approach, and we are working with the weak formulation of mean-field games. As such, if the result in the first step is part of the folklore on stochastic differential games\footnote{\label{foot:intro}Proofs of related results appear notably in \citeauthor*{hamadene1997bsdes} \cite{hamadene1997bsdes}, \citeauthor*{hamadene1998backward} \cite{hamadene1998backward}, \citeauthor*{el2003bsdes} \cite{el2003bsdes}, \citeauthor*{lepeltier2009nash} \cite{lepeltier2009nash}, \citeauthor*{hamadene2021risk} \cite{hamadene2015existence,hamadene2021risk}, \citeauthor*{frei2011financial} \cite{frei2011financial}, \citeauthor*{espinosa2015optimal} \cite{espinosa2015optimal}, \citeauthor*{elie2019contracting} \cite{elie2019contracting}, \citeauthor*{baldacci2021optimal} \cite{baldacci2021optimal} or \citeauthor*{jusselin2021optimal} \cite{jusselin2021optimal}. }, our characterisation in the second step is, as far as we know, completely new. 
It introduces a new class of BSDEs where both the driving Brownian motion and the underlying probability measure are to be found as part of the solution, and depend on it in a non-linear way, in the sense that the Radon--Nikod\'ym density of that measure depends on the solution itself. We offer a well-posedness result for these new equations.
The proof significantly departs from classical approaches in the corresponding literature. We believe this equation to be interesting in and of itself.

\medskip
As such, solvability of this new class of BSDEs provides us with an alternative approach to study existence and uniqueness of general mean-field games with interactions through the controls.
In many ways, our existence and uniqueness result weakens the conditions in the extant literature. 
More precisely, besides imposing relatively standard Lipschitz-continuity assumptions on the data of the problem---including an important dissipativity condition on the drift---and the maximisers of the players' (reduced) Hamiltonian, the somewhat `restrictive' conditions are put on the terminal reward. In fact, we either assume it to be small or to be sufficiently smooth. Nonetheless, these conditions are immediately satisfied in most cases, for instance for problems without a terminal reward or with a reward of quadratic type, making the dissipativity assumption the only really strong assumption in our setting.
However, unlike similar results in \cite{carmona2015probabilistic} or \cite{carmona2018probabilisticI,carmona2018probabilisticII} for mean-field games (with interactions through the controls), we do not need to assume restrictive structural or separability conditions on the running reward and hence on the Hamiltonian.
In addition, our uniqueness result does not rely on the celebrated \emph{Lasry--Lions monotonicity} conditions or the newer \emph{displacement monotonicity} condition of \citeauthor*{gangbo2022mean} \cite{gangbo2022mean} and \citeauthor*{jackson2023quantitative} \cite{jackson2023quantitative}.
This aspect of our approach seems notable. Let us also mention the recent work by \citeauthor*{djete2023mean} \cite{djete2023mean}, which, unlike ours and the aforementioned ones allows players to control the volatility of their state processes, proves existence of $\eps$-strong mean-field equilibria, as well as what he coins measure-valued mean-field equilibria, in a very general setting with interactions through the controls, but which also requires stringent separability conditions.

\medskip
Let us now discuss our results on the convergence to the mean-field game limit, which is the main contribution of the paper.
Our results underline the remarkable effect played by the terminal reward of the game.
In fact, for games without terminal reward, we prove a general convergence result of the value function of the $N$-player game to that of the mean-field game with an explicit convergence rate, under similar assumptions to the ones which allow us to prove well-posedness for our generalised McKean--Vlasov BSDEs. In particular, when the game has a non-trivial terminal reward, we impose dissipativity conditions on the drift, and sufficient smoothness of the terminal reward.
Overall, our results in terms of convergence of Nash equilibria compare to those in \cite{cardaliaguet2019master} and \cite{lauriere2022convergence} as follows
\begin{enumerate}
\item[$(i)$] we can work with general non-Markovian dynamics: this is the first such result in the literature, since both the maximum principle approach or the analytical approach through the master equation are inherently limited to the Markovian case;

\item[$(ii)$] unlike \cite{cardaliaguet2019master,lauriere2022convergence} which assume a constant volatility for the state variables of the players, our main result allows to have general non-Markovian, uncontrolled, volatilities;

\item[$(iii)$] our approach is purely probabilistic and does not require existence of the master equation of the mean-field game or a bound on its second derivative as in \cite{cardaliaguet2019master}. Moreover, we derive convergence of the value function of each player (not of an average) and convergence of the sequence of Nash equilibria (\emph{i.e.} of the controls);

\item[$(iv)$] thanks to our approach using weak formulation for optimal control problems, and even if we are initially only considering open-loop equilibria, under modest additional assumptions we can, as in \cite{carmona2015probabilistic}, obtain results on closed-loop controls as well. This makes our take on the problem slightly more flexible than the aforementioned references on the limit theory for closed-loop controls.
\end{enumerate}

Of course, as in standard control theory, the DPP approach is not a replacement for the maximum principle approach: they both have their own advantages and drawbacks. For instance, the approach of \cite{lauriere2022convergence} covers also the limit theory for the optimal control of McKean--Vlasov equations, while ours does not readily extend to that setting. Our contribution is to show how one can leverage the DPP approach to get quantitative estimates for the limit theory for mean-field games at a level of generality inaccessible with existing alternative techniques. As also illustrated by \citeauthor*{elie2021mean} \cite{elie2021mean} and \citeauthor*{barrasso2022controlled} \cite{barrasso2022controlled}, where mean-field games with volatility control are related to second-order BSDEs of McKean--Vlasov type similar in spirit to our McKean--Vlasov BSDEs, our approach also has the potential to be extended to more general games, with both volatility control and common noise. The volatility control case being significantly harder to deal with using Pontryagin's maximum principle, and requiring typically strong structural assumptions, our approach could prove more successful there as well.
Further observe that we do not treat games with common noise here, and these interesting problems are left for future research.

\medskip
The paper is organised as follows: \Cref{sec:pres} introduces both the $N$-player game and the mean-field game, and presents our main results, namely for convergence of Nash equilibria in \Cref{thm:main.limit}, for existence and uniqueness of mean-field equilibria in \Cref{thm:main.existence.MFG}. \Cref{sec:examples} also provides two examples of application, and \Cref{sec:consequences} explores implications of our results for convergence of solutions to the master equation, and for closed-loop controls. \Cref{sec:limitth} is dedicated to the proof of our limit theorems, and also contains our BSDE characterisations for Nash and mean-field equilibria, see \Cref{thm:Bellman,prop:char.mfe} as well as the case study of a toy example where our method is put into action. 
The final section of the paper studies a new class of BSDEs used in the proof existence of mean-field equilibria.

\medskip
{\footnotesize
\textbf{Notations:} Let $\mathbb{N}^\star \coloneqq \mathbb{N}\setminus\{0\}$ and let $\mathbb{R}_+^\star $ be the set of real positive numbers. Fix an arbitrary Polish space $E$ endowed with a metric $d_E$. Throughout this paper, for every $p$-dimensional $E$-valued vector $e$ with $p\in \mathbb{N}^\star $, we denote by $e^{1},\ldots,e^{p}$ its coordinates, and for any $i\in\{1,\dots,p\}$, by $e^{-i}\in E^{p-1}$ the vector obtained by suppressing the $i$-th coordinate of $e$.
For $(\alpha,\beta) \in \R^p\times\R^p$, we also denote by $\alpha \cdot \beta$ the usual inner product, with associated norm $\|\cdot\|$, which we simplify to $|\cdot|$ when $p$ is equal to $1$. For any $(\ell,c)\in\mathbb N^\star \times\mathbb N^\star $, $E^{\ell\times c}$ will denote the space of $\ell\times c$ matrices with $E$-valued entries. Elements of the matrix $M\in E^{\ell\times c}$ will be denoted by $(M^{i,j})_{(i,j)\in\{1,\dots,\ell\}\times\{1,\dots,c\}}$, and the transpose of $M$ will be denoted by $M^\top$. We identify $E^{\ell\times 1}$ with $E^\ell$. The trace of a matrix $M\in E^{\ell\times \ell}$ will be denoted by $\mathrm{Tr}[M]$. 
For any $x\in E^{\ell\times c}$ and $y\in E^\ell$, we also define, for any $i\in\{1,\dots,c\}$, $y\otimes_ix\in E^{\ell\times (c+1)}$ as the matrix whose column $j\in\{1,\dots,i-1\}$ is equal to the $j$-th column of $x$, whose column $j\in\{i+1,\dots,c+1\}$ is equal to the $(j-1)$-th column of $x$, and whose $i$-th column is $y$. We also abuse notations and extend these notations to $E^{\ell\times c}$-valued processes. It will often happen that we consider elements of with an upper index $N$, say $M^N\in E^{\ell\times c}$ or $x^N\in E^c$ for some $(\ell,c)\in\N^\star\times\N^\star$. In those cases, we write for any $(i,j)\in\{1,\dots,\ell\}\times\{1,\dots,c\}$, $x^{i,N}$, $x^{-i,N}$, $M^{i,j,N}$, instead of $(x^N)^i$, $(x^N)^{-i}$, $(M^N)^{i,j}$.

\medskip
Given a positive integer $\ell$, and a vector $x \in E^\ell$, for notational simplicity, we will always denote by
\[
L^\ell(x)\coloneqq \frac1{\ell}\sum_{j=1}^{\ell}\delta_{x^j},
\] 
the empirical measure associated to $x$. For any $p>0$, we also denote by $\Pc(E)$ the set of probability measures on $E$ (endowed with its Borel $\sigma$-algebra) and by $\Pc_p(E)$ the subset of $\Pc(E)$ containing measures with finite $p$-th moment. Notice then that for any $x\in E^\ell$, we have $L^\ell(x)\in\Pc_p(E)$, for any $p>0$.

\medskip
Let $\Bc(E)$ be the Borel $\sigma$-algebra on $E$ (for the topology generated by the metric $d_E$ on $E)$. For any $p\geq 1$, for any two probability measures  $\mu$ and $\nu$ on $(E,\Bc(E))$ with finite $p$-moments, we denote by $\cW_{p}(\mu, \nu)$ the $p$-Wasserstein distance between $\mu$ and $\nu$, that is
\begin{equation*}
	\cW_{p}(\mu, \nu) \coloneqq  \bigg(\inf_{\pi\in\Gamma(\mu,\nu)}\int_{E\times E}d(x,y)^p\pi(\mathrm{d}x,\mathrm{d}y) \bigg)^{1/p},
\end{equation*}
where the infimum is taken over the set $\Gamma(\mu,\nu)$ of all couplings $\pi$ of $\mu$ and $\nu$, that is, probability measures on $\big(E^2,\Bc(E)^{\otimes 2}\big)$ with marginals $\mu$ and $\nu$ on the first and second factors respectively.

\medskip
We fix throughout the paper a time horizon $T>0$, and for any positive integer $k$, we let $\Cc_k$ be the space of continuous functions from $[0,T]$ to $\R^k$. Besides, for any $(x,y)\in\Cc_k\times\Cc_k$, we write 
\[
\|x-y\|_\infty\coloneqq \sup_{t\in[0,T]}\|x(t)-y(t)\|.\]
When $k=1$, we simplify the notation to $\Cc\coloneqq \Cc_1$. We will also use the notation $\|f\|_\infty$ to denote the (smallest) upper bound of any bounded function $f$ defined on appropriate spaces.

}

\section{Stochastic games in weak formulation: setting and main results}\label{sec:pres}

\subsection{Probabilistic setting}\label{sec:setting}
Let us describe the stochastic differential game we are interested in.
We fix three positive integers $\np$, $\xdim$ and $\bmdim$, which represent respectively the number of players, the dimension of the state process of each player, and the dimension of the Brownian motions driving these state processes. 
We fix a probability space $(\Omega,\Fc,\P)$ carrying a sequence of independent, $\R^d$-valued Brownian motions $(W^i)_{i\in\N^\star}$, and for any $i\in\N^\star$, we denote by $\F^i\coloneqq (\Fc^i_t)_{t\in[0,T]}$ the $\P$-completed natural filtration of $W^i$. Expectations (resp. conditional expectations) under $\P$ will always be denoted using the symbol $\E$, and we will precise the measure whenever expectations (resp. conditional expectations) are taken under a measure different from $\P$.

\medskip
Throughout this work we fix a Borel-measurable map $\sigma:[0,T]\times\Cc_m\longrightarrow \RR^{\xdim\times \bmdim}.$ Our main condition on $\sigma$ is the following, which is assumed to hold throughout the paper.
\begin{assumption}\label{assump:sigma}
	Fix some $\R^m$-valued sequence $(X_0^i)_{i\in\N^\star}$. 
	The function $\sigma$ is uniformly bounded in its first variable, of linear growth in the second one, and for any $i\in\N^\star$, there exists a unique strong solution $X^i$ on $(\Omega,\Fc,\P)$ of the {\rm SDE}
	\begin{equation}\label{eq:start def X}
		X^i_t=X_0^i +\int_0^t\sigma_s(X^i_{\cdot\wedge s})\mathrm{d}W^i_s,\; t\in[0,T],\; \P\text{\rm--a.s.}
	\end{equation}
\end{assumption}
\begin{remark}
	It is well-known that the existence and uniqueness of a strong solution for the {\rm SDE} appearing in {\rm \Cref{assump:sigma}} is guaranteed as soon as $\sigma$ is, for instance, uniformly Lipschitz-continuous with linear growth with respect to its second variable $($for the supremum metric on $\Cc_m)$.
	It is also obvious that the processes $(X^i)_{i\in\N^\star}$ are $\P$-independent.
\end{remark}
We will simplify notations when $i=1$, and define $X\coloneqq X^1$, $W\coloneqq W^1$, as well as $\Fc_t\coloneqq \Fc_t^{1}$, $t\in[0,T]$. It will also be useful to define the $N$-fold product filtration $\F_N\coloneqq (\Fc_{N,t})_{t\in[0,T]}$, where for any $t\in[0,T]$, $\Fc_{N,t}$ is the $\P$-completion of $\bigotimes_{i=1}^N\Fc^i_t$. We also denote by $\X^N$ the $\R^{m\times N}$-valued process $(X^1,\dots, X^N)$. Recall that $\FF_\np$ and all the $(\F^i)_{i\in\N^\smalltext{\star}}$ satisfy, under $\P$, the usual conditions. Thus, we know that $\FF_\np$ (resp. all the filtrations $(\F^i)_{i\in\N^\star}$) satisfy the martingale representation property, meaning that any $(\FF_\np,\PP)$-martingale (resp. for any $i\in\N^\star$, any $(\F^i,\P)$-martingale) can be represented as a stochastic integral with respect to $(W^i)_{i\in\{1,\dots,N\}}$ (resp. for any $i\in\N^\star$, with respect to $W^i$).

\medskip
A number of spaces will play an important role in the paper. Let therefore $(E,\|\cdot\|_E)$ be a generic finite-dimensional normed vector space, $\G$ a generic filtration, and $\Gc$ a generic sub--$\sigma$-algebra of $\Fc$ in our probability space $(\Omega,\Fc,\P)$. We also let $\Tc(\G)$ be the set of $\G$--stopping times taking values in $[0,T]$.

\medskip
$\bullet$ For any $p\in[1,\infty]$, $\L^p(E,\Gc)$ is the space of $E$-valued, $\Gc$-measurable random variables $R$ such that 
\[
	\|R\|_{\L^\smalltext{p}(E,\Gc)}\coloneqq \Big(\E\big[\|R\|_E^p\big]\Big)^{\frac1p}<\infty,\; \text{when}\; p<\infty,\; \|R\|_{\L^\smalltext{\infty}(E,\Gc)}\coloneqq \inf\big\{\ell\geq0:\|R\|_E\leq \ell,\; \P\text{\rm --a.s.}\big\}<\infty.
\]

\medskip
$\bullet$ For any $p\in[1,\infty)$, $\H^p(E,\G)$ is the space of $E$-valued, $\G$-predictable processes $Z$ such that 
\begin{equation*}
	\|Z\|_{\H^\smalltext{p}(E,\G)}^p\coloneqq \EE\bigg[\bigg(\int_0^T\|Z_s\|_E^2\mathrm{d}s\bigg)^{p/2}\bigg]<\infty.
\end{equation*}

$\bullet$ $\H^2_{\rm BMO}(E,\G)$ is the space of $E$-valued, $\G$-predictable processes $Z$ such that
\begin{equation*}
	\|Z\|_{\H^2_{\text{\fontsize{4}{4}\selectfont$\mathrm{BMO}$}}(E,\G)}^2 \coloneqq  \sup_{\tau\in\Tc(\G)}\bigg\|\EE\bigg[\int_\tau^T\|Z_s\|_E^2\d s\bigg| \cG_{\tau} \bigg]\bigg\|_{\L^\smalltext{\infty}(E,\Gc_{\text{\fontsize{4}{4}\selectfont$T$}})} <\infty.
\end{equation*}

$\bullet$ For any $p\in[1,\infty]$, $\S^p(E,\G)$ is the space of $E$-valued, continuous, $\G$-adapted processes $Y$ such that 
\begin{equation*}
	\|Y\|_{\S^\smalltext{p}(E,\G)}\coloneqq \bigg(\EE\bigg[\sup_{t\in[0,T]}\|Y_t\|_E^p\bigg]\bigg)^{\frac1p}<\infty,\;  \text{when}\; p<\infty,\; \|Y\|_{\S^\smalltext{\infty}(E,\G)}\coloneqq \bigg\|\sup_{t\in[0,T]}\|Y_t\|_E\bigg\|_{\L^\smalltext{\infty}(E,\Gc_{\text{\fontsize{4}{4}\selectfont$T$}})}<\infty.
\end{equation*}

We will sometimes need to consider those spaces but associated to another probability measure $\Q$ on $(\Omega,\Fc)$. In this case, we will adjust our notations to $\L^p(E,\Gc,\Q)$, $\H^p(E,\G,\Q)$, $\H^2_{\rm BMO}(E,\G,\Q)$ and $\S^p(E,\G,\Q)$.

\subsection{The finite-player game}\label{sec:finitegame}

Let $A$ be a non-empty compact\footnote{We assume here compactness of $A$ mostly for simplicity and to alleviate integrability considerations which, we believe, would distract the reader from our main arguments. An extension to the unbounded case following similar lines is possible, but would require more sophisticated estimates (or need stronger growth conditions on $f$).} Polish space endowed with a metric $\bar d$, whose Borel $\sigma$-algebra is denoted by $\Bc(A)$.
Consider the drift function
\[
	b:[0,T]\times\Cc_m\times \cP_2(\Cc_m\times A)\times A\longrightarrow \RR^\bmdim.
\]
The function $b$ is assumed to be Borel-measurable with respect to all its arguments.
We define for any $\alpha\coloneqq (\alpha^i)_{i\in\{1,\dots, N\}}$, where each $\alpha^i$ is an $A$-valued $\F_N$-predictable process, the probability measure $\P^{\alpha,N}$ on $(\Omega,\cF)$, whose density with respect to $\PP$ is given by
\[
	\frac{\mathrm{d}\PP^{\alpha,N}}{\mathrm{d}\PP}\coloneqq \cE\bigg(\int_0^\cdot\sum_{i=1}^\np b_s\big(X^i_{\cdot\wedge s},L^N(\XX^N_{\cdot\wedge s},\alpha_s),\alpha^i_s\big)\cdot \mathrm{d}W^i_s\bigg)_T,
\]
where $\cE(M)_\cdot \coloneqq  \exp(M_\cdot - 1/2[ M]_\cdot )$ denotes the stochastic exponential of the continuous local martingale $M$.
The class of \emph{admissible} strategy profiles $\cA^N$ in the $N$-player game is the set of $\F_N$-predictable, $A^N$-valued processes $\alpha = (\alpha^i)_{i \in \{1,\dots,N\}}$.
We denote by $\cA$ the set of strategies $\alpha^i$ such that $(\alpha^j)_{j\in \{1,\dots,N\}} \in \cA^N$ for some $(\alpha^j)_{j\in \{1,\dots,N\}\setminus\{i\}}$.
Observe that we have for $i\in\{1,\dots,\np\}$
\begin{equation}
\label{eq:state.W.alpha}
	X^i_t=X^i_0+\int_0^t\sigma_s(X^i_{\cdot\wedge s}) b_s\big(X^i_{\cdot\wedge s},L^N(\XX_{\cdot\wedge s},\alpha_s),\alpha^i_s\big)\mathrm{d}s+\int_0^t\sigma_s(X^i_{\cdot\wedge s})\mathrm{d}W_s^{\alpha,i},\; t\in[0,T],
\end{equation}
where by Girsanov's theorem, for any $i\in\N^\star$
\[
	W^{\alpha,i}_\cdot\coloneqq \begin{cases}
\displaystyle W^i_\cdot-\int_0^\cdot b_s\big(X^i_{\cdot\wedge s},L^N(\XX^N_{\cdot\wedge s},\alpha_s),\alpha^i_s\big)\mathrm{d}s,\; \text{if}\; i\in\{1,\dots,N\},\\[0.5em]
\displaystyle W^i_\cdot,\; \text{if}\; i\geq N+1,
\end{cases}
\]
is an $\R^d$-valued, $\PP^{\alpha,N}$--Brownian motion.

\medskip
For any $\alpha\in\Ac^N$, and any $i\in\{1,\dots,N\}$, we formulate the control problem of player $i$, given that other players have played $\alpha^{-i}$ as 
\[
    	V^{i,N}(\alpha^{-i}) \coloneqq \sup_{\alpha\in\cA}\EE^{\PP^{\smalltext{\alpha}\smalltext{\otimes}_{\tinytext{i}}\smalltext{\alpha}^{\tinytext{-}\tinytext{i}}\smalltext{,}\smalltext{N}}}\bigg[\int_0^Tf_s\big(X^i_{\cdot\wedge s},L^N(\XX^N_{\cdot\wedge s},\alpha\otimes_i\alpha^{-i}_s),\alpha_s\big)\mathrm{d}s+g\big(X^i,L^N(\XX^N)\big)\bigg],
\]
for a given terminal reward $g:\Cc_m\times \cP_2(\Cc_m)\longrightarrow \RR$, and running reward $f:[0,T]\times\Cc_m\times \cP_2(\Cc_m\times A)\times A\longrightarrow \RR$, which are both assumed to be Borel-measurable.
This is a general stochastic differential game in the weak formulation.
As usual, we are interested in Nash equilibria defined as follows.
\begin{definition}
\label{def:Nash}
	A Nash equilibrium is a family of $\np$ control processes $\hat\alpha^N\in\cA^\np$ such that for any $i\in\{1,\dots,\np\}$, we have
	\[
		V^{i,N}\big(\hat\alpha^{N,-i}\big)=\EE^{\PP^{\smalltext{\hat\alpha}^{\tinytext{N}}\smalltext{,}\smalltext{N}}}\bigg[\int_0^Tf_s\big(X^i_{\cdot\wedge s},L^N(\XX^N_{\cdot\wedge s},\hat\alpha^N_s),\hat\alpha^{N,i}_s\big)\mathrm{d}s+g\big(X^i,L^N(\XX^N)\big)\bigg].
	\]
	We denote by $\mathcal{NA}$ the set of all Nash equilibria. 
\end{definition}
We can now state our main assumptions on $f$, $g$ and $b$.
\begin{assumption}
\label{ass.Lambda.charac}

$(i)$ The function $f$ satisfies that there is a constant $\ell_f>0$ and some $a_o\in A$ such that for all $(t,\x,a,\xi)\in [0,T]\times \cC_m\times A\times \cP_2(\cC_m\times A)$
	\begin{align*}
		|f_t(\bx,\xi,a)| &\le \ell_f\bigg(1 +\bar d^2(a,a_o) + \|\x\|_\infty^2  + \int_{\cC_\smalltext{m}\times A}\big(\|x\|^2_\infty+\bar d^2(e,a_o)\big)\xi(\diff x,\diff e) \bigg);
	\end{align*}
	
	\medskip
	
	$(ii)$ the map $g$ satisfies that for a constant $\ell_g>0$ and for all $(\x,\xi)\in\Cc_m\times\Pc_2(\Cc_m)$
	\[
	|g(\x,\xi)|\leq \ell_g\bigg(1+\|\x\|_\infty^2 +\int_{\cC_\smalltext{m}}\|x\|^2_\infty\xi(\diff x) \bigg);
	\]
	
	$(iii)$ the map $b$ is bounded.
	
\end{assumption}
We start by providing a characterisation of Nash equilibria that will serve us when studying the convergence problem.
To this end, we need to introduce some preliminary notations.
Consider the function $h$ given by
\begin{equation}
\label{eq:h.def}
	h_t(\x,\xi,z,a) \coloneqq  b_t(\x, \xi, a)\cdot z + f_t(\x, \xi, a), \; (t, \x, \xi, z, a) \in [0,T]\times \Cc_m\times \cP_2(\Cc_m\times A)\times \RR^d\times A.
\end{equation}
Elements of the argmax of $h$ will play a fundamental role in what follows, which is why we introduce as well the set
\[
\A(t,\x,\xi,z)\coloneqq \argmax_{a\in A}\big\{h_t(\x,\xi,z,a)\big\},\; (t, \x, \xi, z) \in [0,T]\times \Cc_m\times \cP_2(\Cc_m\times A)\times \RR^d,
\]
and we let $\A$ be the set of all Borel-measurable maps $\hat a$ from $[0,T]\times \Cc_m\times \cP_2(\Cc_m\times A)\times \RR^d$ to $A$ such that for any $(t, \x, \xi, z) \in [0,T]\times \Cc_m\times \cP_2(\Cc_m\times A)\times \RR^d$
\[
\hat a(t, \x, \xi, z)\in \A(t, \x, \xi, z).
\]
We also define the map $H:[0,T]\times \Cc_m\times \cP_2(\Cc_m\times A)\times \RR^d\longrightarrow \R$
\[
H_t(\x,\xi,z)\coloneqq \sup_{a\in A}\big\{h_t(\x,\xi,z,a)\big\},\; (t, \x, \xi, z) \in [0,T]\times \Cc_m\times \cP_2(\Cc_m\times A)\times \RR^d.
\]
The function $H$ is naturally related to the Hamiltonian of the control problem faced by a representative player in the mean-field game we will describe in \Cref{sec:MFG}, and the elements of the argmax of $h$ will be related to mean-field equilibria. However the corresponding notions in the $N$-player game need to be adjusted, which is what we now do. Let us thus introduce the map $H^N:[0,T]\times\Cc_{m}^N\times (\R^{d})^{N\times N}\times A\times A^{N}\longrightarrow \R^N$, which is such that for any $(t,\x,z,a,e)\in[0,T]\times \Cc_{m\times N}\times (\R^{d})^{N\times N}\times A\times A^{N}$
\[
H^N_t(\x,z,a,e)\coloneqq \begin{pmatrix}
h_t\big(\x^1,L^N(\x,a\otimes_1e^{-1}),z^{1,1},a^1\big)+\sum_{j\in\{1,\dots,N\}\setminus\{ 1\}}b_t\big(\x^j,L^N(\x,a \otimes_1e^{-1}),e^j\big)\cdot z^{1,j}\\\
\vdots\\
h_t\big(\x^N,L^N(\x,a\otimes_Ne^{-N}),z^{N,N},a^N\big)+\sum_{j\in\{1,\dots,N\}\setminus\{ N\}}b_t\big(\x^j,L^N(\x,a \otimes_Ne^{-N}),e^j\big)\cdot z^{N,j}
\end{pmatrix}.
\]

We can now formalise what we mean by fixed-points for $H^N$.
\begin{definition}
\label{def:fixed-point-Hc}
For any $(t,\x,z)\in[0,T]\times\Cc_m\times (\R^{d})^{N\times N}$, a fixed-point of $H^N$ is a vector $a \in A^N$ such that for any $i\in\{1,\dots,N\}$
\[
a^i\in \argmax_{a^{\text{\fontsize{4}{4}\selectfont$\prime$}}\in A}\Bigg\{h_t\big(\x^i,L^N(\x,a^\prime \otimes_ia^{-i}),z^{i,i},a^\prime\big)+\sum_{j\in\{1,\dots,N\}\setminus\{ i\}}b_t\big(\x^j,L^N(\x,a^\prime \otimes_ia^{-i}),a^j\big)\cdot z^{i,j}\Bigg\}.
\]
For every $(t,\x,z)\in[0,T]\times\Cc_{m\times N}\times (\R^{d})^{N\times N}$, we denote by $\Oc^N(t,\x,z)$ the corresponding set, and we note that a fixed-point of $H^N$ is a map $\hat a:[0,T]\times\Cc_{m\times N}\times (\R^{d })^{N\times N}\longrightarrow A^N$ such that for any $(t,\x,z)\in[0,T]\times\Cc_{m\times N}\times (\R^{d})^{N\times N}$, $\hat a(t,\x,z)\in \Oc^N(t,\x,z)$. The corresponding set of all fixed-points of $H^N$ is denoted by $\Oc^N$.
\end{definition}

We are now ready for the following result which provides a necessary condition on Nash equilibria for the $N$-player game.
\begin{proposition}
\label{thm:Bellman}
	If $\hat\alpha^N\in\Ac^N$ is a Nash equilibrium for the $\np$-player game, then for each $i \in \{1, \dots, \np\}$
	\begin{equation}\label{eq:defhatalpha}
	\hat\alpha^{i,N}_t  \in \cO^\np\big(t, \X^N_{\cdot\wedge t}, Z^N_t\big),\; \d t\otimes\d \P\text{\rm--a.e.},
	\end{equation}
	where $(Y^N,Z^N)\coloneqq (Y^{i,\np}, Z^{i,j,\np})_{(i,j)\in\{1,\dots,\np\}^\smalltext{2} }$ is a solution to the coupled system of {\rm BSDEs}
	\begin{align}
	 \label{eq:bsde.main}
	 	Y^{i,\np}_t  &=  g\big(X^i, L^\np(\XX^N)\big) + \int_t^Tf_s\big(X^i_{\cdot\wedge s}, L^{\np}(\XX^N_{\cdot\wedge s}, \hat{\alpha}^N_s), \hat\alpha_s^{i,\np} \big)\mathrm{d}s -\sum_{j=1}^\np\int_t^T Z^{i,j,\np}_s\cdot \mathrm{d}W^{\hat\alpha^{\text{\fontsize{4}{4}\selectfont $N$}},j}_s, \; t\in[0,T],\; \PP^{\hat\alpha^{\text{\fontsize{4}{4}\selectfont $N$}},N}\text{\rm--a.s.}
	 \end{align}
	Finally, the value function of the $i$-th player satisfies $V^{i,N}\big(\hat{\alpha}^{-i,N}\big) = Y^{i,\np}_0$. 
\end{proposition}

\subsection{The mean-field game}\label{sec:MFG}
Let us now describe the mean-field game formally associated to the $\np$-player game introduced in \Cref{sec:finitegame}. 
We work on the space $(\Omega,\Fc,\P)$ defined in \Cref{sec:setting}. We let $\mathfrak P$ be the set of Borel-measurable maps $[0,T]\ni t\longmapsto \xi_t\in\Pc_2(\Cc_m\times A)$. 
For a given $\F$-predictable and $A$-valued process $\alpha$ and $\xi\coloneqq (\xi_t)_{t\in[0,T]}\in\mathfrak P$, we define the probability measure $\PP^{\alpha,\xi}$ on $(\Omega,\Fc)$ by
\begin{equation*}
	\frac{\mathrm{d}\PP^{\alpha,\xi}}{\mathrm{d}\PP} \coloneqq  \cE\bigg(\int_0^\cdot b_s(X_{\cdot\wedge s}, \xi_s, \alpha_s)\cdot \mathrm{d}W_s \bigg)_T.
\end{equation*}
We let $\mathfrak A$ be the set of $\F$-predictable, $A$-valued processes.
By Girsanov's theorem, the process $X$ satisfies
\begin{equation*}
	X_t=X_0+\int_0^t\sigma_s(X_{\cdot\wedge s}) b_s\big(X_{\cdot\wedge s},\xi_s,\alpha_s\big)\mathrm{d}s+\int_0^t\sigma_s(X_{\cdot\wedge s})\mathrm{d}W_s^{\alpha,\xi},\; t\in[0,T],\; \PP^{\alpha,\xi}\text{--a.s.} , 
\end{equation*}
where $W^{\alpha,\xi} \coloneqq  W- \int_0^\cdot b_s(X_{\cdot\wedge s}, \xi_s, \alpha_s)\mathrm{d}s$ is an $\R^d$-valued, $\PP^{\alpha,\xi}$--Brownian motion.
Given a measure flow $\xi\in\mathfrak P$, the infinitesimal agent faces the control problem of maximising the reward function
\begin{equation*}
	J^\xi(\alpha) \coloneqq  \EE^{\PP^{\smalltext{\alpha}\smalltext{,}\smalltext{\xi}}}\bigg[\int_0^Tf_s\big(X_{\cdot\wedge s},\xi_s,\alpha_s\big)\mathrm{d}s+g(X,\xi^1_T)\bigg],\; \alpha\in\mathfrak A,
\end{equation*}
where $\xi^1_T\in \cP_2(\Cc_m)$ is the first marginal of $\xi_T$. In other words, the value of the problem is, for given $\xi\in\mathfrak P$ 
\begin{equation*}
	V^\xi \coloneqq  \sup_{\alpha \in \mathfrak A}J^\xi(\alpha).
\end{equation*}
We can now give the definition of a mean-field equilibrium.
\begin{definition}
	A solution of the mean-field game, which we will refer to as a mean-field equilibrium, is defined as a control process $\hat\alpha \in \mathfrak A$ such that there is $\xi\in \mathfrak P$ satisfying $V^\xi = J^\xi(\hat\alpha)$ and
	\begin{equation}
	\label{eq:mfe.def}
		\PP^{\hat\alpha,\xi}\circ (X_{\cdot\wedge t},\hat\alpha_t)^{-1} = \xi_t,\; \text{\rm for Lebesgue--almost every $t\in[0,T]$}.
	\end{equation}
\end{definition}

We now proceed with another characterisation result.
In this case, we give both a sufficient and a necessary equilibrium condition.
Below and henceforth, we write\footnote{When a mean field equilibrium exists, there is $(\hat\alpha,\xi)$ satisfying \eqref{eq:mfe.def}.
The probability measure $\P^{\hat\alpha,\xi}$ is then (simply) denoted $\P^{\hat\alpha}$.}
\begin{equation*}
	\cL_{\hat\alpha}(\Gamma) \text{ for the law of the random variable } \Gamma \text{ under } \P^{\hat\alpha} \text{ and $\cL(\Gamma)$ the law of $\Gamma$ under $\P$.}
\end{equation*}

\begin{proposition}[Characterisation of mean-field equilibrium]
\label{prop:char.mfe}
	Let {\rm \Cref{ass.Lambda.charac}} be satisfied. 
An admissible control $\hat\alpha \in \mathfrak{A}$ is a mean-field equilibrium if and only if it satisfies $\hat\alpha_t \coloneqq  \hat a\big(t,X_{\cdot\wedge t}, \cL_{\hat\alpha}(X_{\cdot\wedge t},\hat\alpha_t),Z_t\big),\; \d t\otimes\d \P\text{\rm--a.e.}$ for some $\hat a\in\A$, where $(Y, Z)$ solves the generalised {\rm McKean--Vlasov BSDE}
	\begin{equation}
	\label{eq:bsde.char.mfg}
		\begin{cases}
			\displaystyle Y_t = g\big(X, \cL_{\hat\alpha}(X)\big) + \int_t^Tf_s\big(X_{\cdot\wedge s}, \cL_{\hat\alpha}(X_{\cdot\wedge s}, \hat\alpha_s),  \hat\alpha_s\big)\mathrm{d}s - \int_t^TZ_s\cdot \mathrm{d}W^{\hat\alpha}_s,\; t\in[0,T],\; \PP^{\hat\alpha}\text{\rm--a.s.},\\[0.8em]
			\displaystyle\hat\alpha_t = \hat a\big(t,X_{\cdot\wedge t}, \cL_{\hat\alpha}(X_{\cdot\wedge t},\hat\alpha_t),Z_t\big), \; \frac{\mathrm{d}\PP^{\hat\alpha}}{\mathrm{d}\PP} \coloneqq  \cE\bigg(\int_0^Tb_s\big(X_{\cdot\wedge s},\cL_{\hat\alpha}(X_{\cdot\wedge s}, \hat\alpha_s),\hat\alpha_s \big)\cdot \mathrm{d}W_s\bigg),
		\end{cases}
	\end{equation}
	where $\cL_{\hat\alpha}(X) \coloneqq  \PP^{\hat\alpha}\circ X^{-1}$ is the law of $X$ under the measure $\PP^{\hat\alpha}$, $W^{\hat\alpha}\coloneqq  W - \int_0^\cdot b_s\big(X_{\cdot\wedge s},\cL_{\hat\alpha}(X_{\cdot\wedge s}, \hat\alpha_s),\hat\alpha_s \big)\mathrm{d}s $ and
	\begin{equation}
	\label{eq:bsde.char.mfg.integrability}
		\E^{\P^{\hat\alpha}}\bigg[\sup_{t \in [0,T]}|Y_t|^2 + \int_0^T\|Z_t\|^2\mathrm{d}t \bigg]<\infty.
	\end{equation}
	Moreover, we have that $Y_0=V^{\Lc_{\smalltext{\hat\alpha}}(X,\hat\alpha)}$ is the associated value function.
\end{proposition}
\rm\Cref{prop:char.mfe} asserts that solving the mean-field game is equivalent to solving \rm\Cref{eq:bsde.char.mfg}.
We call it a generalised McKean--Vlasov equation because the drift depends on the law of the unknown $Z$, but in addition the driving Brownian motion and the underlying probability measure under which the law is given are unknown.
This equation seems to not have been investigated in the literature so far. Indeed, observe that it is related to, but is not the one studied by \citeauthor*{carmona2015probabilistic} \cite{carmona2015probabilistic}.
In that paper, the laws $\cL_{\hat\alpha}(X)$ and $\cL_{\hat\alpha}(X_{\cdot\wedge t},\hat\alpha_t)$ are replaced by arbitrary probability distributions and then fixed-points are constructed based on the solutions of the resulting BSDEs, while \Cref{eq:bsde.char.mfg} incorporates already the fixed-point itself. Although the main focus of this paper is not on well-posedness, we do devote \rm\Cref{sec:gen.MckV.BSDE} to the analysis of \rm\Cref{eq:bsde.char.mfg} under extra assumptions.

\subsection{Main results}

Let us now present the main contributions of this article. 
We will state general mean-field game limit results and give existence and uniqueness statement for mean-field games in the weak formulation.
\subsubsection{Convergence of Nash equilibria to mean-field equilibria}

The next result gives a quantitative estimate of the convergence of the Nash equilibria of the finite population game to a mean-field game equilibrium as the number of players grows to infinity.
We consider the following conditions.

\begin{assumption}
\label{ass.Lambda.conv}

$(i)$ {\rm\Cref{assump:sigma}} and {\rm\Cref{ass.Lambda.charac}} hold$;$

\medskip
	$(ii)$ for every $\hat a\in \Oc^N$, there exist some Borel-measurable maps $\Lambda:[0,T]\times\Cc_m\times \Pc_2(\Cc_m)\times \R^d\times \R \longrightarrow A$ and $\aleph^N\coloneqq (\aleph^{i,N})_{i\in\{1,\dots,N\}}:\Cc_m^N\times(\R^{d})^{N\times N}\longrightarrow \R^N$ satisfying for any $i\in\{1,\dots,N\}$
	\[
	\hat a^i(t,\x,z)=\Lambda_t\big(\x^i,L^N(\x),z^{i,i},\aleph^{i,N}_t(\x,z^{i,:}) \big),\; (t,\x,z)\in[0,T]\times\Cc_m^N\times(\R^{d})^{N\times N},
	\]	
	and such that, letting $\xi^1$ be the marginal on $\Cc_m$ of an arbitrary $\xi\in\Pc_2(\Cc_m\times A)$

	\[
	\Lambda_t(\x,\xi^1,z,0)\in\A(t,\x,\xi,z),\; \forall(t,\x,\xi,z)\in[0,T]\times\Cc_m\times\Pc_2(\Cc_m\times A)\times\R^d;
	\]

\smallskip
	$(iii)$ the function $\Lambda:[0,T]\times\Cc_m\times \Pc_2(\Cc_m)\times \R^d\times \R \longrightarrow A$ from $(ii)$ is additionally assumed to be Lipschitz-continuous with Lipschitz constant $\ell_\Lambda>0$, and the map $\aleph^N\coloneqq (\aleph^{i,N})_{i\in\{1,\dots,N\}}:\Cc_m^N\times\R^{d\times N}\longrightarrow \R^N$ satisfies that there is a sequence $(R_N)_{N\in\N^\smalltext{\star}}$ valued in $\R_+$, with $(NR_N^2)_{N\in\N^\smalltext{\star}}$ non-increasing,
	\[
	\lim_{N\to+\infty}NR_N^2=0,\; N^2R_N^2\underset{N\to+\infty}{=}\Oc(1),
	\] 
	and
	\begin{align}
	\label{eq:bound.aleph}
		\big|\aleph_t^{i,N}(\x,z)\big|\leq R_N\bigg(1+ \|\x^i\|_\infty+\sum_{j\in\{1,\dots,N\}}\|z^{i,j}\|\bigg),\; (t,\x,z)\in[0,T]\times\Cc_m^N\times(\R^{d})^{N\times N},\; i\in\{1,\dots,N\};
	\end{align}
	
 	$(iv)$ for any $(a,\xi)\in A\times \cP_2(\cC_m\times A)$, the maps $[0,T]\times\cC_m\ni (t,\bx)\longmapsto f_t(\bx,\xi,a)$ and $[0,T]\times\cC_m\ni (t,\bx)\longmapsto b_t(\bx,\xi,a)$ are $\F$-optional, the functions $b$, $f$, $g$ are Borel-measurable in all their arguments, they are Lipschitz-continuous uniformly in $t$, $b$ is dissipative, and $f$ is also locally Lipschitz-continuous in $a$.
 	That is, there are positive constants $\ell_b$, $\ell_f$, $\ell_g$, $K_b$ and a linearly growing function $\varphi:A^2\longrightarrow \R_+$ such that for any $(t,\bx,\bx^\prime,a, a^\prime,\xi, \xi^\prime,\mu,\mu^\prime)\in [0,T]\times\cC_m^2\times A^2\times \big(\cP_2(\cC_m\times A)\big)^2\times (\cP_2(\cC_m))^2$ 
	\begin{align}
	\label{ass.b.gen}
			\big|b_t(\bx,\xi, a) - b_t(\bx,\xi^\prime, a^\prime)\big|&\le \ell_b \big(  {\mathcal W}_2(\xi, \xi^\prime) + \bar d(a,a^\prime)\big),
	\end{align}
	
	\vspace{-1.8em}
	\begin{align*}
		(\bx - \bx^\prime)\cdot\big(\|b_t(\bx,\xi,a) - b_t(\bx^\prime,\xi,a)\big)\le -K_b\|\bx - \bx^\prime\|_\infty^2
	\end{align*}

\vspace{-1.8em}
	\begin{align*}
		\big|f_t(\bx,\xi, a) - f_t(\bx,\xi^\prime, a^\prime)\big|&\le \ell_f \big( \|\bx - \bx^\prime\|_\infty + {\mathcal W}_2(\xi, \xi^\prime) + \varphi(a, a^\prime)\bar d(a,a^\prime)\big),\; \big\| g(\bx, \mu) - g(\bx, \mu)\big\| \le \ell_g\big( \|\bx - \bx^\prime\| + \cW_2(\mu, \mu)\big);
	\end{align*}

	$(v)$ $\A$ is reduced to one element and the mean-field game admits a unique mean-field equilibrium $\hat\alpha\in\mathfrak A;$

	\medskip
	$(vi)$ for every $N\in \N^\star$, for every probability measure $\Pi$ on $(\Omega,\F)$ and every independent $(\F_N,\Pi)$--Brownian motions $(B^1,\dots, B^N)$, the following forward--backward {\rm SDE} admits at least one solution $(\overline X^{i,N}, \overline Y^{i,N}, \overline Z^{i,j,N})_{(i,j)\in\{1,\dots,N\}^\smalltext{2}} \in (\S^2(\R^m,\F_N))^N\times (\S^2(\R,\F_N))^N\times  (\H^2(\R^d,\F_N))^{N^2}$
	 	\begin{equation}\label{eq:FBSDE}
		\begin{cases}
			\displaystyle \overline X^{i,N}_t = \overline X^i_0 + \int_0^t b_s\big(\overline X^{i,N}_{\cdot\wedge s}, L^N\big(\overline\X^N_{\cdot\wedge s}, \overline \alpha^{N}_s\big), \overline \alpha^{i,N}_s\big)\diff s + \int_0^t\sigma_s\big(\overline X^{i,N}_{\cdot\wedge s}\big)\diff B^{i}_s,\; t\in[0,T],\; \Pi\text{\rm --a.s.},\\[0.8em]
		\displaystyle	\overline Y^{i,N}_t = g\big(\overline X^{i,N}, L^N(\overline \X^N)\big) + \int_t^T f_s\big (\overline X^{i,N}_{\cdot\wedge s}, L^N\big (\overline\X^N_{\cdot\wedge s}, \overline \alpha^{N}_s\big), \overline \alpha^{i,N}_s\big)\diff s - \int_t^T \sum_{j=1}^N\overline Z^{i,j,N}_s\cdot \diff B^{j}_s,\; t\in[0,T],\; \Pi\text{\rm --a.s.},\\[0.9em]
		\displaystyle	\overline \alpha^{i,N}_t\coloneqq  \Lambda_t\big(\overline X^{i,N}_{\cdot\wedge t}, L^N\big (\overline\X^{N}_{\cdot\wedge t}),\overline Z^{i,i,N}_t,0\big),\; \overline\alpha^N\coloneqq (\overline\alpha^{i,N})_{i\in \{1,\dots,N\}}.
		\end{cases}
	\end{equation}
\end{assumption}
Before stating the main convergence result, let us shortly elaborate on {\rm \Cref{ass.Lambda.conv}}.
\paragraph{Comments on the assumptions}
\label{rem:Lambda.depends.xi}
	{\rm\Cref{assump:sigma}} is well-known to be satisfied when $\sigma$ is Lipschitz-continuous, see \emph{e.g.} {\rm\citeauthor*{protter2005stochastic} \cite{protter2005stochastic}}, or under even weaker conditions when $\sigma$ is state dependent or $m=1$, see \emph{e.g.} {\rm\citeauthor*{krylov1980controlled} \cite{krylov1980controlled}} and the references therein.
	{\rm\Cref{ass.Lambda.charac}} corresponds to standard growth conditions assumed throughout the literature and allowing to make the control problems finite-valued.

	\medskip

	Regarding $(ii)$, first note that if we do not have interaction through the controls, then with the regularity assumptions made on $b$ and $f$, standard measurable selection arguments allow to construct a Borel-measurable function $\Lambda$ such that, putting $\hat a^i = \Lambda(t,\x^i, L^\np(\x), z^{i,i})$, we have $\hat a \in \Oc^\np$.
	In particular, $\Lambda$ depends neither on $\np$ nor on $\aleph^N$, which makes \eqref{eq:bound.aleph} in {\rm \Cref{ass.Lambda.conv}.$(ii)$--$(iii)$} trivially satisfied in this case, with $\aleph^{i,N} = 0$. 
	Moreover, well-known convexity properties of $h$ $($at least when $A$ is a finite-dimensional Euclidean space$)$ imply that $\Lambda$ is Lipschitz-continuous, see \emph{e.g.} {\rm \citeauthor*{carmona2021probabilistic} \cite{carmona2021probabilistic}}.
	Thus, the Lipschitz-continuity condition on $\Lambda$ in {\rm\Cref{ass.Lambda.conv}.$(iii)$} is always satisfied when we do not have interaction through the control, and under additional convexity assumptions.

\medskip
	The case of interaction through the control is a little more subtle.
In this case, by measurable selection arguments, we can still construct a Borel-measurable function $\Lambda$ and a function $\aleph^N$ such that, defining $\hat a^i\coloneqq \Lambda(t, \x^i, L^\np(\x), z, \aleph^{i,\np}(a^{-i}))$, $i\in\{1,\dots,N\}$, we have $\hat a \in \Oc^\np$ where $\aleph^{i,\np}$ should be understood as $`$the part of the control of player $i$ due to other players' actions{\rm'}.
Intuitively, one expects that $\aleph^N(a^{-i})$ tends to zero as $\np$ goes to infinity, and that $\Lambda(t, \x^i, \xi, z,0)$ maximises $h_t(\x, \xi, z, a)$ over $a$.
This is exactly what is encoded in {\rm\Cref{ass.Lambda.conv}.$(iii)$}. A similar property is fully worked out in {\rm\citeauthor*{lauriere2022convergence} \cite[Lemma 22]{lauriere2022convergence}} under the assumption that the function $f$ can be decomposed as $f(t, x, \xi,a) = f_1(t, x, \xi^1,a) + f_2(t, x, \xi)$ where $\xi^1$ is the first marginal of $\xi$, with a similar decomposition for $b$. 
Indeed, assume that $A\subset \R^k$ for some $k\in\N^\star$, that $f$ and $b$ can be decomposed as
\[
\varphi_t(\x,\xi,a)=\varphi^1_t(\x,\xi^1,a)+\varphi^2_t(\x,\xi),\; (t,\x,\xi,a)\in[0,T]\times\Cc_m\times\Pc_2(\Cc_m\times A)\times A,\; \varphi\in\{b,f\},
\]
where for $\xi\in\Pc_2(\Cc_m\times A)$, $\xi^1$ is the marginal of $\xi$ on $\Cc_m$, and that first-order conditions characterise the argmax in the definition of fixed-points of $\Hc^N$ and maximisers of $h$. Then assuming enough regularity on $f$ and $b$ $($refer to {\rm\citeauthor*{cardaliaguet2019master} \cite{cardaliaguet2019master}} for details on differentiability on the space of measures$)$ any $\hat a\in \Oc^N$ will satisfy for any $i\in\{1,\dots,N\}$ and any $(t,\x,z)\in[0,T]\times\Cc_{m\times N}\times (\R^{d\times N})^N$
\begin{align*}
	0&=\partial_af^1_t\big(\x^i,L^N(\x),\hat a^i(t,\x,z))+\partial_ab^1_t\big(\x^i,L^N(\x),\hat a^i(t,\x,z))\cdot z^{i,i} +\frac1N\partial_\xi f^2_t\big(\x^i,L^N(\x,\hat a(t,\x,z))\big)\\
	&\quad+\frac1N\partial_\xi b^2_t\big(\x^i,L^N(\x,\hat a(t,\x,z))\big)\cdot z^{i,i} +\frac1N\sum_{j\in\{1,\dots,N\}\setminus\{i\}}\partial_\xi b^2_t\big(\x^j,L^N(\x,\hat a(t,\x,z))\big)\cdot z^{i,j},
\end{align*}
from which we deduce here that 
\begin{align*}
	\aleph^{i,N}_t\big(\x,(z^{i,j})_{j\in\{1,\dots,N\}}\big) &=\frac1N\partial_\xi f^2_t\big(\x^i,L^N(\x,\hat a(t,\x,z))\big)+\frac1N\partial_\xi b^2_t\big(\x^i,L^N(\x,\hat a(t,\x,z))\big)\cdot z^{i,i}\\
	&\quad+\frac1N\sum_{j\in\{1,\dots,N\}\setminus\{i\}}\partial_\xi b^2_t\big(\x^j,L^N(\x,\hat a(t,\x,z))\big)\cdot z^{i,j},
\end{align*}
and then that when $\aleph^N$ is $0$, the first-order conditions become exactly the same as the ones characterising elements of $\A$. Furthermore, whenever the derivatives of $b$ and $f$ appearing above are bounded, the rest of {\rm\Cref{ass.Lambda.conv}}.$(iii)$ holds with $R_N=1/N$.

\medskip
	
	{\rm\Cref{ass.Lambda.conv}}.$(iv)$ corresponds to standard regularity conditions. Observe that due to the weak formulation of the game, these regularity conditions are much weaker than the ones in {\rm\citeauthor*{lauriere2022convergence} \cite{lauriere2022convergence}}, and are in line with the conditions of {\rm \citeauthor*{carmona2015probabilistic} \cite{carmona2015probabilistic}} on existence.

	\medskip
	{The dissipativity assumption 	{ \rm\Cref{ass.Lambda.conv}}.$(iv)$ is key to our analysis.
	In can be seen as restrictive given that the drift cannot be independent of the state, but on the other hand we are able to obtain quantitative convergence without assuming monotonicity of the coefficients.
	The dissipativity condition 	{ \rm\Cref{ass.Lambda.conv}}.$(iv)$ is classical in the study of contractivity for SDEs and convergence to the stationary distribution. 
	}

\medskip
	{ \rm\Cref{ass.Lambda.conv}}.$(v)$ on uniqueness of the mean-field game equilibrium is needed to guarantee existence of a unique solution of the characterising generalised McKean--Vlasov equation \eqref{eq:bsde.char.mfg}. This will be needed to prove a propagation of chaos result required in our argument.
	In {\rm\Cref{thm:main.existence.MFG}} below we present a case in which uniqueness can be proved under suitable regularity and boundedness conditions.
	We also refer to {\rm\citeauthor*{carmona2015probabilistic} \cite{carmona2015probabilistic}} for other assumptions guaranteeing existence and uniqueness.
	Propagation of chaos will also require {\rm\Cref{ass.Lambda.conv}}.$(vi)$.
	Notice that {\rm\Cref{eq:FBSDE}} is a $($classical$)$ forward--backward {\rm SDE} with Lipschitz-continuous coefficients.
	Such equations have been extensively studied in the literature, and various set of assumptions are known to guarantee their well-posedness in arbitrary large time.
	{For instance, when our coefficients are state-dependent, if we additionally assume $f$ and $g$ to be bounded in $x$, then {\rm \citeauthor*{delarue2002existence} \cite[Theorem 2.6]{delarue2002existence}} guarantees {\rm\Cref{ass.Lambda.conv}}.$(vi)$.
	Path-dependent {\rm FBSDEs }are studied in the recent paper by {\rm\citeauthor*{hu2023path} \cite{hu2023path}}.} 
	For other references on the existence of {\rm FBSDEs}, we further refer the reader for instance to {\rm \citeauthor*{ma1994solving} \cite{ma1994solving}}, {\rm\citeauthor*{ma2015well} \cite{ma2015well}}, {\rm\citeauthor*{peng1999fully} \cite{peng1999fully}}, {\rm\citeauthor*{yong2010forward} \cite{yong2010forward}}, or {\rm\citeauthor*{zhang2006wellposedness} \cite{zhang2006wellposedness}}
	{\color{black}It is likely that the dissipativity condition on the drift already allows to guarantee {\rm\Cref{ass.Lambda.conv}}.$(vi)$. We make this assumption as we could not find a directly citeable reference and the paper is already rather long}.


\medskip
Throughout the paper, for any positive integers $n$, $\np$ and any $q>1$, we will denote
	\begin{equation}
	\label{eq:def_r-nmqp}
		r_{\text{\fontsize{5}{5}\selectfont$\np$},n,q}\coloneqq \begin{cases}
			\np^{-1/2} + \np^{-(q-2)/q},\; \text{\rm if } n<4,\;  \text{\rm and } q\neq 4,\\[0.3em]
			\np^{-1/2}\log(1+\np) + \np^{-(q-2)/q},\; \text{\rm if } n = 4,\; \text{\rm and } q \neq 4,\\[0.3em]
			\np^{-2/n} + \np^{-(q-2)/q},\; \text{\rm if } n>4,\; \text{\rm and } q\neq n/(n-2).
		\end{cases}
	\end{equation}
	These quantities are related to the rate of convergence we can obtain. Our first main result is stated under the simplifying assumption that there is no terminal reward $g$ in the game, and we explain afterwards how this can be extended under appropriate structural conditions.
	
\begin{theorem}
\label{thm:main.limit}
	Let {\rm\Cref{ass.Lambda.conv}} hold.
	Let $(\hat\alpha^{\np})_{N\in\N^\smalltext{\star}}$ be a sequence of Nash equilibria for the $\np$-player game.
	There is a constant $\delta$ depending on $\ell_b, \ell_g,\ell_f,\ell_\Lambda$ and $T$ such that if $K_b\ge \delta$, then for each $i\in\{1,\dots,\np\}$, the sequence $(V^{i,\np})_{\np\in\N^\smalltext{\star}}$ converges to the value function $V^{\hat\xi}$ of the mean-field game, where $\hat\xi\in\mathfrak B$ is such that $\PP^{\hat\alpha,\hat\xi}\circ (X_{\cdot\wedge t},\hat\alpha_t)^{-1} = \hat\xi_t,$ for Lebesgue--almost every $t\in[0,T]$.
	More precisely, we have
	\begin{equation}
	\label{eq:conv empirical}
		\big|V^{i,\np} - V^{\hat\xi}\big|^2 
		\le C\bigg(\frac1N + NR_N^2 + \gamma^N\bigg),\; \forall N\in \N^\star,
	\end{equation}
	and the sequence of Nash equilibria $(\hat\alpha^{i,N})_{N\in \N^\smalltext{\star}}$ converges to the mean-field equilibrium $\hat\alpha^i$ in the sense that
	\begin{equation}
		\label{eq:conv.law.hatalpha}
			\int_0^T\cW_2^2\big(\P^{\hat\alpha^{\text{\fontsize{4}{4}\selectfont $N$}},N}\circ(\hat\alpha^{i,N}_s)^{-1}, \cL_{\hat\alpha}(\hat\alpha_s^i) \big)\diff s \le  C\bigg(\frac1N + NR_N^2 + \gamma^N\bigg),\; \forall N \in \N^\star,
	\end{equation}
	where, letting $\hat{\bm{\alpha}}^N$ be the vector of $\P^{\hat\alpha}$--i.i.d. processes $(\hat\alpha^{1}, \dots, \hat\alpha^{N})$, we defined
	\begin{equation}
	\label{eq:def.gamma}
		\gamma_N \coloneqq \sup_{t\in [0,T]}\E^{\P^{\smalltext{\hat\alpha}}}\Big[\Wc_2^2\big(L^N(\X^N_{\cdot\wedge t}), \Lc_{\hat \alpha}(X_{\cdot\wedge t})\big) + \Wc_2^2\big(L^N(\hat{\bm{\alpha}}^N_{t}), \Lc_{\hat \alpha}(\hat\alpha_{t})\big) \Big].
	\end{equation}
	If in addition the functions $b$ and $f$ are state-depended in the law, \emph{i.e.} $b$ and $f$ are defined on $[0,T]\times \Cc_m\times \cP_2(\R^m\times A)\times A$, and $A\subset \R^k$ for some $k\in \N^\star$, then the rate reduces to 
	\begin{equation}
	\label{eq:conv.empirical.rate}
		\big|V^{i,\np} - V^{\hat\xi}\big|^2 + \int_0^T\cW_2^2\big(\P^{\hat\alpha^{\text{\fontsize{4}{4}\selectfont $N$}},N}\circ(\hat\alpha^{i,N}_s)^{-1}, \cL_{\hat\alpha}(\hat\alpha_s)\big)\diff s
			\le C\bigg(\frac1N +  NR_N^2 +  r_{\text{\fontsize{5}{5}\selectfont$\np$},m, q} + r_{\text{\fontsize{5}{5}\selectfont$\np$},k, q} \bigg),\; \forall  (N,q)\in \N^\star\times(2,+\infty).
	\end{equation}
\end{theorem}

An important remark is in order.
\begin{remark}
\label{rem.rate}
	Let us now comment on the non-asymptotic convergence rates in {\rm\Cref{eq:conv empirical}}.
	The reader will notice that these rates solely depend on the speed of convergence of the empirical measures $L^N(\XX^N)$ and $L^N(\hat\balpha^N_t)$ to the laws of $X$ and $\hat\alpha$ in the Wasserstein distance and on $R_N$, while the constant $C$ depends on $T$ and all the bounds and Lipschitz constant introduced in the assumptions.
	In view of the discussion in {\rm Section \ref{rem:Lambda.depends.xi}}, 
	the rate $R_N$ is zero when we do not consider interaction through the control and is just $1/N$ under further structural assumptions on $b$ and $f$.
	Regarding $\gamma_N$, it is well-known that convergence rates for empirical measures of random variables on Polish spaces are difficult to obtain.
	The law of such random variables are usually required to satisfy strong integrability conditions or some functional inequalities, see {\rm\citeauthor*{bolley2005weighted} \cite{bolley2005weighted}}, {\rm\citeauthor*{boissard2011simple} \cite{boissard2011simple}, \citeauthor*{boissard2014mean} \cite{boissard2014mean}}, {\rm\citeauthor*{fournier2015rate} \cite{fournier2015rate}}, or {\rm\citeauthor*{weed2019sharp} \cite{weed2019sharp}}.
	
	\medskip
	In the existing literature on the mean-field limit $($discussed in the introduction$)$, only the papers by {\rm\citeauthor*{lauriere2022convergence} \cite{lauriere2022convergence}}, {\rm\citeauthor*{jackson2023quantitative} \cite{jackson2023quantitative}} and {\rm\citeauthor*{delarue2020master} \cite{delarue2020master}} and the monograph by {\rm\citeauthor*{cardaliaguet2019master} \cite{cardaliaguet2019master}} obtain convergence rates. 
	In {\rm\cite{delarue2020master}}, concentration of measures results are obtained for the empirical law of the state process at equilibrium, and  {\rm\cite{cardaliaguet2019master}} provides a convergence rate for the convergence of the value functions.
	Both papers rely on existence and uniqueness of a solution of the master equation with bounded first and second derivatives.
	Closer to our work is {\rm\cite{lauriere2022convergence}} where a convergence rate for the Nash equilibrium $($in the Markovian setting$)$ is also given using the theory of coupled forward--backward {\rm SDEs} {\color{black}and a dissipativity condition similar to ours}.
	As the reader will observe, the proof of {\rm\Cref{thm:main.limit}} does not make use of properties of the master equation, and accommodates a state-dependent volatility, as well as a completely non-Markovian framework.
	Furthermore, compared to {\rm\cite{lauriere2022convergence}}, {\rm \Cref{thm:main.limit}} and its corollaries offer a substantial gain of regularity on the coefficients of the game. In fact, we only assume the functions $b,f$ and $g$ to be Lipschitz-continuous and such that the optimiser $\Lambda$ is again Lipschitz-continuous.
	Granted, this weakening of the regularity requirements is also due to the fact that we obtain the convergence of the laws of the Nash equilibrium.
	In fact, we have pointwise convergence of the equilibria only on a new probability space.

\medskip
	{\color{black}We should also point out that the constant $\delta$ in the statement of the theorem can be made explicit by following our successive arguments. We decided not to do so because it is quite involved and no intuition could be gained from the explicit form of the constant.}
\end{remark}

\subsubsection{Existence of mean-field equilibria}
We now complement the convergence theorems from the previous section with an existence result.
Here are our assumptions for the existence of mean-field game equilibria.

\begin{assumption}
\label{assum:main.conv}
$(i)$ There exists a Lipschitz-continuous map $\Lambda:[0,T]\times \Cc_m\times \cP_2(\Cc_m)\times \RR^d \longrightarrow A$ such that for any $(t, \x, \xi, z) \in [0,T]\times \Cc_m\times \cP_2(\Cc_m\times A)\times \RR^d$ 
\begin{equation*}
	\Lambda_t(\x, \xi^1,z) \in \argmax_{a\in A} \big\{h_t( \x, \xi, z,a)\big\}
\end{equation*}
where $\xi^1$ is the first marginal of $\xi;$
and the function
 $\sigma_t:\cC_m\longrightarrow \R^{\xdim\times \bmdim}$ is $\ell_\sigma$--Lipschitz-continuous, uniformly in $t\in[0,T]$, for some $\ell_\sigma\in(0,\infty)^2$;
\medskip

$(ii)$ The function $b$ is $\ell_b$-Lipschitz in all variables and the functions $b$, $f$, and $g$ satisfy {\rm\Cref{ass.Lambda.conv}}.$(iv)$ {\color{black} with $K_b\ge \ell_b^2 + (4C_{\rm BDG} +1)\ell_g^2 + \ell_b\ell_\Lambda $, where $C_{\rm BDG}$ is the constant appearing in Burkholder--Davis--Gundy inequality with exponent $1$}.
\end{assumption}
\begin{assumption}
\label{ass.g.smooth}
  \begin{itemize}
\item[$(i)$] the function $g:\R^m\times \cP_2(\R^m)\longrightarrow \R$ is such that for every $\mu \in \cP_2(\R^m)$, $\R^m\ni x\longmapsto g(x,\mu)$ is twice continuously differentiable$;$ for every $x \in \R^m$, the map $\cP_2(\R^m)\ni \mu\longmapsto g(x,\mu)$ is continuously differentiable and for $(x,\mu)\in\R^m\times\Pc_2(\R^m)$, the map $\Pc_2(\R^m)\ni v\longmapsto \partial_\mu g(x,\mu)(v)$ admits a version such that $\R^m\times\Pc_2(\R^m)\times\Pc_2(\R^m)\ni(x,\mu,v)\longmapsto \partial_\mu g(x,\mu)(v)$ is locally bounded, and $v \longmapsto \partial_\mu g(x,\mu)(v)$ is continuously differentiable with locally bounded derivative$;$ 
	\medskip

	\item[$(ii)$] the functions $\sigma:[0,T]\times\R^m\longrightarrow \R^{m\times d}$ and $\R^m\times \Pc_2(\R^m)\ni (x,\mu)\longmapsto \partial_xg(x,\mu)\sigma_t(x)$ are Lipschitz-continuous$;$
	
	\medskip

	\item[$(iii)$] the function $b$, $g$ and
	\begin{align}
	\label{eq:def.tilde.f.conv}
	\notag
	\widetilde f_t(\x_{\cdot\wedge t}, \xi,a) &\coloneqq  f_t(\x_{\cdot\wedge t}, \xi, a) + \frac12\mathrm{Tr}\big[\partial_{xx}g(\x_t, \xi^1_t)\sigma_t(\x_t)\sigma_t^\top(\x_{ t})\big] + \frac12\int_{\mathbb{R}^\smalltext{\xdim}}\mathrm{Tr}\big[ \partial_a\partial_\mu g(\x_t, \xi^1_t)(a)\sigma_t(a)\sigma_t(a)^\top \big] \xi^1(\diff a)\\
	&\quad +  b_t(\x_{\cdot \wedge t},\xi,a)\cdot \partial_xg(\x_t,\xi^1_t)\sigma_t(\x_{t}), \; (t,\x,\xi,a)\in[0,T]\times\Cc_m\times \Pc_2(\R_m\times A)\times A,
	\end{align}
	satisfy {\rm\Cref{ass.Lambda.conv}}.$(iv)$ {\color{black} with $K_b\ge \ell_b^2 + (4C_{\rm BDG} +1)\ell_g^2 + \ell_b\ell_\Lambda $}.
\end{itemize}
\end{assumption}

\begin{theorem}
\label{thm:main.existence.MFG}
	Let {\rm \Cref{assum:main.conv}} hold. There is $\Psi>0$ such that if $\|g\|_\infty\le \Psi$, then
 the mean-field game admits a mean-field equilibrium $\hat\alpha$.
	If {\rm\Cref{ass.g.smooth}} further holds, then the constant $\Psi$ can be taken arbitrary. Moreover, if for any $(t, \x, \xi, z) \in [0,T]\times \Cc_m\times \cP_2(\Cc_m\times A)\times \RR^d$ the set $\A(t,\x,\xi,z)$ is a singleton, then there is at most one mean-field equilibrium.
\end{theorem}

The method we use to derive existence is very different from the one proposed by \citeauthor*{carmona2015probabilistic} \cite{carmona2015probabilistic} based on---a version of---Kakutani’s fixed-point theorem.
In fact, our arguments are rather based on general characterisations of mean-field games by backward SDEs.
The \emph{caveat} here is that, due to the weak formulation of the control problem, the characterising BSDE \eqref{eq:bsde.char.mfg} is not a standard equation of McKean--Vlasov type, but the driving Brownian motion as well as the underlying probability measure are unknown.
We study well-posedness of this equation in \Cref{sec:gen.MckV.BSDE}. 
Notice that our result has a specific feature: as soon as the terminal reward is $0$, the constant $\Psi$ is $0$ as well, and we get existence of a mean-field equilibrium under general assumptions for a general non-Markovian problem allowing for interactions through the controls. Moreover, uniqueness then only requires that the Hamiltonian of the players has a unique maximiser, and does not involve the standard Lasry--Lions's monotonicity condition (see for instance \cite{lasry2007mean}) generally assumed in the literature. As far as we know, such results are new.
Notice however that {\rm \Cref{thm:main.existence.MFG}} does not cover the case of games with quadratic costs unless stronger regularity and boundedness conditions are satisfied. For such games, our characterising BSDE will have a quadratic generator, and as is well-known from the results on multidimensional quadratic BSDEs, such equations can have infinitely many solutions \citeauthor*{frei2011financial} \cite{frei2011financial}.
Similarly, as observed by \citeauthor*{tchuendom2018uniqueness} \cite{tchuendom2018uniqueness}, the mean-field game can have many solutions.

\subsubsection{Examples}\label{sec:examples}

Let us at this point give two examples to which our mean-field limit result applies.
The first example comes from the classical problem of optimal execution in financial models with price impact.
The second example showcases a game in which the non-Markovian structure considered in this work applies.

\paragraph{A non-Markovian price impact model}
This first example is treated for instance by \citeauthor{carmona2015probabilistic} \cite{carmona2015probabilistic} to which we refer for details.
Let us assume for simplicity that $d=m=1$ and that $A \subseteq \R$ is a closed bounded subset. 
In this game, $\np$ traders invest on the same stock whose price $S$ is subject to (instantaneous) price impact.
{\color{black}Let us assume that inventory of trader $i\in\{1,\dots,N\}$ is given by 
\begin{equation*}
	\d X^i_t = \big(-K_bX_t^i + \alpha^i_t\big)\d t + \sigma\d (W^\alpha_t)^i,
\end{equation*}
where $K_b>0$,} $\sigma>0$, $\alpha^i$ is an $\F^\np$-predictable measurable process taking values in a subset $A$ of $\R$.
Assume that the investors are risk-neutral, face transaction cost $c:\R\longrightarrow \R$ and a terminal liquidation constraint $g$. Then the $i$-th trader's control problem given that other traders played the controls $(\alpha^{-i})\in\Ac^{N-1}$ is
\begin{equation*}
	\inf_{\alpha\in\Ac}\E^{\P^{\smalltext{\alpha}\smalltext{\otimes}_\tinytext{i}\smalltext{\alpha}^{\tinytext{-}\tinytext{i}}}}\bigg[\int_0^T\bigg(\frac{\gamma(X^i_{\cdot\wedge t})}{\np}\bigg(\sum_{j=1}^Nc^\prime(\alpha^j_t)\bigg) - c(\alpha^i_t)-k(t,X^i_{\cdot\wedge t})\bigg)\d t + g(X^i_T)\bigg],
\end{equation*}
for two real-valued function $\gamma$ and $k$.
In this cost function, the first term in the time integral represents the price impact induced by the trading strategies of all the agents, the second term is a trading cost incurred to player $i$, while the third term represents a penalty for holding a large inventory. The cost function of the associated mean-field game is given by
\begin{equation*}
	\E^{\P^{\smalltext{\alpha}\smalltext{,}\smalltext{\nu}}}\bigg[\int_0^T\bigg( \gamma(X_{\cdot\wedge t})\int_{\mathbb{R}}c^\prime(a)\nu(\diff a) - c(\alpha_t) - k(t, X_{\cdot\wedge t})\bigg)\diff t + g(X_T) \bigg],\; \nu \in \cP_2(A).
\end{equation*}

\medskip
When $g$ and $k$ are Markovian functionals, the convergence of this game to the corresponding mean-field game was analysed in \cite{lauriere2022convergence}.
In the present non-Markovian case, the convergence  follows as a consequence of our {\rm \Cref{thm:main.limit}}. 
In particular, the functions $g$ and $k$ are Lipschitz-continuous,
it is also customary to take $c(a)=|a|^2/2$ (\emph{i.e.} quadratic transaction cost) and $c^\prime(a) =a$ (\emph{i.e.} linear price impact).
With these specifications, we have the following result.

\begin{corollary}
\label{cor:price.impact}
	Assume that the functions $\gamma$, $c^\prime,c,$ $k$ and $g$ are Lipschitz-continuous $($with $\ell_g$ denoting the Lipschitz constant of $g)$, the function $g$ is twice continuously differentiable with bounded derivatives, and that the functions $\gamma$ and $g$ are bounded.
	If for each $N$ the finite population game admits a Nash equilibrium $\hat\alpha^\np$ {\color{black}and $K_b$ is large enough,} then	for each $i \in \{i,\dots,\np\}$, we have
	\begin{equation}
	\label{eq:conv.examples}
	 	\big|V^{i,\np} - V^{\hat\xi}\big|^2 + \int_0^T\cW_2^2\big(L^N(\hat{\bm{\alpha}}^N_t), \cL_{\hat\alpha}(\hat\alpha_t)\big)\diff t \le C\bigg(\frac1N +r_{\text{\fontsize{5}{5}\selectfont$\np$},1,q} \bigg).
	 \end{equation} 
\end{corollary}

\paragraph{Large population games with time-delayed state dynamics}

The non-Markovian setting of the present paper lends itself well to the case of games with delayed response in the state process.
Games with delay appear in several applications in finance and engineering.
Notably, mean-field games with delay have been investigated for linear--quadratic MFGs by \citeauthor*{huang2018linear} in  \cite{huang2018linear}, in the context of systemic risk by \citeauthor*{carmona2018systemic} \cite{carmona2018systemic} (who consider a delay on the control) and in the context of labor income investment by \citeauthor*{djehiche2020optimal} \cite{djehiche2020optimal}.
A toy model of game with delay can be formulated by specifying the coefficients as
\begin{equation*}
	\sigma_t(\x)=\sigma, \; b_t(\x, \xi, a) \coloneqq  \bar b_t\bigg(\x_{t-\tau}, \int_{\R^\smalltext{m}} x\d\xi^1_{t-\tau}(x), a \bigg),\; 	f_t(\x, \xi, a) = \bar f_t\bigg(\x_{t-\tau}, \int_{\R^\smalltext{m}} x\d\xi^1_{t-\tau}(x), a \bigg),\;\text{and}\; g(X)\coloneqq  \bar g(X_{T-\tau}),
\end{equation*}
where for any $\xi\in\Pc_2(\Cc_m\times A)$, $\xi^1$ is the marginal of $\xi$ on $\Cc_m$, and for any $t\in[0,T]$, $\xi^1_t$ is the projection of $\xi^1$ on the $t$-value of the underlying path.
Hereby, $A\subseteq \mathbb{R}^m$ is a closed set, $\bar b$ and $\bar f$ mapping $[0,T]\times \mathbb{R}^m\times \mathbb{R}^m\times A$ to $\mathbb{R}^d$ and $\mathbb{R}$ respectively, and $\tau\in [0,T]$, with the convention $X_{-t} = 0$ for every $t>0$.
In this setting, if the functions $\bar b,\bar f$ and $\bar g$ are Lipschitz-continuous and such that the optimising function $\Lambda$ is Lipschitz, then we can derive convergence of the Nash equilibrium of the $N$-player game to the corresponding mean field game, provided that the generalised McKean--Vlasov equation \eqref{eq:McKV.eq.intro} admits a solution.

\medskip

Observe that the game we just discussed is arguably a very simple example of game with delay.
For instance, delays of the form $\int_{-T}^0c(X_{t+u})m(\d u)$ could be considered, with appropriately chosen Borel measure $m$ on $[-T,0]$ and function $c$, we could also incorporate interaction through control.
Notice however that this setting covers only delay in the state, and not in the control.
But, if the controls are of closed-loop form then delays on the control can be recast into delays on the state, compare {\rm \Cref{eq:close.loop}}.

\subsection{Two noteworthy consequences}\label{sec:consequences}
Let us now discuss two interesting byproducts of our method and results.
The first pertains to the link with Hamilton--Jacobi--Bellman (HJB) techniques used in the analytic approach to control and mean-field games.
The second explores the convergence of large population games when players' strategies are restricted to be closed-loop controls.

\subsubsection{PDE interpretations}
In the Markovian case, the results of this paper can be easily recast in terms of partial differential equations (PDE).
In fact, assume that $A\subseteq \mathbb{R}^{\ell}$ for some $\ell \in \mathbb{N}^\star$.
If the functions $b$, $f$, $g$, and $\sigma$ in the stochastic differential game depend on the current position $X_t^i$ of the state process (as opposed to dependence in the history $X^i_{\cdot\wedge t}$ of the state), then in view of \eqref{eq:state.W.alpha} the BSDE system \eqref{eq:bsde.main} characterising the $\np$-player game becomes, for $i\in\{1,\dots,N\}$
\begin{align*}
	\mathrm{d}X^i_t &= \sigma_t(X^i_t)b_t\big(X^i_t, L^N(\XX^N_t,\hat\alpha_t),\hat\alpha^{i,\np}_t\big)\mathrm{d}t + \sigma_t(X^i_t)\cdot\mathrm{d}(W^{\hat\alpha}_t)^i,\\
	\mathrm{d}Y^{i,\np}_t &= - h_t\big(X^i_t, L^\np(\XX^N_t,\hat\alpha_t),  Z^{i,i,\np}_t, \hat\alpha^{i,\np}_t\big)\mathrm{d}t - \sum_{j\in\{1,\dots,N\}\setminus\{i\}}  b_t\big(X^j_t, L^N(\X^N_t,\hat\alpha_t),\hat\alpha^{j,\np}_t\big)\cdot {Z}_t^{i,j,\np}\mathrm{d}t   + \sum_{j=1}^\np Z^{i,j,\np}_t\cdot \mathrm{d}(W_t^{\hat\alpha})^j,\\
	 Y_T^{i,\np}& = g\big(X^i_T, L^\np(\XX^N_T)\big),\; \hat\alpha^{N}_t\in \Oc^N\big(t,\X^N_{\cdot\wedge t},(Z^{i,i,N}_t)_{i\in\{1,\dots,N\}}\big),\; 	\d t\otimes\d \P\text{\rm--a.e.},
\end{align*}
where we recall that by \Cref{assump:sigma}, $\hat\alpha^N$ takes the form $\hat\alpha_\cdot^{i,N} = \Lambda_\cdot\big( X^i_\cdot, L^N(\X^N_{\cdot\wedge t}), Z^{i,i,N}_\cdot, \aleph^{i,N}_\cdot(\X^N_{\cdot }, Z^N_\cdot)\big)$, $i\in\{1,\dots,N\}$.

\medskip
 We assume for simplicity that the map $\aleph^{i,N}$ appearing in \Cref{ass.Lambda.conv} is a constant $C/N$ see \emph{e.g.} \rm\Cref{sec:case.study} for an example (recall again that in the case when there is no interaction through the control we have $\aleph^{N}  =  0$.) Therefore, it follows by standard BSDE theory, see \emph{e.g.} \citeauthor*{el1997backward} \cite[Theorem 4.2]{el1997backward} (or \citeauthor*{carmona2016lectures} \cite[Sections 5.3 and 5.4]{carmona2016lectures}), combined with \Cref{thm:Bellman}, that the value function satisfies $V(\hat\alpha^{-i}) = v^{i,\np}(0, X^1_0,\dots, X^\np_0)$ where $v^{\np} \coloneqq  (v^{i,\np},\dots, v^{\np,\np})$ is the unique solution, in an appropriate sense, of the system of PDEs, for $i\in\{1,\dots,N\}$
\begin{equation}
\label{eq:N.hjb.control}
\left\{ 	\begin{aligned}
		&\partial_{t}v^{i,\np}(t,x) + \sum_{j=1}^N\mathrm{Tr}\bigg[\partial_{jj}v^{i,\np}(t,x)\frac{(\sigma_t\sigma_t^\top)(x^j)}2 \bigg]+  \sum_{j=1}^N\sigma_t(x^j)b_t\big(x^j, L^\np(x,\Gamma(v)(t,x)),\Gamma^j(v)(t,x)\big)\cdot\partial_{j}v^{j,\np}(t,x) \\
		& + f_t\big(x^i, L^\np(x,\Gamma(v)(t,x)),\Gamma^i(v)(t,x)\big)  = 0,\; (t,x)\in[0,T)\times\R^{m\times N},\\
		& v^{i,\np}(T,x) = g\big(x^i, L^\np(x)\big),\; x\in\R^{m\times N},\; (t,x)\in[0,T]\times\R^{m\times N},
	\end{aligned}\right.
\end{equation}
where we defined the operator $\Gamma$, for any smooth map $\varphi:[0,T]\times\R^{m\times N}\longrightarrow \R^N$ as
\[
\Gamma(\varphi)(t,x) \coloneqq  \big(\Lambda_t(x^j, L^N(x), \partial_{j}\varphi^j(t,x)),C/N\big)_{j\in\{1,\dots,N\}}.
\]
This equation is nothing but a system of $\np$ HJB equations associated with the stochastic differential game, see \emph{e.g.} \cite{delarue2020master,cardaliaguet2019master}.
In these works, the authors based their argument for the convergence of the $N$-player game to the mean-field game on the convergence of the solution $v^{i,\np}$---when it is \emph{smooth enough}---to the solution $v$ of the so-called master equation given by 
\begin{equation}
\label{eq:master.control}
\left\{ 	\begin{aligned}
		&\partial_tv(t, x, \mu) + \frac12\mathrm{Tr}\big[\partial_{xx}v(t, x, \mu)(\sigma_t\sigma_t^\top)(x) \big] + \sigma_t(x) b_t\big(x, \xi, \overline\Gamma_t(\varphi)(x,\mu)\big)\cdot\partial_xv(t, x, \mu) + f_t\big(x, \xi, \overline\Gamma_t(\varphi)(x,\mu)\big) \\
		& + \int_{\mathbb{R}^m}\bigg( \sigma_t(y)b_t\big(y,\xi,\overline\Gamma_t(v)(y,\mu) \big)\cdot \partial_\mu v(t, x,\mu)(y)+ \mathrm{Tr}\bigg[\partial_{y\mu} v(t,x,\mu)(y)\frac{(\sigma_t\sigma_t^\top)(x)}2\bigg]\bigg)\mathrm{d}\mu(y)=0,\; (t,x,\mu)\in[0,T)\times\R^m\times\Pc_2(\R^m),\\
	&	v(T, x,\mu) = g(x, \mu),\; (x,\mu)\in\R^m\times\Pc_2(\R^m),\;\text{and where}\; \xi \coloneqq  \cL(\chi, \bar\Gamma_t(\varphi)(\chi,\mu)) \text{ with $\cL(\chi)=\mu$},
	\end{aligned}
	\right.
\end{equation}
where we now defined the operator $\overline \Gamma$ acting on smooth functions $\varphi:[0,T]\times\R^m\times \Pc_2(\R^m) \longrightarrow \R$
\[
\overline\Gamma_t(\varphi)(x,\mu) \coloneqq  \Lambda_t\big(x, \mu, \partial_x\varphi(t, x, \mu),0\big),\; (t,x,\mu)\in[0,T]\times\R^m\times\Pc_2(\R^m).
\]
In fact, it holds that $v(0,x,\delta_x) = V$ where $V$ is the value of the mean-field game, when the state $X$ starts at time $0$ from $x\in\R^m$. When the master equation admits a classical solution, it follows that $Y_t = v(t,X_t, \cL(X_t))$,  where $Y_t$ is the unique solution of the McKean--Vlasov equation 
\begin{equation}
\label{eq:McKV.eq.intro}
	\begin{cases}
		\displaystyle\mathrm{d}X_t =\sigma_t(X_t) b_t\big(X_t, \cL_{\hat\alpha}(X_t,\hat\alpha_t), \hat\alpha\big)\mathrm{d}t + \sigma_t(X_t)\cdot\mathrm{d}W_t^{\hat\alpha},\\[0.5em]
		\displaystyle\mathrm{d}Y_t = -f_t\big(X_t,\cL_{\hat\alpha}(X_t, \hat\alpha_t),\hat\alpha_t\big)  + Z_t\cdot\mathrm{d}W_t^{\hat\alpha},\\[0.5em]
		\displaystyle Y_T = g\big(X_T, \cL(X_T)\big),\; \hat\alpha_t = \Lambda_t(X_t, \cL_{\hat\alpha}(X_t),Z_t,0),
	\end{cases}
\end{equation}
recall \Cref{prop:char.mfe}.
In this case, \Cref{thm:main.limit} provides probabilistic arguments for the convergence of solutions of the coupled system of HJB equations \eqref{eq:N.hjb.control} to the master equation.

\medskip
 In addition, it follows from \Cref{rem.rate} that an explicit non-asymptotic convergence rate can be given.
 Note in passing that existence of smooth solutions of the master equation has been investigated by \citeauthor*{cardaliaguet2019master} \cite{cardaliaguet2019master}, and \citeauthor*{chassagneux2022probabilistic} \cite{chassagneux2022probabilistic}. 
 Thus, we have the following corollary which is a direct consequence of \Cref{thm:main.limit} and \Cref{rem.rate}.
 \begin{corollary}\label{cor:master}
 	Under the conditions of {\rm\Cref{thm:main.limit}}, if the {\rm PDE} \eqref{eq:master.control} admits a classical solution $v$ such that $v(0,\chi,\mu) = Y^{\chi,\mu}_0$ $($where $Y^{\chi,\mu}$ solves {\rm\Cref{eq:McKV.eq.intro}} with $X$ starting for the random variable $\chi$ with law $\mu)$, then 
 	for any $i\in\{1,\dots,N\}$
 	\begin{equation*}
 		\E^{\P^{\smalltext{\hat\alpha}}}\big[\big|v^{i,\np}(0, \chi^i,\dots, \chi^\np) - V(0, \chi^i,\mu)\big|^2\big] \le C\bigg(\frac1N+r_{\text{\fontsize{5}{5}\selectfont$\np$},1+d,q}\bigg),\; \forall  (N,q)\in \N^\star\times(2,+\infty)
 	\end{equation*}
 	where $(\chi^i)_{i\in\{1,\dots,N\}}$ are $N$ {\rm i.i.d.} and $\cF_0$-measurable random variables with finite $q$-moment.
 \end{corollary}
In other words, \Cref{thm:main.limit} provides probabilistic arguments for the convergence of partial differential equations.
Results in this direction have been pioneered by \citeauthor*{cardaliaguet2019master} \cite[Theorem 2.13]{cardaliaguet2019master} using fully analytic techniques. 
Notice however that \cite{cardaliaguet2019master} includes common noise and sets the problem on the torus whereas the present case considers interaction through the controls.
The convergence of viscosity solutions was more recently investigated by \citeauthor*{gangbo2021finite} \cite{gangbo2021finite}.

\medskip
Without the Markovian assumption made in this subsection, the value function of the $\np$-player game (at equilibrium) can be seen as a viscosity solution of a path-dependent PDE. A simple case when this hold is when $\sigma$ is constant, see \cite[Theorem 4.3]{ekren2014viscosity}.
 But of course, a path-dependent PDE interpretation of the mean-field game in this setting is still uncharted ground.

\subsubsection{The case of closed-loop controls}
\label{eq:close.loop}

Closed-loop control, which are given as functions of the states, are arguably best suited to model players' behaviours.
In fact, in most games players update their controls based on the position of (all) the participants to the game, rather than based on the randomness (or noise) in the system as suggested by open-loop controls.
In other terms, the controls are functions of the states.
The important difference in the context of stochastic differential games is that, while for open-loop controls the fact that a given player changes their strategy does not have any incidence on the strategies of the other players, this is not the case for closed-loop controls, at least in the strong formulation (See \cite[Section 2.1.2]{carmona2018probabilisticI} for details).
This is due to the fact that in the strong formulation of the game, a change in the control should imply a change in the state.
Let us recall the following definition of closed-loop Nash equilibrium (in the present weak formulation) for completeness.
It is taken and adapted from \cite[Definition 2.6]{carmona2018probabilisticI}.
\begin{definition}
\label{def:closed.loop.Nash}
	A closed-loop Nash equilibrium is a family of $\np$ control processes $\hat\alpha\coloneqq (\hat\alpha^i)_{i\in\{1,\dots,N\}}\in\cA^\np$ such that for any $i\in\{1,\dots,\np\}$, $\hat\alpha^i_t = \hat\phi^i(t, \X_{\cdot\wedge t}^N)$ for some Borel-measurable function $\hat\phi^i$, and we have
\[
	\EE^{\PP^{\smalltext{\hat\alpha}}}\bigg[\int_0^Tf_s\big( X^i_{\cdot\wedge s},L^N(\XX^N_{\cdot\wedge s},\hat\alpha_s),\hat\alpha^i_s\big)\mathrm{d}s+g\big(X^i,L^N(\X^N)\big)\bigg]\ge \EE^{\PP^{\smalltext{\alpha}}}\bigg[\int_0^Tf_s\big(X^i_{\cdot\wedge s},L^N(\X^N_{\cdot\wedge s},\alpha_s),\alpha^i_s\big)\mathrm{d}s+g\big(X^i,L^N(\X^N)\big)\bigg],
\]
for every $i\in\{1,\dots,N\},$ and every Borel-measurable function $\phi^i$, where we defined 
\[
	\alpha_t\coloneqq  \big(\hat\phi^1(t,\X^N_{\cdot\wedge t}),\dots,\hat\phi^{i-1}(t,\X^N_{\cdot\wedge t}), \phi^i(t, \X^N_{\cdot\wedge t}),\hat\phi^{i+1}(t,\X_{\cdot\wedge t}^N),\dots,\hat\phi^N(t, \X^N_{\cdot\wedge t})\big),
\]
where $\XX^N$ is the state process when the control $\alpha$ is used.
\end{definition}

\medskip
One notable advantage of the present weak formulation is that it allows to derive convergence of closed-loop Nash equilibria as well.
We owe this to the fact that changes of the control do not affect the state process, but only its law (see also \citeauthor*{possamai2020zero} \cite{possamai2020zero} for additional advantages of the weak formulation for stochastic differential games).
The convergence of closed-loop Nash equilibrium from the weak formulation perspective is easily seen by observing that the PDE representations of our $N$-player and mean-field games coincide with equations derived in \cite{cardaliaguet2019master,delarue2020master} in the context of closed-loop (or Markovian) controls.
More generally, we have the following.
\begin{corollary}
	Assume that the matrix $\sigma^\top\sigma$ is invertible and that $m=d$.
	Under the conditions of {\rm\Cref{thm:main.limit}}, if there is a closed-loop Nash equilibrium $\hat\alpha^{\np}= (\hat\alpha^{1,\np},\dots, \hat\alpha^{\np,\np})$, then $(V^{i,\np})_{N\in \mathbb{N}^\star} $ converges to $ V^{\hat\xi}$  in the sense of \eqref{eq:conv empirical}.
\end{corollary}
\begin{remark}
	The reader will observe that assuming that the matrix $\sigma^\top\sigma$ is invertible and that $m=d$ is needed only to guarantee that the $(\P$-completed$)$ natural filtrations of $\X^N$ and $(W^i)_{i \in \{1,\dots,N\}}$ are identical. Any other conditions implying this identification can be used instead.
\end{remark}

\begin{proof}
	Since $\sigma^\top\sigma$ is invertible and $m=d$, the $\P$-completed filtrations of $\X^N$ and $(W^i)_{i\in \{1,\dots,N\}}$ coincide (see for instance \citeauthor*{soner2011quasi} \cite[Lemma 8.1]{soner2011quasi}).
	Let $\hat\alpha^N\coloneqq (\hat\alpha^{i,\np})_{i\in\{1,\dots,\np\}}$ be a closed-loop Nash equilibrium.
	Then, for any $i\in\{1,\dots,N\}$, it can be written as $\hat\alpha^{i,\np}_t = \hat\phi^i(t, \X_{\cdot\wedge t}^N)$, $\d t\otimes\d\P$--a.e., for some Borel-measurable function $\hat\phi^i$.
	Since $\X^N$ is adapted to the ($\P$-completed) natural filtration of $(W^1,\dots, W^\np)$, it follows that $\hat\alpha^{i,\np}$ is an open-loop control for every $i\in\{1,\dots,N\}$.
	Let us now show that $\hat\alpha^{\np}$ is an open-loop Nash equilibrium.
	Let $\alpha^i$ be an admissible control, and define $\alpha^N\coloneqq (\alpha_t^i)_{i \in \{1,\dots,\np\}}$ by
	\begin{equation*}
		\alpha_t^N \coloneqq  \big(\hat\phi^1(t, \X_{\cdot\wedge t}^N),\dots, \hat\phi^{i-1}(t, \X^N_{\cdot\wedge t}),\alpha^i_t, \hat\phi^{i+1}(t, \X^N_{\cdot\wedge t}),\dots, \hat\phi^\np(t,\X^N_{\cdot\wedge t}) \big),\; \d t\otimes\d\P\text{\rm--a.e.},
	\end{equation*}
	for some open-loop control $\alpha^i$. Since the filtrations of $\X^N$ and $(W^i)_{i\in \{1,\dots,N\}}$ coincide, it follows that there is a Borel-measurable function $\phi$ such that $\alpha^i = \phi(t, \X^N_{\cdot\wedge t})$.
	In particular, $X$ remains unchanged when using the control $\alpha$.
	Using that $(\hat\alpha^{i,\np})_{i\in\{1,\dots,\np\}}$ is a closed-loop Nash equilibrium, we therefore have
	\begin{equation*}
		J\big((\hat\alpha^{i,\np})_{i\in\{1,\dots,\np\}}\big) \ge J\big((\alpha^i)_{i\in \{1,\dots,\np\}}\big),
	\end{equation*}
	showing by \Cref{def:closed.loop.Nash} that $\hat\alpha$ is an open-loop equilibrium as well.
	The result now follows from \Cref{thm:main.limit}.
\end{proof}

The papers \cite{cardaliaguet2019master}, \cite{delarue2020master} and \cite{lacker2020convergence} consider closed-loop controls and common noise.
Observe however, that they use completely different arguments, and the type of limits obtained are quite different from ours.

\section{Existence and uniqueness of mean field games in the weak formulation}

\subsection{BSDE characterisation of Nash equilibria}\label{sec:characNash}

We start by a characterisation result for Nash equilibria, namely \Cref{thm:Bellman}.
To derive it, we adapt the well-known Bellman optimality principle (sometime referred to as the martingale optimality principle) to the case of stochastic differential games. We emphasise that the result is by no means new---see the references in \Cref{foot:intro}---and that we present its derivation for the sake of comprehensiveness.

\medskip
Throughout the section, we assume to be given some $\hat\alpha^N\in\mathcal{NA}$. In order to derive a characterisation of Nash equilibria using BSDEs, it will prove useful to define dynamic versions of the value functions of the players. Namely, we define the function\footnote{To be completely rigorous, one should first define a family of random variables $(V^{i,N}(\tau,\hat \alpha^{-i,N}))_{\tau\in\Tc(\F_N)}$, exactly as in \Cref{eq:deftimevalue}. Then the dynamic programming principle below will apply directly to that family, which will then form by \Cref{lemma:supermart} a so-called super-martingale system, which by the results of \citeauthor*{della1981sur} \cite{della1981sur} can then be aggregated into an $\F_N$-optional process. We decided to skip these (classical) subtleties for the sake of brevity. } 
\begin{equation}\label{eq:deftimevalue}
	V_t^{i,N}\big(\hat\alpha^{-i,N}\big)\coloneqq \esssup_{\alpha\in\cA} \EE^{\PP^{\smalltext{\alpha}\smalltext{\otimes}_\tinytext{i}\smalltext{\hat\alpha}^{ \tinytext{-}\tinytext{i}\tinytext{,}\tinytext{N}}\smalltext{,}\smalltext{N}}}\bigg[\int_t^Tf_s\big(X^i_{\cdot\wedge s},L^\np\big(\XX^N_{\cdot\wedge s},(\alpha\otimes_i\hat{\alpha}^{-i,N})_s\big),\alpha_s\big)\mathrm{d}s+g\big(X^i,L^\np(\XX^N)\big)\bigg|\cF_{N,t}
	\bigg],
\end{equation}
for every $i\in\{1,\dots,\np\}$ and $t\in[0,T]$.

\medskip
The first result below is simply the dynamic programming principle. In our setting, where $b$ is bounded, this can be deduced for instance from \citeauthor*{karoui2013capacities2} \cite[Theorem 3.4]{karoui2013capacities2} (given that we only have drift control here, a more accessible references but with exponential utilities, is \citeauthor*{espinosa2015optimal} \cite[Lemma 4.13]{espinosa2015optimal}).

\begin{lemma}\label{lemma:dynprog}
	Let {\rm \Cref{ass.Lambda.charac}} be satisfied. For any $i\in\{1,\dots,\np\}$, and any $\FF_N$--stopping times $\tau$ and $\rho$ such that $0\leq \tau\leq \rho$, we have
	\[
	V_\tau^{i,N}\big(\hat\alpha^{-i,N}\big)=\esssup_{\alpha\in\cA} \EE^{\PP^{\smalltext{\alpha}\smalltext{\otimes}_\tinytext{i}\smalltext{\hat\alpha}^{ \tinytext{-}\tinytext{i}\tinytext{,}\tinytext{N}}\smalltext{,}\smalltext{N}}}\bigg[\int_\tau^\rho f_s\big(X^i_{\cdot\wedge s},L^\np\big(\XX^N_{\cdot\wedge s},(\alpha\otimes_i\hat{\alpha}^{-i,N})_s\big),\alpha_s\big)\mathrm{d}s+V_\rho^{i,N}\big(\hat\alpha^{-i,N}\big)\bigg|\cF_{N,\tau}\bigg].
\] 
\end{lemma}
The following is a consequence of \Cref{lemma:dynprog}, and is a version of the so-called martingale optimality principle.
\begin{lemma}\label{lemma:supermart}
	Let {\rm \Cref{ass.Lambda.charac}} be satisfied. For any $i\in\{1,\dots,\np\}$, and $\alpha\in\cA$, the process $M^{\alpha,i}$ defined by
	 \[
	 	M^{\alpha,i}_t\coloneqq V_t^{i,N}\big(\hat\alpha^{-i,N}\big)+\int_0^tf_s\big(X^i_{\cdot\wedge s},L^N\big(\XX^N_{\cdot\wedge s},(\alpha\otimes_i\hat\alpha^{-i,N})_s\big),\alpha_s\big)\mathrm{d}s,\; t\in[0,T],
	 \]
	is an $\big(\FF_\np,\PP^{\alpha\otimes_\smalltext{i}\hat\alpha^{\text{\fontsize{4}{4}\selectfont $-i,N$}},N}\big)$--super-martingale belonging to $\S^2(\R,\F_N)$. Moreover, the process $M^{\hat{\alpha}^{i,N},i}$ is an $\big(\FF_\np,\PP^{\hat{\alpha}^{\text{\fontsize{4}{4}\selectfont $N$}},N}\big)$-martingale, and has a continuous $\PP$-modification.
\end{lemma}
\begin{proof}
Fix some $\alpha\in\cA$. 
By \Cref{lemma:dynprog}, we have for any $0\leq u\leq t\leq T$
\[
	V_u^{i,N}\big(\hat{\alpha}^{-i,N}\big)\geq\EE^{\PP^{\smalltext{\alpha}\smalltext{\otimes}_\tinytext{i}\smalltext{\hat\alpha}^{ \tinytext{-}\tinytext{i}\tinytext{,}\tinytext{N}}\smalltext{,}\smalltext{N}}}\bigg[\int_u^t f_s\big(X^i_{\cdot\wedge s},L^\np\big(\XX^N_{\cdot\wedge s},(\alpha\otimes_i\hat{\alpha}^{-i,N})_s\big),\alpha_s\big)\mathrm{d}s+V_t^{i,N}\big(\hat{\alpha}^{-i,N}\big)\bigg|\cF_{N,u}\bigg],
\]
from which the super-martingale property is clear. Let us now check the integrability. First, it is standard to show that for any $\alpha\in\Ac$, we have
\[
V_t^{i,N}\big(\hat{\alpha}^{-i,N}\big)= \esssup_{\beta\in\cA_\smalltext{t}(\alpha)}\EE^{\PP^{\smalltext{\alpha}\smalltext{\otimes}_\tinytext{i}\smalltext{\hat\alpha}^{ \tinytext{-}\tinytext{i}\tinytext{,}\tinytext{N}}\smalltext{,}\smalltext{N}}}\bigg[\int_t^T f_s\big(X^i_{\cdot\wedge s},L^\np\big(\XX^N_{\cdot\wedge s},(\beta\otimes_i\hat{\alpha}^{-i,N})_s\big),\beta_s\big)\mathrm{d}s+g\big(X^i,L^\np(\XX^N)\big)\bigg|\cF_{N,t}\bigg],
\]
where $\Ac_t(\alpha)\coloneqq \{\beta\in\Ac:\beta_s=\alpha_s,\; \mathrm{d}s\otimes\mathrm{d}\P\text{\rm--a.e.}\}$. Then, using {\rm \Cref{ass.Lambda.charac}}, we have for some $C>0$ which may change value from line to line but only depends on $\ell_f$, $\ell_g$,  $T$ and any majorant of $\{\bar d(a,a_o):a\in A\}$ (recall that $A$ is compact), that for any $\alpha\in\Ac$
\begin{align*}
\big|V_t^{i,N}\big(\hat{\alpha}^{-i,N}\big)\big|^2&\leq \esssup_{\beta\in\cA_\smalltext{t}(\alpha)}\EE^{\PP^{\smalltext{\beta}\smalltext{\otimes}_{\tinytext{i}}\smalltext{\hat\alpha}^{\tinytext{-}\tinytext{i}\tinytext{,}\tinytext{N}}\smalltext{,}\smalltext{N}}}\bigg[\int_0^T\big|f_s\big(X^i_{\cdot\wedge s},L^\np\big(\XX^N_{\cdot\wedge s},(\beta\otimes_i\hat{\alpha}^{-i,N})_s\big),\beta_s\big)\big|\mathrm{d}s+\big|g\big(X^i,L^\np(\XX^N)\big)\big|\bigg|\cF_{N,t}
	\bigg]^2\\
	&\leq C\esssup_{\beta\in\cA_\smalltext{t}(\alpha)}\EE^{\PP^{\smalltext{\beta}\smalltext{\otimes}_{\tinytext{i}}\smalltext{\hat\alpha}^{\tinytext{-}\tinytext{i}\tinytext{,}\tinytext{N}}\smalltext{,}\smalltext{N}}}\bigg[1+\int_0^T\bigg(\bar d^2(\beta_s,a_o)+\frac1N\sum_{j=1}^N\bar d^2\big((\beta\otimes_i\hat\alpha^{-i,N})_s^j,a_o\big)\bigg)\mathrm{d}s+\max_{i\in\{1,\dots,N\}}\|X^i\|_\infty^2\bigg|\Fc_{N,t}\bigg]^2\\
	&\leq C\esssup_{\beta\in\cA_\smalltext{t}(\alpha)}\EE^{\PP^{\smalltext{\beta}\smalltext{\otimes}_{\tinytext{i}}\smalltext{\hat\alpha}^{\tinytext{-}\tinytext{i}\tinytext{,}\tinytext{N}}\smalltext{,}\smalltext{N}}}\bigg[1+\max_{i\in\{1,\dots,N\}}\|X^i\|_\infty^2\bigg|\Fc_{N,t}\bigg]^2.
\end{align*}
It is immediate to show that the family $\big\{\EE^{\PP^{\beta\otimes_{\text{\fontsize{4}{4}\selectfont $i$}}\hat\alpha^{\text{\fontsize{4}{4}\selectfont $-i,N$}},N}}\big[1+\max_{i\in\{1,\dots,N\}}\|X^i\|_\infty^2\big|\Fc_{N,t}\big]:\beta\in\Ac_t(\alpha)\big\}$ is upward directed, so that there is some $\Ac_t(\alpha)$-valued sequence $(\beta^n)_{n\in\N}$ such that
\[
\esssup_{\beta\in\cA_\smalltext{t}(\alpha)}\EE^{\PP^{\smalltext{\beta}\smalltext{\otimes}_{\tinytext{i}}\smalltext{\hat\alpha}^{\tinytext{-}\tinytext{i}\tinytext{,}\tinytext{N}}\smalltext{,}\smalltext{N}}}\bigg[1+\max_{i\in\{1,\dots,N\}}\|X^i\|_\infty^2\bigg|\Fc_{N,t}\bigg]=\lim_{n\to+\infty}\uparrow\EE^{\PP^{\smalltext{\beta}^\tinytext{n}\smalltext{\otimes}_{\tinytext{i}}\smalltext{\hat\alpha}^{\tinytext{-}\tinytext{i}\tinytext{,}\tinytext{N}}\smalltext{,}\smalltext{N}}}\bigg[1+\max_{i\in\{1,\dots,N\}}\|X^i\|_\infty^2\bigg|\Fc_{N,t}\bigg].
\]
Taking expectations under $\P^{\alpha\otimes_\smalltext{i}\hat\alpha^{\smalltext{-}\smalltext{i}\smalltext{,}\smalltext{N}},N}$ above, using the monotone convergence theorem, the fact that for any $n\in\N$, $\P^{\beta^\smalltext{n}\otimes_\smalltext{i}\hat\alpha^{\smalltext{-}\smalltext{i}\smalltext{,}\smalltext{N}},N}$ and $\P^{\alpha\otimes_\smalltext{i}\hat\alpha^{\smalltext{-}\smalltext{i}\smalltext{,}\smalltext{N}},N}$ coincide on $\Fc_{N,t}$, we deduce thanks to Doob's inequality that
\[
\EE^{\PP^{\smalltext{\alpha}\smalltext{\otimes}_\tinytext{i}\smalltext{\hat\alpha}^{ \tinytext{-}\tinytext{i}\tinytext{,}\tinytext{N}}\smalltext{,}\smalltext{N}}}\bigg[\sup_{t\in[0,T]}\big|V_t^{i,N}\big(\hat{\alpha}^{-i,N}\big)\big|^2\bigg]\leq C \lim_{n\to+\infty}\EE^{\PP^{\smalltext{\beta}^\tinytext{n}\smalltext{\otimes}_{\tinytext{i}}\smalltext{\hat\alpha}^{\tinytext{-}\tinytext{i}\tinytext{,}\tinytext{N}}\smalltext{,}\smalltext{N}}}\bigg[1+\max_{i\in\{1,\dots,N\}}\|X^i\|_\infty^4\bigg]\leq C\esssup_{\beta\in\cA^\smalltext{N}}\E^{\P^\smalltext{\beta}}\bigg[1+\max_{i\in\{1,\dots,N\}}\|X^i\|_\infty^4\bigg],
\]
the latter being finite since by boundedness of $b$, and compactness of $A$, we know that $\|\X^N\|_\infty$ has moments of any order under any $\PP^{\beta}$, which are bounded uniformly over $\beta\in\Ac^N$. 

\medskip
Using the definition of $M^{\alpha,i}$, it is then immediate using similar arguments that $M^{\alpha,i}$ is an $\big(\FF_\np,\PP^{\alpha\otimes_\smalltext{i}\hat\alpha^{\text{\fontsize{4}{4}\selectfont $-i,N$}},N}\big)$--super-martingale, which in addition belongs to $\S^2(\R,\F_N)$. Next, since $\hat\alpha^N\in\mathcal{NA}$, we have for any $\FF_\np$--stopping time $\tau$ 
\begin{align*}
	V_0^{i,N}\big(\hat{\alpha}^{-i,N}\big) = M^{\hat{\alpha}^{\text{\fontsize{4}{4}\selectfont $i,N$}},i}_0\geq \EE^{\PP^{\smalltext{\hat{\alpha}}^{\tinytext{N}}\smalltext{,}\smalltext{N}}}\big[M^{\hat{\alpha}^{\smalltext{i}\smalltext{,}\smalltext{N}},i}_\tau\big]&\geq \EE^{\PP^{\smalltext{\hat{\alpha}}^{\tinytext{N}}\smalltext{,}\smalltext{N}}}\bigg[\int_0^Tf_s\big(X^i_{\cdot\wedge s},L^\np\big(\XX^N_{\cdot\wedge s},(\alpha\otimes_i\hat{\alpha}^{-i,N})_s\big),\alpha_s\big)\mathrm{d}s+g\big(X^i,L^\np(\XX^N)\big)\bigg]
	\\
	&=V_0^{i,N}\big(\hat{\alpha}^{-i,N}\big).
\end{align*}
Because $\FF_\np$ is right-continuous, the fact that the above holds for an arbitrary stopping time $\tau$ implies indeed that $M^{\hat{\alpha}^{\text{\fontsize{4}{4}\selectfont $i,N$}},i}$ is an $\big(\FF_\np,\PP^{\hat{\alpha}^{\text{\fontsize{4}{4}\selectfont $N$}},N}\big)$-martingale. 
The existence of a c\`adl\`ag $\PP$-modification is then again due to the right-continuity of $\FF_\np$, and the fact that this modification is actually continuous comes from the martingale representation property, recall that $\F_N$ is a completed Brownian filtration.
\end{proof}

We can proceed with the 
\begin{proof}[Proof of \Cref{thm:Bellman}]
For any $i\in\{1,\dots,\np\}$, using the martingale representation theorem (more precisely here one should use \citeauthor*{jacod2003limit} \cite[Theorem III.5.24]{jacod2003limit}) and the integrability of $M^{\hat\alpha^{\text{\fontsize{4}{4}\selectfont $i,N$}},i}$, we know that there exists an $\RR^{d\times \np}$-valued, $\FF_\np$-predictable process $\ZZ^{i,\np}\coloneqq  (Z^{i,1,\np}, \dots, Z^{i,\np,\np})$ such that for any $j\in\{1,\dots,N\}$, it holds that
\[
	\EE^{\PP^{\smalltext{\hat{\alpha}}^{\tinytext{N}}\smalltext{,}\smalltext{N}}}\bigg[\sum_{j=1}^\np\int_0^T\|Z_s^{i,j,\np}\|^2\mathrm{d}s\bigg]<\infty, \; \text{\rm and}\; 
	M^{\hat{\alpha}^{\text{\fontsize{4}{4}\selectfont $i,N$}},i}_t=M^{\hat{\alpha}^{\text{\fontsize{4}{4}\selectfont $i,N$}},i}_0+\sum_{j=1}^\np\int_0^tZ_s^{i,j,\np}\cdot\mathrm{d}\big(W_s^{\hat{\alpha}^{\text{\fontsize{4}{4}\selectfont $N$}},N}\big)^{j},\; t\in[0,T].
\]
For any $\alpha\in\cA$, this implies that
\begin{align*}
	\mathrm{d}M^{\alpha,i}_s &= \Big(f_s\big(X^i_{\cdot\wedge s},L^{\np}\big(\XX^N_{\cdot\wedge s},(\alpha\otimes_i\hat{\alpha}^{-i,N})_s\big),\alpha_s\big)-f_s\big(X^i_{\cdot\wedge s},L^{\np}\big(\XX^N_{\cdot\wedge s},(\alpha\otimes_i\hat{\alpha}^{-i,N})_s\big),\hat{\alpha}^N_s\big)\Big)\mathrm{d}s + \d M^{\hat\alpha^{\text{\fontsize{4}{4}\selectfont $i,N$}},i}_s\\
	&= \Big( b_s\big(X^i_{\cdot\wedge s},L^\np\big(\XX^N_{\cdot\wedge s},(\alpha\otimes_i\hat{\alpha}^{-i,N})_s\big),\alpha^i_s\big) - b_s\big(X^i_{\cdot\wedge s},L^\np\big(\XX^N_{\cdot\wedge s},\hat{\alpha}^N_s\big),\hat{\alpha}^{i,N}_s\big)\Big)\cdot Z^{i,i,\np}_s\mathrm{d}s\\
	&\quad +\Big(f_s\big(X^i_{\cdot\wedge s},L^\np\big(\XX^N_{\cdot\wedge s},(\alpha\otimes_i\hat{\alpha}^{-i,N})_s\big),\alpha_s\big)-f_s\big(X^i_{\cdot\wedge s},L^\np\big(\XX^N_{\cdot\wedge s},\hat{\alpha}^N_s\big),\hat{\alpha}^{i,N}_s\big)\Big)\mathrm{d}s+\sum_{j=1}^\np Z_s^{i,j,\np}\cdot\mathrm{d}\big(W_s^{\alpha\otimes_\smalltext{i}{\hat\alpha}^{\text{\fontsize{4}{4}\selectfont $-i,N$}}}\big)^{j}\\
	&\quad +\sum_{j\in\{1,\dots,N\}\setminus\{ i\}}\Big(b_s\big(X^j_{\cdot\wedge s},L^\np\big(\XX^N_{\cdot\wedge s},(\alpha\otimes_i\hat{\alpha}^{-i,N})_s\big),\hat\alpha^j_s\big) - b_s\big(X^j_{\cdot\wedge s},L^\np\big(\XX^N_{\cdot\wedge s},\hat{\alpha}^{N}_s\big),\hat\alpha^j_s\big) \Big)\cdot Z^{i,j,\np}_s\mathrm{d}s .
\end{align*}
Since $M^{\alpha,i}$ must be an $\big(\FF_\np,\PP^{\alpha\otimes_\smalltext{i}\hat{\alpha}^{\text{\fontsize{4}{4}\selectfont $-i,N$}},N}\big)$--super-martingale, we deduce that (recall the function $h$ defined in \Cref{eq:h.def})
\begin{align*}
	&h\big(X^i_{\cdot\wedge s},L^\np\big(\XX^N_{\cdot\wedge s},(\alpha\otimes_i\hat{\alpha}^{-i,N})_s\big),Z^{i,i,\np}_s,\alpha_s\big)+\sum_{j\in\{1,\dots,N\}\setminus\{ i\}}b_s\big(X^j_{\cdot\wedge s},L^\np\big(\XX^N_{\cdot\wedge s},(\alpha\otimes_i\hat{\alpha}^{-i,N})_s\big),\hat\alpha^j_s\big)\cdot Z^{i,j,\np}_s \\
	&\quad \leq h\big(X^i_{\cdot\wedge s},L^\np\big(\XX^N_{\cdot\wedge s},\hat{\alpha}^N_s\big),Z^{i,i,\np}_s,\hat{\alpha}^{i,N}_s\big) +\sum_{j\in\{1,\dots,N\}\setminus\{ i\}}b_s\big(X^j_{\cdot\wedge s},L^\np\big(\XX^N_{\cdot\wedge s},\hat{\alpha}^{N}_s\big),\hat\alpha^j_s\big)\cdot Z^{i,j,\np}_s,\; \mathrm{d}s\otimes\mathrm{d}\PP\text{\rm--a.e.},
\end{align*}
and therefore that for any $i\in\{1,\dots,\np\}$
\begin{align*}
	&f_s\big(X^i_{\cdot\wedge s},L^\np\big(\XX^N_{\cdot\wedge s},\hat{\alpha}^N_s\big),\hat{\alpha}^{i,N}_s\big) +\sum_{j=1}^Nb_s\big(X^j_{\cdot\wedge s},L^\np\big(\XX^N_{\cdot\wedge s},\hat{\alpha}^{N}_s\big),\hat\alpha^j_s\big)\cdot Z^{i,j,\np}_s\\
 	&=h_s\big(X^i_{\cdot\wedge s},L^\np\big(\XX^N_{\cdot\wedge s},\hat{\alpha}^N_s\big),Z^{i,i,\np}_s,\hat{\alpha}^{i,N}_s\big) +\sum_{j\in\{1,\dots,N\}\setminus\{i\}}b_s\big(X^j_{\cdot\wedge s},L^\np\big(\XX^N_{\cdot\wedge s},\hat{\alpha}^{N}_s\big),\hat\alpha^j_s\big)\cdot Z^{i,j,\np}_s\\
   & =\sup_{a\in A}\bigg\{h\big(X^i_{\cdot\wedge s},L^\np\big(\XX^N_{\cdot\wedge s},(a\otimes_i\hat{\alpha}^{-i,N})_s\big),Z^{i,i,\np}_s,a\big) +\sum_{j\in\{1,\dots,N\}\setminus\{i\}}b_s\big(X^j_{\cdot\wedge s},L^\np\big(\XX^N_{\cdot\wedge s},(a\otimes_i\hat{\alpha}^{-i,N})_s\big), \hat\alpha^j_s\big)\cdot Z^{i,j,\np}_s \bigg\} ,\; \mathrm{d}s\otimes\mathrm{d}\PP\text{\rm--a.e.}
\end{align*}
This exactly means that $\hat\alpha^{N}_t\in \Oc^N\big(t,\X^N_{\cdot\wedge t},(Z^{i,j,N}_t)_{(i,j)\in\{1,\dots,N\}^\smalltext{2}}\big),\; 	\d t\otimes\d \P\text{\rm--a.e.}$. Defining now for $i\in\{1,\dots,N\}$ the $\RR$-valued process $Y^{i,N}\coloneqq V^{i,N}(\hat{\alpha}^{-i,N})$, we have thus obtained that $(Y^{i,N},Z^{i,j,N})_{(i,j)\in\{1,\dots, \np\}^\smalltext{2}}$ satisfies
BSDE \eqref{eq:bsde.main}.

\medskip

Finally, we deduce by integrability property of $M^{\hat\alpha^{i,N},i}_t$ derived in \rm\Cref{lemma:supermart} that 
	\begin{equation*}
		\EE^{\PP^{\smalltext{\hat{\alpha}}^{\tinytext{N}}\smalltext{,}\smalltext{N}}}\bigg[\sup_{t \in [0,T]}|Y^i_t|^2 + \sum_{j=1}^N\int_0^T\|Z^{i,j,N}_s\|^2\mathrm{d}s \bigg]<\infty,\; i\in\{1,\dots,N\}.
	\end{equation*}
\end{proof}

\begin{remark}
	We proved here that a Nash equilibrium is necessarily related to the solution of the above {\rm BSDE}, and that it has to be equal to a fixed point of the function $\Hc^N$ in the sense of {\rm \Cref{def:fixed-point-Hc}}. 
	We can also provide a converse statement in the sense that if there exists a fixed--point for $\Hc^N$ and a sufficiently integrable solution to the {\rm BSDE}, then it allows to construct a Nash equilibrium. 
	The reasoning is clear, and uses in particular the comparison theorem for one-dimensional {\rm BSDEs}.
	An argument along these lines will be used for mean-field games in the next section.
\end{remark}

\subsection{Existence and characterisation of mean-field equilibria}
\label{sec:MFG-cha}
Let us now focus on deriving a characterisation similar to that of \Cref{thm:Bellman}, but for mean-field games.
In essence, we will derive first a reverse result: we give a condition based on a BSDE, guaranteeing that a control strategy is a mean-field equilibrium.
A direct byproduct of this proposition is a new method to prove existence of mean-field equilibria in the weak formulation.
This method is adopted to prove \Cref{thm:main.existence.MFG}. Note however that the derived BSDE is rather esoteric.
In fact, the underlying probability measure and the driving noise both depend on the unknown. In that sense, it is reminiscent of the so-called McKean--Vlasov second-order BSDEs introduced in \citeauthor*{elie2021mean} \cite{elie2021mean} for a specific model, and in \citeauthor*{barrasso2022controlled} \cite{barrasso2022controlled} in a general setting, in order to characterise mean-field equilibria in stochastic differential games where volatility control is allowed (notice however that these references do not provide well-posedness results, and simply point out the connection).
The study of existence and uniqueness of this new type of equations is done in \Cref{sec:exists.MKVBSDE}. At the end of the section, we give a version of \Cref{thm:Bellman} adapted to the mean-field game setting, showing that any mean-field equilibrium must arise as solutions to the aforementioned new type of BSDE, which in turn will yield the argument for uniqueness of the mean-field game. 

\medskip

\begin{proof}[Proof of \rm\Cref{prop:char.mfe}]
	\emph{Step 1: necessary condition.} Let us first assume that \rm\Cref{eq:bsde.char.mfg} admits a solution satisfying \rm\Cref{eq:bsde.char.mfg.integrability} where $\hat\alpha\coloneqq \hat a\in\A$ is a maximiser of the Hamiltonian.
	Since $Z$ is sufficiently integrable, we have by taking expectations
	\begin{equation*}
		Y_0 = \EE^{\PP^{\smalltext{\hat\alpha}}}\bigg[g\big(X, \cL_{\hat\alpha}(X)\big) + \int_0^Tf_s\big(X_{\cdot\wedge s}, \cL_{\hat\alpha}(X_{\cdot\wedge s}, \hat\alpha_s),\hat\alpha_s\big)\mathrm{d}s \bigg],
	\end{equation*}
	and by Girsanov's theorem it holds that
	\begin{equation*}
		Y_t = g\big(X, \cL_{\hat\alpha}(X)\big) + \int_t^T \Big(f_s\big(X_{\cdot\wedge s}, \cL_{\hat\alpha}(X_{\cdot\wedge s}, \hat\alpha_s),\hat\alpha_s\big) - b_s\big(X_{\cdot\wedge s},\cL_{\hat\alpha}(X_{\cdot\wedge s}, \hat\alpha_s), \hat\alpha_s \big)\cdot Z_s\Big)\mathrm{d}s - \int_t^TZ_s\cdot \mathrm{d}W_s,\; \P\text{\rm--a.s.}
	\end{equation*}
	On the other hand, define $\xi_t \coloneqq  \cL_{\hat\alpha}(X_{\cdot\wedge t}, \hat\alpha_t)$, and let the first marginal of $\xi_t$ be denoted by $\xi^1_t$. We have by definition that $\xi\coloneqq (\xi_t)_{t\in[0,T]}\in\mathfrak B$. Now let $\alpha \in \mathfrak A$ be an arbitrary control strategy.
	The following (linear) BSDE parameterised by $\alpha$ admits a unique solution with $(Y^\alpha,Z^\alpha)\in\S^2(\R,\F,\P^\alpha)\times\H^2(\R,\F,\P^\alpha)$
	\begin{equation*}
		Y_t^\alpha = g(X, \xi^1) + \int_t^T f_s\big(X_{\cdot\wedge s}, \xi_s,\alpha_s\big) \mathrm{d}s - \int_t^TZ_s^\alpha\cdot \mathrm{d}W^{\alpha}_s,\; t\in[0,T],\; \P\text{\rm--a.s.}
	\end{equation*}
	This is obvious for instance from the well-posedness results in \citeauthor*{el1997backward} \cite{el1997backward}, since the terminal condition is square-integrable under $\P^\alpha$ by \Cref{ass.Lambda.charac}, the generator is clearly uniformly Lipschitz-continuous by boundedness of $b$, and its value at $0$ has the required integrability because of the growth condition on $f$ from \Cref{ass.Lambda.charac}.
	By the comparison theorem for BSDEs (written under $\P$) see again \cite{el1997backward}, we have $Y^\alpha_0 \le Y_0$ (recall that by definition, $\hat\alpha_t\in \A(t,X_{\cdot\wedge t},\xi^1_t,Z_t)$) and,
	applying Girsanov's theorem again it holds
	\begin{equation*}
		Y^\alpha_0  = \EE^{\PP^{\smalltext{\alpha}}}\bigg[g(X,\cL_\alpha(X)) +\int_0^Tf_s(X_{\cdot\wedge s}, \xi_s,\alpha_s)\mathrm{d}s  \bigg].
	\end{equation*}
	This shows that $\hat\alpha$ is optimal (since it is obvious here that $\hat\alpha\in\mathfrak A$), and by construction, we have $\xi_\cdot = \PP^{\hat\alpha}\circ (X_{\cdot\wedge \cdot}, \hat\alpha_\cdot)^{-1}$, which ends the proof of this implication.
\medskip

	\emph{Step 2: sufficient condition.}
	Let us now assume that $\hat\alpha \in \mathfrak{A}$ is a mean-field equilibrium.
	Exactly as in \Cref{sec:characNash}, we let $\xi\coloneqq \Lc_{\hat\alpha}(X,\hat\alpha)$, and define
	\begin{equation}\label{eq:deftimevalueMFG}
		V_t^{\xi}\coloneqq \esssup_{\alpha\in\mathfrak A}\EE^{\PP^{\smalltext{\hat\alpha}\smalltext{,}\smalltext{\xi}}}\bigg[\int_t^Tf_s(X_{\cdot\wedge s},\xi_s,\alpha_s)\mathrm{d}s+g(X,\xi_T^1)\bigg|\cF_{t} \bigg].
\end{equation}

	Again, the following dynamic programming principle holds, in the sense that for any $\FF$--stopping times $0\leq \tau\leq \rho$, we have
	\[
		V_\tau^{\xi}=\esssup_{\alpha\in\mathfrak A} \EE^{\PP^{\smalltext{\alpha}\smalltext{,}\smalltext{\xi}}}\bigg[\int_\tau^\rho f_s(X_{\cdot\wedge s},\xi_s,\alpha_s)\mathrm{d}s+V_\rho^{\xi}\bigg|\cF_{\tau}\bigg].
	\] 
As a consequence, and following exactly the same reasoning as in the proof of \Cref{lemma:supermart}, we have that for any $\alpha\in\mathfrak A$, the process $M^{\alpha}$ defined by
	 \[
	 	M^{\alpha}_t\coloneqq V_t^{\xi}+\int_0^tf_s(X_{\cdot\wedge s},\xi_s,\alpha_s)\mathrm{d}s,\; t\in[0,T],
	 \]
	is an $(\FF,\PP^{\alpha,\xi})$--super-martingale in $\S^2(\R,\F)$, the process $M^{\hat{\alpha}}$ is an $(\FF,\PP^{\hat\alpha,\xi})$-martingale in $\S^2(\R,\F)$, and has a continuous $\PP$-modification. Now, using the martingale representation theorem and the square-integrability of $M^{\hat\alpha,\xi}$, we know that there exists a process $Z\in\H^2(\R^d,\F,\P^{\hat\alpha,\xi})$ such that 
	\[
		M^{\hat{\alpha},\xi}_t=M^{\hat{\alpha},\xi}_0+\int_0^tZ_s\cdot\mathrm{d}W_s^{\hat{\alpha},\xi},\; t\in[0,T].
	\]
	For any $\alpha\in\mathfrak A$, this implies that
	\begin{align*}
		\mathrm{d}M^{\alpha}_s &= \big(f_s(X_{\cdot\wedge s},\xi_s,\alpha_s)-f_s(X_{\cdot\wedge s},\xi_s,\hat{\alpha}_s)\big)\mathrm{d}s + \d M^{\hat\alpha}_s\\
		&= \big( b_s(X_{\cdot\wedge s},\xi_s,\alpha_s) - b_s(X_{\cdot\wedge s},\xi_s,\hat{\alpha}_s)\big)\cdot Z_s\mathrm{d}s +\big(f_s(X_{\cdot\wedge s},\xi_s,\alpha_s)-f_s(X_{\cdot\wedge s},\xi_s,\hat{\alpha}_s)\big)\mathrm{d}s+ Z_s\cdot\mathrm{d}W_s^{\alpha,\xi}.
	\end{align*}

	Since $M^{\alpha}$ must be an $\big(\FF,\PP^{\alpha,\xi}\big)$--super-martingale, we have $
		h\big(X_{\cdot\wedge s},\xi_s,Z_s,\alpha_s\big)\leq h\big(X_{\cdot\wedge s},\xi_s,Z_s,\hat{\alpha}_s\big),\; \mathrm{d}s\otimes\mathrm{d}\PP\text{\rm--a.e.},$ and therefore that
	\[
  		h\big(X_{\cdot\wedge s},\xi_s,Z_s,\hat{\alpha}_s\big)=\sup_{a\in A}\big\{h\big(X_{\cdot\wedge s},\xi_s,Z_s,a\big)\big\},\; \mathrm{d}s\otimes\mathrm{d}\PP\text{\rm--a.e.}
	\]
	This exactly means that $\hat\alpha_t \in \A\big(t,X_{\cdot\wedge t}, \cL_{\hat\alpha}(X_{\cdot\wedge t},\hat\alpha_t),Z_t\big) ,\; \mathrm{d}t\otimes\mathrm{d}\PP\text{\rm--a.e.}$ Defining now the $\RR$-valued process $Y\coloneqq V^{\hat\alpha,\xi}$, we have thus obtained that the process $(Y,Z)$ satisfies
	\begin{align*}
		Y_t= g(X,\xi_T^1) + \int_t^Th_s\big(X_{\cdot\wedge s},\xi_s,Z_s,\hat\alpha_s\big)\mathrm{d}s- \int_t^TZ_s\cdot \mathrm{d}W_s.
	\end{align*}
	This proves using Girsanov's theorem that $(Y,Z)$ solves BSDE \eqref{eq:bsde.char.mfg}, and it has the required integrability in \Cref{eq:bsde.char.mfg.integrability}.
\end{proof}

The first consequence of the above characterisation of the mean-field game in the weak formulation by backward SDEs is the existence and uniqueness result given in \Cref{thm:main.existence.MFG}.
\begin{proof}[Proof of \Cref{thm:main.existence.MFG}]
	By the characterisation \rm\Cref{prop:char.mfe}, the mean field game admits a mean field equilibrium $\hat\alpha \in \mathfrak{A}$ if and only if the BSDE \eqref{eq:bsde.char.mfg} admits a solution where $\Lambda_t(\x, \xi^1, z) \in \argmax_{a\in A}\{h_t(\x, \xi, a, z)\}$.
	Under \rm\Cref{assum:main.conv}, \rm\Cref{eq:bsde.char.mfg} reduces to
	\begin{equation*}
	\label{eq:bsde.char.mfg.proof.mfg.exists}
		\begin{cases}
			\displaystyle Y_t = g\big(X, \cL_{\hat\alpha}(X)\big) + \int_t^Tf_s\big(X_{\cdot\wedge s}, \cL_{\hat\alpha}(X_{\cdot\wedge s}, \hat\alpha_s),  \hat\alpha_s\big)\mathrm{d}s - \int_t^TZ_s\cdot \mathrm{d}W^{\hat\alpha}_s,\; t\in[0,T],\; \PP^{\hat\alpha}\text{\rm--a.s.},\\[0.8em]
			\displaystyle\hat\alpha_t = \Lambda_t\big(X_{\cdot\wedge t}, \cL_{\hat\alpha}(X_{\cdot\wedge t}),Z_t\big), \; \frac{\mathrm{d}\PP^{\hat\alpha}}{\mathrm{d}\PP} \coloneqq  \cE\bigg(\int_0^Tb_s\big(X_{\cdot\wedge s},\cL_{\hat\alpha}(X_{\cdot\wedge s}, \hat\alpha_s),\hat\alpha_s \big)\cdot \mathrm{d}W_s\bigg).
		\end{cases}
	\end{equation*}
	By \rm\Cref{assum:main.conv}, \rm\Cref{assum:gen.MkV} is satisfied with $K_B\coloneqq K_b - \ell_b\ell_\Lambda$.
	Thus, by \Cref{thm:Gen.MkV.BSDE}, the generalised McKean--Vlasov BSDE \eqref{eq:bsde.char.mfg.proof.mfg.exists} admits a solution satisfying \eqref{eq:bsde.char.mfg.integrability}.
	This shows existence and uniqueness of the mean-field equilibrium.
\end{proof}

\section{Limit theorems for large population games and existence of mean-field equilibria}\label{sec:limitth}

The goal of this section is to prove the main results of the article, namely \rm\Cref{thm:main.limit} and its corollaries.
The plan we follow is to begin by proving the characterisation results \rm\Cref{thm:Bellman,prop:char.mfe} so that the convergence problem becomes a propagation of chaos question.
But beforehand, we fully analyse a toy model.
The aim here is to present a simple example that will make the method developed in this article fully transparent to the reader before delving into the more involved general setting.
\subsection{A case study}
\label{sec:case.study}
We assume that the drift $b$ and the reward functions $f$ and $g$ are such that, given $\alpha^{-i}\in\Ac^{N-1}$, the problem faced by player $i\in\{1,\dots,N\}$ takes the form
\begin{equation}
\label{eq:value.example}
 	V^{i,N}(\alpha^{-i}) \coloneqq \sup_{\alpha\in\cA}\EE^{\PP^{\smalltext{\alpha}\smalltext{\otimes}_\tinytext{i}\smalltext{\alpha}^{\tinytext{-}\tinytext{i}\tinytext{,}\tinytext{N}}}}\bigg[\int_0^T\bigg(-\frac12|\alpha_s|^2 + \frac{\kappa_1}{N}\sum_{j=1}^Nf(X^i_s) + \frac{\kappa_2}{N}\sum_{j=1}^N\alpha^j_s\bigg)\mathrm{d}s+ g(X^i_T)\bigg],
\end{equation}

\vspace{-1.2em}
\[
\frac{\diff \PP^{\alpha\otimes_\smalltext{i}\alpha^{\smalltext{-}\smalltext{i}},N}}{\mathrm{d} \PP} \coloneqq \cE\bigg(\sum_{j=1}^N\int_0^\cdot(\alpha^j_s-kX_s^j)\diff W^j_s\bigg), 
\]
where $X^i$ satisfies $ X^i_t = X^i_0+ \sigma W^i_t$ and $f$, $g$ are two bounded, Lipschitz-continuous functions. 
We further assume that $d=m=1$ and $A \subseteq \RR$ is a compact set containing $0$.
We are going to show that a Nash equilibrium for this game converges to a mean-field equilibrium and compute the convergence rate.

 \medskip
	
\emph{Step 1: characterisation for the $N$-player game.}	Let us assume that for all $N\in \NN^\star$ this game admits a Nash equilibrium $(\hat\alpha^{1,N}, \dots, \hat\alpha^{N,N})$.
Then in particular, for each $i\in\{1,\dots,N\}$, the control problem \eqref{eq:value.example} obtained by replacing $\alpha^{-i}$ by $\hat \alpha^{-i,N}$ is solved by $\hat\alpha^{i,N}$.
Thus, standard stochastic control arguments (see the proof of {\rm \Cref{thm:Bellman}} below for details) allow to obtain that $\hat\alpha^{i,N}$ maximises the Hamiltonian along a BSDE solution. 
That is, it satisfies $\diff t\otimes \diff \PP \text{--a.e.}$
\begin{equation}
\label{eq:cond.alpha.example}
 	\hat\alpha^{i,N}_t = \argmax_{a\in A}\bigg\{ \frac{\kappa_1}{N}\sum_{j=1}^Nf(X^j_t) + \frac{\kappa_2}{N}\sum_{j\in \{1,\dots,N\}\setminus \{i\}}\hat\alpha^{j,N}_t + \frac{\kappa_2}{N}a + (a - kX_t^i)Z^{i,i,N}_t + \sum_{j\in \{1,\dots,N\}\setminus \{i\}}^N(\hat\alpha^{j,N}_t-kX_t^j)Z^{i,j,N}_t -\frac12|a|^2 \bigg\},
\end{equation}
where $(Y^{i,N}, Z^{i,j,N})_{(i,j)\in \{1,\dots,N\}^\smalltext{2}}$ solves the BSDE
\begin{align*}
 	Y^{i,N}_t &= g(X^i_T) + \int_t^T\sup_{a\in A}\bigg\{ -\frac12|a|^2 + \frac{\kappa_1}{N}\sum_{j=1}^Nf(X^j_s) + \frac{\kappa_2}{N}\sum_{j\in \{1,\dots,N\}\setminus \{i\}}^N\big(\hat\alpha^{j,N}_s + (\hat\alpha^{j,N}_s-kX_s^j)Z^{i,j,N}_s\big)+\frac{\kappa_2}{N}a + (a-kX_s^i)Z^{i,i,N}_s \bigg\}\mathrm{d} s  \\
 	&\quad - \sum_{j=1}^N\int_t^TZ^{i,j,N}_s\diff W^{j}_s,
\end{align*}
and we have $V^{i,N}(\alpha^{-i}) = Y^{i,N}_0$.
The unique maximiser in {\rm \Cref{eq:cond.alpha.example}} is given by $\hat\alpha^{i,N}_t = \mathrm{P}_A\big( Z^{i,i,N}_t + \frac{\kappa_2}{N}\big),\; \diff t\otimes \diff \PP \text{--a.e.,}$ where $\mathrm{P}_A$ is the projection operator on the set $A$.
In particular, the function $\Lambda$ introduced in {\rm \Cref{ass.Lambda.conv}} reduces to
\begin{equation*}
	\Lambda_t\big(x, \xi, z,\aleph^{i,N}(\x,z)\big) = \mathrm{P}_A\big(z+\aleph^{i,N}_t(\x,z)\big) , \;(t,x,\xi,z,\aleph)\in[0,T]\times\R\times\Pc_2(\R)\times\R\times\R,
\end{equation*}
so that we have here $ \aleph_t^{i,N}(\x,z)\coloneqq  \kappa_2/N$, for any $(t,\x,z)\in[0,T]\times\Cc^N\times\R^{N\times N}$. Thus, the above BSDE simplifies to
\begin{equation}
\label{eq:N.BSDE.example}
	Y^{i,N}_t = g(X^i_T) + \int_t^T\bigg(\frac{1}{N}\sum_{j=1}^N\bigg(\kappa_1f(X^j_s) + \kappa_2\mathrm{P}_A\bigg(Z^{j,j,N}_s +\frac{\kappa_2}{N}\bigg)\bigg) -\frac12\bigg|\mathrm{P}_A\bigg(Z^{i,i,N}_s +\frac{\kappa_2}{N}\bigg)\bigg|^2 \bigg)\diff s	 - \sum_{j=1}^{N}\int_t^TZ^{i,j,N}_s\diff W^{\hat\alpha^{\text{\fontsize{4}{4}\selectfont $N$}},j}_s,\; \P\text{\rm --a.s.},
\end{equation}
with $W^{\hat\alpha^{\text{\fontsize{4}{4}\selectfont $N$}},j}_\cdot \coloneqq W^j_\cdot - \int_0^\cdot (\hat\alpha^{j,N}_s-kX_s^j)\diff s$, $j\in\{1,\dots,N\}$.

 	\medskip

\emph{Step 2: characterisation for the mean-field game.}
Next, assume that for any $i\in\{1,\dots,N\}$, we can uniquely solve the BSDE 
\begin{equation}
\label{eq:Gen.bsde.example}
	\begin{cases}
		\displaystyle Y_t^i = g(X_T^i)  + \int_t^T\bigg(-\frac12\big|\mathrm{P}_A(Z_s^i)\big|^2 +\EE^{\PP^{\smalltext{\hat\alpha}^{\tinytext{i}}}}\big[\kappa_1f(X^i_s) + \kappa_2\mathrm{P}_A(Z_s^i) \big]\bigg)\diff s - \int_t^TZ_s^i\diff W^{\hat\alpha,i}_s,\; t\in[0,T],\; \PP^{\hat\alpha^{\text{\fontsize{4}{4}\selectfont $i$}}}\text{--a.s.},\\[0.8em]
 		\displaystyle	\frac{\diff \PP^{\hat\alpha^{\text{\fontsize{4}{4}\selectfont $i$}}}}{\diff \PP} = \cE\bigg(\int_0^\cdot \big(\mathrm{P}_A(Z_s^i) -kX_s^i\big)\diff W_s^i\bigg)_T, \; W^{\hat\alpha,i} \coloneqq  W^i - \int_0^\cdot \big(\mathrm{P}_A(Z_s^i) - kX_s^i\big)\diff s.
	\end{cases}
\end{equation}	
Then, since $\hat\alpha_t^i\coloneqq  \mathrm{P}_A(Z_t^i) = \Lambda_t(X_{\cdot\wedge t}^i, \xi_t, Z_t^i,0)$ maximises (uniquely) the Hamiltonian, \emph{i.e.}
\begin{equation*}
	\hat\alpha_t^i = \argmax_{a\in A}\bigg\{-\frac12|a|^2 + \int_{\RR^\smalltext{2}}(\kappa_1 f(u) + \kappa_2 v)\xi(\diff u,\diff v) + aZ_t^i \bigg\},
\end{equation*}
it follows as in {\rm \Cref{prop:char.mfe}} that $\hat\alpha$ is the unique solution of the following mean-field game: find $\xi \in\mathfrak B$, and $\hat\alpha\in\mathfrak A$ such that $\hat \alpha$ attains the supremum in the definition of $V^\xi$ with
\begin{equation}
\label{eq:MFG.case.study}
	\begin{cases}
		\displaystyle V^\xi\coloneqq \sup_{\alpha\in \mathfrak{A}}\EE^{\PP^\smalltext{\alpha}}\bigg[\int_0^T\bigg( \int_{\RR^\smalltext{2}}(\kappa_1 f(u) + \kappa_2 v)\xi_s(\diff u,\diff v)-\frac12|\alpha_s|^2\bigg)\mathrm{d}s + g(X_T^i) \bigg],\\[1em]
		\displaystyle\diff X_t^i = \sigma \diff W_t^i, \; \frac{\diff \PP^{\alpha}}{\diff \PP }\coloneqq  \cE\bigg(\int_0^\cdot \alpha_s-kX_s^i\diff W^i_s\bigg)_T,\; \alpha\in\mathfrak A,
	\end{cases}
\end{equation}
and such that the equilibrium condition $\PP^{\hat\alpha^i}\circ (X_t^i, \hat\alpha_t^i)^{-1} = \xi_t$, $\diff t\otimes \diff \PP \text{--a.e.}$ holds.
Moreover, we have $V^{\hat\xi} = Y_0$ with $\hat\xi\coloneqq \PP^{\hat\alpha}\circ (X_t, \hat\alpha_t)^{-1} $.
This follows by the comparison theorem for BSDEs and martingale representation, see the proof of {\rm \Cref{prop:char.mfe}} for details.
The fact that $\hat\alpha$ is the unique mean-field equilibrium follows by uniqueness of the BSDE \eqref{eq:Gen.bsde.example} and \rm\Cref{prop:char.mfe}, while well-posedness of \Cref{eq:Gen.bsde.example} is discussed in {\rm \Cref{thm:Gen.MkV.BSDE}}.
Again, solving this equation is equivalent to solving the mean-field game itself.
In order to derive the convergence of $V^{1,N}(\hat\alpha^{-1})$ to $V^{\hat\xi}$, we will use propagation of chaos arguments to show that $Y^{1,N}_0$ converges to a process $Y_0$.
To make the exposition in this case study even simpler, we will assume $\kappa_2 = 0$.

 	\medskip

\emph{Step 3: propagation of chaos}.
Let us therefore consider the particle system $\big(\barX^{i,N}, \barY^{i,N}, \barZ^{i,j,N}\big)_{(i,j)\in\{1,\dots,N\}^\smalltext{2}}$ formed by solving the coupled FBSDE 
\begin{equation}
\label{eq:N.BSDE.example.aux}
	\begin{cases}
		\displaystyle \barX^{i,N}_t = X_0^i + \int_0^t\big(\mathrm{P}_A(\barZ^{i,i,N}_s) - k\barX^{i,N}_s\big)\diff s + \sigma W^{\hat\alpha,i}_t,\; t\in[0,T],\\[0.8em]
		\displaystyle \barY^{i,N}_t = g(\barX^{i,N}_T) + \int_t^T\bigg(-\frac12\big|\mathrm{P}_A\big(\barZ^{i,i,N}_s\big)\big|^2 + \frac{\kappa_1}{N}\sum_{j=1}^Nf(\barX^{j,N}_s)\bigg) \diff s - \sum_{j=1}^{N}\int_t^T\barZ^{i,j,N}_s\diff W^{\hat\alpha,j}_s,\; t\in[0,T],\; \P^{\hat\alpha}\text{--a.s.},
	\end{cases}
\end{equation}
with the same Brownian motions $(W^{\hat\alpha,1},\dots, W^{\hat\alpha, N})$ given in \Cref{eq:N.BSDE.example} and the probability measure $\P^{\hat\alpha}$ with density 
\begin{equation*}
	\frac{\diff \P^{\hat\alpha}}{\diff \P} \coloneqq  \cE\bigg(\sum_{j=1}^N\int_0^\cdot\big(\hat\alpha^i_s - kX^i_s\big)\diff W^i_s \bigg)_T.
\end{equation*}
Observe that $\P^{\hat\alpha}\circ\big(\barX^{i,N}, \barY^{i,N}, \barZ^{i,j,N}\big)^{-1} = \P^{\hat\alpha^{\text{\fontsize{4}{4}\selectfont $N$}},N}\circ \big(X^i, Y^{i,N}, Z^{i,j,N}\big)^{-1}.$
Thus, it follows that $Y^{i,N}_0 = \barY^{i,N}_0$.
Therefore, it suffices to derive the rate of convergence of the sequence $(\barY^{i,N}_0)_{N\in\N^\star}$ to $Y^i_0$.
To do so, let us first apply It\^o's formula to $(\delta X^{i,N})^2\coloneqq  (\barX^{i,N} - X^i)^2$.
This yields, thanks to Lipschitz-continuity of the projection operator and Young's inequality
\begin{equation}
\label{eq:estim.x.case.study}
 	\e^{\beta t}|\delta X^{i,N}_t|^2 \le \int_0^t\big((\bar\eps^{-1}+\beta - 2k)\e^{\beta s}|\delta X^{i,N}_s |^2 + \bar\eps\e^{\beta s} |\delta Z^{i,i,N}_s|^2\big)\diff s,
\end{equation}
for all $\bar\eps>0$, where we put $\delta Z^{i,j,N}\coloneqq  \barZ^{i,j,N} - Z^i\mathbf{1}_{\{i = j\}}$.
Similarly, applying It\^o's formula to $\e^{\beta t}(\delta Y^{i,N})^2$ with $\delta Y^{i,N} \coloneqq  \barY^{i,N} - Y^i$, for every $\varepsilon>0$ we have 
\begin{align*}
\notag
 		|\delta Y^{i,N}_0|^2 + (1 - \eps)\sum_{j=1}^N\int_0^T\e^{\beta s}|\delta Z^{i,j,N}_s|^2\diff s 
 		&\le \ell_g^2\e^{\beta T}|\delta X_T|^2 + \int_0^T\e^{\beta s}\Big((C_A +1)\varepsilon^{-1} - \beta \Big)|\delta Y^{i,N}_s|^2\d s\\
 		&\quad + \varepsilon \kappa_1^2\int_0^T \e^{\beta s}\bigg(\frac{1}{N}\sum_{j=1}^Nf(\overline X^{j,N}_s) - \E^{\P^{\smalltext{\hat\alpha}}}[f(X_s)]\bigg)^2\diff s
 		 - \sum_{j=1}^N\int_0^T2\e^{\beta s}\delta Y^{i,N}_s\delta Z^{i,j,N}_s\diff W^{\hat\alpha,j}_s,
\end{align*}
where $C_A$ is a constant depending on $A$.
Letting $\varepsilon<1$ and $\beta \ge (C_A + 1)\varepsilon^{-1}$, taking expectation on both sides, if $2k\ge \beta + \bar\varepsilon^{-1}$, then by \eqref{eq:estim.x.case.study} we have
\begin{align}
\notag
 	\E^{\P^{\smalltext{\hat\alpha}}}\bigg[ |\delta Y^{i,N}_0|^2 + (1 - \eps)\sum_{j=1}^N\int_0^T\e^{\beta s}|\delta Z^{i,j,N}_s|^2\diff s \bigg]
 	\notag
 		&\le \E^{\P^{\smalltext{\hat\alpha}}}\bigg[\ell_g^2\e^{\beta T}|\delta X_T|^2 + \frac{\varepsilon \kappa_1^2\ell_f}{N}\sum_{j=1}^N\int_0^T \e^{\beta s}|\delta X^{j,N}_s|^2\mathrm{d}s\\
 		\notag
 		&\quad + \kappa_1^2\int_0^T\bigg(\frac1N\sum_{j=1}^Nf(X^i_s) - \E^{\P^{\smalltext{\hat\alpha}}}[f(X_s)]\bigg)^2\diff s\bigg]\\
 		\label{eq:estim.y.case.study}
 		&\le \bar\varepsilon\ell_g^2\E^{\P^{\smalltext{\hat\alpha}}}\bigg[\int_0^T |\delta Z^{i,i,N}_s|^2\d s \bigg] +  \frac{2\varepsilon\bar \varepsilon \kappa_1^2\ell_fT}{N}\sum_{j=1}^N\E^{\P^{\smalltext{\hat\alpha}}}\bigg[\int_0^T \e^{\beta s}|\delta Z^{i,i,N}_s|^2\d s\bigg] + E_N,
\end{align}
with
\begin{equation*}
	E_N \coloneqq 2\kappa_1^2\E^{\P^{\smalltext{\hat\alpha}}}\bigg[\int_0^T\bigg(\frac1N\sum_{j=1}^Nf(X^i_s) - \E^{\P^{\smalltext{\hat\alpha}}}[f(X_s)]\bigg)^2\diff s\bigg]\le C/N,
\end{equation*}
where the inequality follows using standard law of large number arguments.
Thus, averaging on both sides over $i\in \{1,\dots,N\}$, we have
\begin{align*}
	\E^{\P^{\smalltext{\hat\alpha}}}\bigg[ |\delta Y^{i,N}_0|^2 + &(1 - \eps - \bar\varepsilon(\ell_g^2 + \kappa_1^2 \ell_fT))\frac1N\sum_{i=1}^N\sum_{j=1}^N\int_0^T\e^{\beta s}|\delta Z^{i,j,N}_s|^2\diff s \bigg] \le E_N.
\end{align*}
Thus, first taking $\bar\eps\in (0,1)$ such that $\bar\eps < (\ell_g^2 + 2\ell_fT\kappa_1^2)^{-1}$, and then $\eps>0$ such that $\eps < 1 -\bar\eps\big( \ell_g^2 + 2\ell_fT\kappa_1^2 \big)$, we have 
\begin{equation}
\label{eq:estim.sum.z.case.study}
	\sum_{i=1}^N\sum_{j=1}^N\E^{\P^{\smalltext{\hat\alpha}}}\bigg[ \int_0^T\e^{\beta s}|\delta Z^{i,j,N}_s|^2\diff s\bigg] \le C,
\end{equation}
for some constant $C>0$ that does not depend on $N$.
In particular, the minimum value allowed for $k$ is 
\begin{equation*}
	k  
	  \ge \inf\big\{(C_A+1)\eps^{-1} + \bar\eps^{-1}:\, \eps<(\ell_g^2 + 2\ell_fT\kappa_1^2)^{-1},\,\, \eps < 1 -\bar\eps\big( \ell_g^2 + 2\ell_fT\kappa_1^2\big)\big\}.
\end{equation*}
Thus, coming back to \Cref{eq:estim.y.case.study}, we have
\begin{align*}
	|\delta Y^{i,N}_0|^2 &+ (1 - \eps - \bar\eps\ell_g^2)\sum_{j=1}^N\E^{\P^{\smalltext{\hat\alpha}}}\bigg[\int_0^T\e^{\beta s}|\delta Z^{i,j,N}_s|^2\diff s\bigg] \le \frac{2T\kappa_1^2}{N}\E^{\P^{\smalltext{\hat\alpha}}}\bigg[\int_0^T\sum_{j=1}^N\e^{\beta s}|\delta Z^{j,j,N}_s|^2\diff s\bigg] + E_N\le \frac{C}{N},
\end{align*}
where we used \Cref{eq:estim.sum.z.case.study} to estimate the first term on the right hand side.
Hence, by the choice of $\bar\eps$ and $\eps$ we have $|\delta Y^{i,N}_0| \le C/N$.
We have thus obtained the following.
\begin{proposition}
	Let $\kappa_2 = 0$.
	Assume that for each $N$ the $N$-player game described in \eqref{eq:value.example} admits a Nash equilibrium $(\hat\alpha^{i,N})_{i\in \{1,\dots,N\}}$ and that {\rm\Cref{eq:Gen.bsde.example}} 
	admits a unique solution.
	There is a constant $\delta>0$ depending on $f$, $g$ and $T$ such that if $k\ge \delta$,
	we have
	\begin{equation*}
		\big|V^{i,N}(\hat\alpha^{-i}) - V^{\hat\xi} \big|^2 + \int_0^T\cW_2^2\Big(\P^{\hat\alpha^N,N}\circ (\hat\alpha^{i,N}_s)^{-1}, \P^{\hat\alpha}\circ (\hat\alpha_s)^{-1} \Big)\diff s \le \frac{C}{N}, \; \forall N\in \mathbb{N}^\star,
	\end{equation*}
	for some constant $C>0$.
\end{proposition}
	\begin{proof}
		The proof of the bound of $|V^{i,N}(\hat\alpha^{-i}) - V^{\hat\xi} |^2$ is done above.
		It remains to show the convergence of the law of $\hat\alpha^{i,N}$.
		We have by Lipschitz-continuity of the projection operator
		\begin{align*}
			\cW_2^2\big(\P^{\hat\alpha^\smalltext{N},N}\circ (\hat\alpha^{i,N}_t)^{-1}, \P^{\hat\alpha}\circ (\hat\alpha_t)^{-1} \big) &\le \cW_2^2\big(\P^{\hat\alpha}\circ (\barZ^{i,i,N}_t)^{-1}, \P^{\hat\alpha}\circ Z_t^{-1} \big) \le \E^{\P^{\smalltext{\hat\alpha}}}\big[|\barZ^{i,i,N}_t - Z_t^i|^2 \big],
		\end{align*}
		from which we deduce the bound.
	\end{proof}

\subsection{Proof of Theorem \ref{thm:main.limit}}
\label{sec:proof.convergence}
This section is dedicated to the proof of the convergence given in \rm\Cref{thm:main.limit}. The main idea is to extend the strategy of \Cref{sec:case.study} to the general case. Throughout this section, we let {\rm \Cref{ass.Lambda.conv}} hold, and fix a map $\Lambda$ from {\rm \Cref{ass.Lambda.conv}.$(ii)$}. 

\subsubsection{Step 1: the characterising equations}
Let $\hat{\alpha}^\np \in \mathcal{NA}$ be fixed and denote the associated value function of player $i$ by $V^{i,\np}(\hat\alpha^{-i,\np})$.
By \Cref{thm:Bellman} and \rm\Cref{ass.Lambda.conv}.$(ii)$ there is a function $ \Lambda:[0,T]\times\cC_\xdim\times\cP(\cC_\xdim)\times \RR^{ \bmdim} \times \RR \longrightarrow A$ such that for each $i \in\{ 1, \dots, \np\}$, we have 
\begin{equation}
\label{eq:rep.Nash.proof}
	\hat\alpha^{i,\np}_t  = \Lambda_t\Big(X^i_{\cdot\wedge t}, L^\np(\XX^N_{\cdot\wedge t}), Z_t^{i,i,\np},\aleph^{i,N}_t\big(\X^N_{\cdot\wedge t},(Z^{i,j,N}_t)_{j\in\{1,\dots,N\}}\big)\Big),\; \d t\otimes\d \P\text{\rm--a.e.},\; \text{and}\;  V^{i,N}(\hat\alpha^{-i,N}) = Y^{i,N}_0,
\end{equation}
where $(Y^{i,\np},  Z^{i,j,\np})_{(i,j)\in\{1,\dots,\np\}^\smalltext{2}}$ solves the coupled system of BSDEs
\begin{align}
\label{eq:bsde.n.proof.main.lim}
	Y^{i,\np}_t & =  g\big(X^{i}, L^\np(\XX^N)\big) + \int_t^Tf_s\big(X^i_{\cdot\wedge s}, L^{\np}(\XX^N_{\cdot\wedge s}, \hat{\alpha}_s^N), \hat\alpha_s^{i,\np}\big)\mathrm{d}s -\sum_{j=1}^\np\int_t^T  Z^{i,j,\np}_s\cdot \mathrm{d}W^{\hat\alpha^N,j}_s, \; \P^{\hat\alpha^{\text{\fontsize{4}{4}\selectfont $N$}},N}\text{\rm--a.s.}
\end{align} 
and $\Lambda_t\big(X^i_{\cdot\wedge t}, L^\np(\XX^N_{\cdot\wedge t}), Z_t^{i,i,\np},\aleph^{i,N}_t\big(\X^N_{\cdot\wedge t},(Z^{i,j,N}_t)_{j\in\{1,\dots,N\}}\big)\big) \in \cO^\np(t, (X^i)_{i\in \{1,\dots,\np\}}, (Z^{i,j,\np})_{(i,j)\in \{1,\dots,\np\}^\smalltext{2}})$, that is, $\Lambda$ maximises
	\begin{equation*}
		h_t\big(\x^i,L^N(\x,a^\prime \otimes_ia^{-i}),z^{i,i},a^\prime\big)+\sum_{j\in\{1,\dots,N\}\setminus\{ i\}}b_t\big(\x^j,L^N(\x,a^\prime \otimes_ia^{-i}),a^j\big)\cdot z^{i,j}.
	\end{equation*}	
	along the processes $(X^i, Z^{i,j,\np})_{(i,j)\in \{1,\dots,\np\}^\smalltext{2}}$.
	Moreover, by \rm\Cref{ass.Lambda.conv}.$(ii)$, for every $(t, \x, \xi, z) \in [0,T]\times \cC_m\times \cP_2(\cC_m\times A)\times \R^d$
	\begin{equation*}
 		\Lambda_t(\x, \xi^1, z,0) = \argmax_{a\in A}\big\{h_t(\x,\xi,z,a)\big\},\; (t, \x, \xi, z) \in [0,T]\times \Cc_m\times \cP_2(\Cc_m\times A)\times \RR^d,
 	\end{equation*}
 	where $\xi^1$ is the first marginal of $\xi$.
 	Thus, since the mean-field game admits a unique mean-field equilibrium $\hat\alpha$, it follows by \rm\Cref{prop:char.mfe} that the generalised McKean--Vlasov BSDE
 	\begin{equation}
	\label{eq:bsde.char.mfg.proof.conv} 
		\begin{cases}
			\displaystyle Y_t = g\big(X, \cL_{\hat\alpha^1}(X)\big) + \int_t^Tf_s\big(X_{\cdot\wedge s}, \cL_{\hat\alpha^{\text{\fontsize{4}{4}\selectfont $1$}}}(X_{\cdot\wedge s}, \hat\alpha^1_s),  \hat\alpha_s^1\big)\mathrm{d}s - \int_t^TZ_s\cdot \mathrm{d}W^{\hat\alpha^{\text{\fontsize{4}{4}\selectfont $1$}}}_s,\; t\in[0,T],\; \PP^{\hat\alpha^\smalltext{1}}\text{\rm--a.s.},\\[0.8em]
			\displaystyle\hat\alpha_t^1 \coloneqq  \Lambda_t\big(X_{\cdot\wedge t}, \cL_{\hat\alpha^{\text{\fontsize{4}{4}\selectfont $1$}}}(X_{\cdot\wedge t}),Z_t,0\big), \; \frac{\mathrm{d}\PP^{\hat\alpha^{\text{\fontsize{4}{4}\selectfont $1$}}}}{\mathrm{d}\PP} \coloneqq  \cE\bigg(\int_0^Tb_s\big(X_{\cdot\wedge s},\cL_{\hat\alpha^{\text{\fontsize{4}{4}\selectfont $1$}}}(X_{\cdot\wedge s}, \hat\alpha^1_s),\hat\alpha^1_s \big)\cdot \mathrm{d}W_s\bigg),
		\end{cases}
	\end{equation}
	admits a unique solution $(Y,Z)$ such that $(Y, Z) \in \S^2(\R,\F,\P^{\hat\alpha})\times\H^2(\R^{d},\F,\P^{\hat\alpha})$.
	In the above, a special role was given to the choice $i=1$, but we actually have that for any $i\in\{1,\dots,N\}$, given the Brownian motion $W^i$, the strategy $\hat\alpha^i\in \mathfrak{A}$ is the mean-field equilibrium for the game with Brownian motion $W^i$.
	The associated value function is $V^{\cL_{\smalltext{\hat\alpha}^\tinytext{i}}(X^\smalltext{i},\hat\alpha^\smalltext{i})} = Y^i_0$ where
	\begin{equation}
	\label{eq:bsde.char.mfg.proof.conv.ith} 
		\begin{cases}
			\displaystyle Y_t^i = g\big(X^i, \cL_{\hat\alpha^{\text{\fontsize{4}{4}\selectfont $i$}}}(X^i)\big) + \int_t^Tf_s\big(X_{\cdot\wedge s}^i, \cL_{\hat\alpha^{\text{\fontsize{4}{4}\selectfont $i$}}}(X_{\cdot\wedge s}^i, \hat\alpha_s^i),  \hat\alpha_s^i\big)\mathrm{d}s - \int_t^TZ_s^i\cdot \mathrm{d}W^{\hat\alpha^{\text{\fontsize{4}{4}\selectfont $i$}}}_s,\; t\in[0,T],\; \PP^{\hat\alpha^{\text{\fontsize{4}{4}\selectfont $i$}}}\text{\rm--a.s.},\\[0.8em]
			\displaystyle\hat\alpha_t^i \coloneqq  \Lambda_t\big(X_{\cdot\wedge t}^i, \cL_{\hat\alpha^i}(X_{\cdot\wedge t}^i),Z_t^i,0\big), \; \frac{\mathrm{d}\PP^{\hat\alpha^{\text{\fontsize{4}{4}\selectfont $i$}}}}{\mathrm{d}\PP} \coloneqq  \cE\bigg(\int_0^Tb_s\big(X_{\cdot\wedge s}^i,\cL_{\hat\alpha^{\text{\fontsize{4}{4}\selectfont $i$}}}(X_{\cdot\wedge s}^i, \hat\alpha_s^i),\hat\alpha_s^i \big)\cdot \mathrm{d}W_s^i\bigg).
		\end{cases}
	\end{equation}
By uniqueness of \rm\Cref{eq:bsde.char.mfg.proof.conv}, $\cL_{\hat\alpha^{\text{\fontsize{4}{4}\selectfont $i$}}}(Z^i) = \cL_{\hat\alpha^{\text{\fontsize{4}{4}\selectfont $j$}}}(Z^j)$ for all $(i,j)\in \{1,\dots,N\}^2$, and by construction, $\cL_{\hat\alpha^{\text{\fontsize{4}{4}\selectfont $i$}}}(X^i) = \cL_{\hat\alpha^{\text{\fontsize{4}{4}\selectfont $j$}}}(X^j)$ for all $(i,j)\in \{1,\dots,N\}^2$.  
We can write \Cref{eq:bsde.char.mfg.proof.conv.ith} under the probability measure $\P^{\hat\alpha}$ where, abusing notations slightly
	\begin{equation}
	\label{eq:true.measure}
		\frac{\diff \P^{\hat\alpha}}{\diff \P}\coloneqq \cE\bigg(\sum_{i=1}^\np\int_0^Tb_s\big(X_{\cdot\wedge s}^i,\cL_{\hat\alpha^i}(X_{\cdot\wedge s}^i, \hat\alpha_s^i),\hat\alpha_s^i \big)\cdot \mathrm{d}W_s^i\bigg),
	\end{equation}
	since by independence of $(W^{1},\dots, W^N)$, this equation remains the same under the measure $\P^{\hat\alpha}$. Besides, the families $\X^N=(X^1,\dots, X^N)$, $\bm{\alpha}^N\coloneqq (\hat\alpha^1,\dots,\hat\alpha^N)$, and $\Z^N\coloneqq  (Z^1,\dots, Z^N)$ are i.i.d. under $\P^{\hat\alpha}$ and for any $i\in\{1,\dots,N\}$
	\[
		\P^{\hat\alpha}\circ X^{-1} = \cL_{\hat\alpha}(X)\equiv \cL_{\hat\alpha^{\text{\fontsize{4}{4}\selectfont $i$}}}(X^i),\; \text{and} \; \P^{\hat\alpha}\circ (X,\hat\alpha)^{-1} = \cL_{\hat\alpha}(X, \hat\alpha)\equiv \cL_{\hat\alpha^{\text{\fontsize{4}{4}\selectfont $i$}}}(X^i,\hat\alpha^i).
	\]

\subsubsection{Step 2: reduction to propagation of chaos} 
Fix some $i\in\{1,\dots,N\}$. Since we have $V^{i,N}_0(\hat{\alpha}^{-i,N}) = Y^{i,N}_0$ where $Y^{i,N}$ solves BSDE \eqref{eq:bsde.n.proof.main.lim}, and  $V^{\Lc_{\smalltext{\hat\alpha}}(X,\hat\alpha)} = Y^i_0$  where $(Y^i,Z^i)$ solves the generalised McKean--Vlasov BSDE \eqref{eq:bsde.char.mfg.proof.conv.ith}, it remains to show that the sequence $(Y^{i,\np}_0)_{N\in\N^\smalltext{\star}}$ converges to $Y_0^i$ at the stated rate.
This will be obtained from the following decomposition
	\begin{align*}
		\big|Y^{i,\np}_0 - Y_0^i\big|^2 &\le 2\big|Y^{i,\np}_0 - \widetilde Y^{i,\np}_0\big|^2 +
		2\big|\widetilde Y^{i,\np}_0 - Y_0^i\big|^2,
	\end{align*}
	where $\big(\widetilde X^{i,N}, \widetilde Y^{i,N}, \widetilde Z^{i,j,N}\big)_{(i,j)\in\{1,\dots,N\}^\smalltext{2}}$ is an auxiliary interacting particle system obtained by solving the (coupled) FBSDE
	\begin{equation}
	\label{eq:auxiliary.1}
		\begin{cases}
		\displaystyle	\tX^{i,N}_t = X^i_0 + \int_0^t b_s\big(\tX^{i,N}_{\cdot\wedge s}, L^N\big(\widetilde\X^N_{\cdot\wedge s}, \widetilde \alpha^{N}_s\big), \widetilde \alpha^{i,N}_s\big)\diff s + \int_0^t\sigma_s\big(\tX^{i,N}_{\cdot\wedge s}\big)\diff W^{\hat\alpha^{\text{\fontsize{4}{4}\selectfont $N$}},i}_s,\; t\in[0,T],\; \P^{\hat\alpha^{\text{\fontsize{4}{4}\selectfont $N$}},N}\text{\rm--a.s.},\\[0.8em]
			\displaystyle\tY^{i,N}_t = g\big(\tX^{i,N}, L^N(\widetilde\X)\big) +\int_t^T f_s\big(\tX^{i,N}_{\cdot\wedge s}, L^N(\widetilde\X^N_{\cdot\wedge s}, \widetilde \alpha^{N}_s), \widetilde \alpha^{i,N}_s\big)\diff s - \int_t^T \sum_{j=1}^N\tZ^{i,j,N}_s\cdot\diff W^{\hat\alpha^{\text{\fontsize{4}{4}\selectfont $N$}},j}_s,\; t\in[0,T],\;\P^{\hat\alpha^{\text{\fontsize{4}{4}\selectfont $N$}},N}\text{--a.s.},\\[0.8em]
		\displaystyle	\widetilde \alpha^{i,N}_t\coloneqq  \Lambda_t\big(\tX^{i,N}_{\cdot\wedge t}, L^N(\widetilde\X^{N}_{\cdot\wedge t}),\tZ^{i,i,N}_t,0\big).
		\end{cases}
	\end{equation}
	This equation admits at least one square integrable solution by \Cref{ass.Lambda.conv}.$(vi)$.
	Observe that here the probability measure $\P^{\hat\alpha^{\text{\fontsize{4}{4}\selectfont $N$}},N}$ and the Brownian motions $(W^{\hat\alpha,1},\dots, W^{\hat\alpha,N}) $ are fixed as given in \Cref{eq:bsde.n.proof.main.lim}.
	In a first step, we show that $|Y^{i,N}_0 - \tY^{i,N}_0|^2$ converges to zero at a given rate, which will require the next lemma.
	\begin{lemma}
	\label{lem:H2.bounded.aux}
		For every $i\in \{1,\dots, N\}$, the processes $\big(\tZ^{i,j,N}\big)_{j\in \{1,\dots,N\}}$ solving {\rm\Cref{eq:auxiliary.1}} satisfy the following bound
		\begin{equation*}
			\E^{\P^{\smalltext{\hat\alpha}^{\tinytext{N}}\smalltext{,}\smalltext{N}}}\bigg[\sum_{j=1}^N\int_0^T\|\tZ^{i,j,N}_s\|^2\diff s \bigg]\le C,\; \forall N\in \mathbb{N}^\star,
		\end{equation*}
	for a constant $C>0$ that does not depend on $N$.
	\end{lemma}
\begin{proof}
	Fix some $i\in\{1,\dots,N\}$. Let $\beta>0$, apply It\^o's formula to $\e^{\beta t}(\tY^{i,N}_t)^2$ and use the growth conditions on $f$ and $g$ in \Cref{ass.Lambda.charac}, the boundedness of $A$, as well as Young's inequality to get
	\begin{align*}
		\e^{\beta t}|\tY^{i,N}_t|^2 &\le 2\ell_g^2\e^{\beta T}\bigg( \|X^i\|_\infty^4 + \frac{1}{N}\sum_{j=1}^N\|X^j\|_\infty^4\bigg) + \int_t^T\e^{\beta s}\bigg( (4\ell_f^2  - \beta)|\tY^{i,N}_s|^2 +  \|\tX^{i,N}_{\cdot\wedge s}\|_\infty^4 + \frac{1}{N}\sum_{j=1}^N\|\tX^{j,N}_{\cdot\wedge s}\|_\infty^4 + C_\Lambda \bigg)\diff s\\
		&\quad-\sum_{j=1}^N\int_t^T\|\tZ^{i,j,N}_s\|^2\diff s - \sum_{j=1}^N\int_t^T2\e^{\beta s}\tY^{i,N}_s\tZ^{i,j,N}_s\cdot\diff W^{\hat\alpha^{\text{\fontsize{4}{4}\selectfont $N$}},j}_s,
\end{align*}
	for some constant $C_\Lambda$ coming from the bound of $\Lambda$ (or $A$).
	Therefore, choosing $\beta$ large enough, we obtain that $\big($the stochastic integral here is an $(\F_N,\P^{\hat\alpha^{\text{\fontsize{4}{4}\selectfont $N$}},N})$-martingale by standard arguments using that for any $j\in\{1,\dots,N\}$, $(\widetilde Y^{i,N},\widetilde Z^{i,j,N})\in\S^2(\R,\F_N,\P^{\hat\alpha^{\text{\fontsize{4}{4}\selectfont $N$}},N})\times \H^2(\R^d,\F_N,\P^{\hat\alpha^{\text{\fontsize{4}{4}\selectfont $N$}},N})\big)$
	\begin{equation*}
		\E^{\P^{\smalltext{\hat\alpha}^\tinytext{N}\smalltext{,}\smalltext{N}}}\bigg[\sum_{j=1}^N\int_0^T\|\tZ^{i,j,N}_s\|^2\diff s \bigg] \le C\bigg(1+\max_{i\in \{1,\dots,N\}}\E^{\P^{\smalltext{\hat\alpha}^{\tinytext{N}}\smalltext{,}\smalltext{N}}}\big[\|\tX^{i,N}\|^4_\infty\big]\bigg),
	\end{equation*}
for some constant $C$ that does not depend on $N$.
But since $b$ and $\sigma$ are bounded, it is direct using Burkholder--Davis--Gundy's inequality and Gronwall's inequality that $\tX^{i,N}$ has all its moments bounded uniformly in $N$. This yields the result.
\end{proof}
We now use the bound from \Cref{lem:H2.bounded.aux} to show that the sequences $(Y^{i,N}_0)_{N\in\N^\smalltext{\star}}$ and $(\tY^{i,N}_0)_{N\in\N^\smalltext{\star}}$ are (asymptotically) close.
\begin{proposition}
\label{prop.seq.close}
	Let the conditions of {\rm\Cref{thm:main.limit}} hold.
	The processes $(Y^{i,N}, Z^{i,j,N})_{(i,j)\in\{1,\dots,N\}^\smalltext{2}}$ and $(\tY^{i,N}, \tZ^{i,j,N})_{(i,j)\in\{1,\dots,N\}^\smalltext{2}}$ solving respectively {\rm\Cref{eq:bsde.n.proof.main.lim}} and {\rm\Cref{eq:auxiliary.1}} satisfy
	\begin{equation}
	\label{eq:seq.close}
		|Y^{i,N}_0 - \tY^{i,N}_0|^2 + \E^{\P^{\smalltext{\hat\alpha}^{\tinytext{N}}\smalltext{,}\smalltext{N}}}\bigg[\sum_{j=1}^N\int_0^T\|Z^{i,j,N}_s - \tZ^{i,j,N}_s\|^2\diff s + \|X^{i} - \tX^{i,N}\|_\infty^2 \bigg] \le C\bigg(\frac1N + NR^2_N\bigg),\;\forall  (N,i)\in \N^\star\times\{1,\dots,N\},
	\end{equation}
	for some constant $C>0$ independent of $N$, and where $(R_N)_{N\in\N^\smalltext{\star}}$ is the sequence introduced in {\rm\Cref{ass.Lambda.conv}}.$(iii)$.
\end{proposition}
\begin{proof}
	We begin by applying It\^o's formula to $\|\delta X^{i,N}\|^2\coloneqq  \|X^{i} - \widetilde X^{i,N}\|^2$ and use the dissipativity of $b$ and Young's inequality to obtain for every positive $\eta$ and $\bar\eps$, where we denote by $C_{\rm BDG}$ the best constant appearing in Burkholder--Davis--Gundy's inequality for exponent $1$ (see \citeauthor*{osekowski2010sharp} \cite[Theorem 1.2]{osekowski2010sharp} for an explicit value for this constant)
	\begin{align*}
		&\E^{\P^{\smalltext{\hat\alpha}^{\tinytext{N}}\smalltext{,}\smalltext{N}}}\Big[\e^{\beta T}\|\delta X^{i,N}\|_\infty^2\Big] \\
		&\le \E^{\P^{\smalltext{\hat\alpha}^{\tinytext{N}}\smalltext{,}\smalltext{N}}}\bigg[\int_0^T\e^{\beta s}\Big((\beta + 2\ell_b^2\eta^{-1} + \ell_\sigma^2-2K_b)\|\delta X^{i,N}_{\cdot\wedge s}\|^2_\infty + \bar\eps\cW_2^2\big(L^N(\X^N_{\cdot\wedge s},\hat\alpha^N_s),L^N(\widetilde \X^N_{\cdot\wedge s},\widetilde \alpha^N_s) \big) + \bar\eps \bar d^2(\hat\alpha^{i,N}_s,\widetilde \alpha^{i,N}_s)\Big)\diff s \bigg]\\
		&\quad + 2C_{\rm BDG}\ell_\sigma\E^{\P^{\smalltext{\hat\alpha}^{\tinytext{N}}\smalltext{,}\smalltext{N}}}\bigg[\bigg(\int_0^T \e^{2\beta s}\|\delta X^{i,N}_{\cdot\wedge s}\|_\infty^4\diff s\bigg)^{1/2}\bigg]\\
		&\leq  \E^{\P^{\smalltext{\hat\alpha}^{\tinytext{N}}\smalltext{,}\smalltext{N}}}\bigg[\int_0^T\e^{\beta s}\Big(\big(\beta+2\ell_b^2\bar\eps^{-1} + \ell_\sigma^2(1 + C_{\rm BDG}^2\eta^{-1})-2K_b\big)\|\delta X^{i,N}_{\cdot\wedge s}\|^2_\infty + \bar\eps\cW_2^2\big(L^N(\X^N_{\cdot\wedge s},\hat\alpha^N_s),L^N(\widetilde \X^N_{\cdot\wedge s},\widetilde \alpha^N_s) \big)\Big)\diff s \bigg] \\
		&\quad +  \bar\eps  \E^{\P^{\smalltext{\hat\alpha}^{\tinytext{N}}\smalltext{,}\smalltext{N}}}\bigg[\int_0^T\e^{\beta s}\bar d^2(\hat\alpha^{i,N}_s,\widetilde \alpha^{i,N}_s)\diff s \bigg]+\eta\E^{\P^{\smalltext{\hat\alpha}^{\tinytext{N}}\smalltext{,}\smalltext{N}}}\Big[\e^{\beta T}\|\delta X^{i,N}\|_\infty^2\Big],\; t\in[0,T].
	\end{align*}
	Hence	
	\begin{align*}
	 	&(1-\eta)\E^{\P^{\smalltext{\hat\alpha}^{\tinytext{N}}\smalltext{,}\smalltext{N}}}\Big[\e^{\beta T}\|\delta X^{i,N}\|_\infty^2\Big] \\
		& \le\E^{\P^{\smalltext{\hat\alpha}^{\tinytext{N}}\smalltext{,}\smalltext{N}}}\bigg[\int_0^T\e^{\beta s}\Big(\beta + 2\ell_b^2\bar\eps^{-1} +  \ell_\sigma^2\big(1 + C_{\rm BDG}^2\eta^{-1}\big)+4\ell^2_\Lambda\bar\eps-2K_b\big)\|\delta X^{i,N}_{\cdot\wedge s}\|^2_\infty  +4\ell^2_\Lambda\bar\eps\big(\|\delta Z^{i,i,N}_s\|^2+ |\aleph^{i,N}_s|^2\big)\Big)\diff s\bigg]\\
		&\quad+\bar\eps\E^{\P^{\smalltext{\hat\alpha}^{\tinytext{N}}\smalltext{,}\smalltext{N}}}\bigg[\int_0^T\e^{\beta s}\bigg(\frac{1 + 12\ell_\Lambda^2}N\sum_{j=1}^N\|\delta X^{j,N}_{\cdot\wedge s}\|^2_\infty +\frac{4\ell^2_\Lambda}N\sum_{j=1}^N\|\delta Z^{j,j,N}_s\|^2 + \frac{4\ell^2_\Lambda}N\sum_{j=1}^N|\aleph^{j,N}_s|^2\bigg)\diff s\bigg],\; t\in[0,T].
	\end{align*}
	Moreover, we can estimate $|\aleph^{i,N}|$ using \Cref{ass.Lambda.conv}.$(iii)$.
	In fact, for any $t\in[0,T]$
	 \begin{align}
	 \label{eq:estim.aleph} 
	  	|\aleph^{i,N}_t|^2 \le 3R_N^2 + 3R_N^2{\color{black}\|X^{i,N}_{\cdot\wedge t}\|^2}+ 6NR_N^2\sum_{j=1}^N\|\delta Z^{i,j,N}_t\|^2 + 6NR_N^2\sum_{j=1}^N\|\tZ^{i,j,N}_t\|^2,\; \text{for all}\; N\in \N^\star,\; i \in \{1,\dots,N\}.
	  \end{align} 
	Therefore, using \Cref{eq:estim.aleph} and \Cref{lem:H2.bounded.aux}, we have
	\begin{align}
	\notag
	 & (1-\eta)	\E^{\P^{\smalltext{\hat\alpha}^{\tinytext{N}}\smalltext{,}\smalltext{N}}}\Big[\e^{\beta T}\|\delta X^{i,N}\|_\infty^2\Big] \\\notag
	 & \le \E^{\P^{\smalltext{\hat\alpha}^{\tinytext{N}}\smalltext{,}\smalltext{N}}}\bigg[\int_0^T\e^{\beta s}\bigg( \beta + 2\ell_b^2\bar\eps^{-1} +  \ell_\sigma^2\big(1 + C_{\rm BDG}^2\eta^{-1}\big)+4\ell^2_\Lambda\bar\eps-2K_b\big)\|\delta X^{i,N}_{\cdot\wedge s}\|^2_\infty +\frac{(1 + 12\ell_\Lambda^2)\bar\eps}N\sum_{j=1}^N\|\delta X^{j,N}_{\cdot\wedge s}\|^2_\infty\bigg)\mathrm{d}s\bigg]\\\notag
	 &\quad +24T\ell_\Lambda^2\bar\eps\mathrm{e}^{\beta T}CNR_N^2+4\ell_\Lambda^2\bar\eps\E^{\P^{\smalltext{\hat\alpha}^{\tinytext{N}}\smalltext{,}\smalltext{N}}}\bigg[\int_0^T\e^{\beta s}\bigg(\|\delta Z^{i,i,N}_s\|^2+\frac1N\sum_{j=1}^N\|\delta Z^{j,j,N}_s\|^2 \bigg)\diff s\bigg]\\\label{eq:x.estim.close}
	  &\quad +12\ell^2_\Lambda R_N^2\bar\eps\E^{\P^{\smalltext{\hat\alpha}^{\tinytext{N}}\smalltext{,}\smalltext{N}}}\bigg[\int_0^T\e^{\beta s}\bigg(2+\| X^{i,N}_{\cdot\wedge s}\|^2_\infty+\frac1N\sum_{j=1}^N\| X^{j,N}_{\cdot\wedge s}\|^2_\infty +2N\sum_{j=1}^N\|\delta Z^{i,j,N}_s\|^2+2\sum_{j=1}^N\sum_{k=1}^N\|\delta Z^{k,j,N}_s\|^2\bigg)\mathrm{d}s\bigg].
	\end{align}
	Using the fact that $X^{i,N}$ has moments of any order under $\P^{\hat\alpha^{\text{\fontsize{4}{4}\selectfont $N$}},N}$ which are uniformly bounded in $N$ by, say $C_X>0$, and summing the previous inequality over $i\in\{1,\dots,N\}$, we deduce that if $K_b$ satisfies
	\begin{equation}
	\label{eq:cond.Kb.proof}
		K_b\ge \frac12\big(\beta + 2\ell_b^2\bar\eps^{-1} +  \ell_\sigma^2\big(1 + C_{\rm BDG}^2\eta^{-1}\big)+\big(1+16\ell^2_\Lambda\big)\bar\eps\big),
	\end{equation}
	then 
	\begin{align}
	\notag
	  	(1-\eta)\sum_{i=1}^N\E^{\P^{\smalltext{\hat\alpha}^{\tinytext{N}}\smalltext{,}\smalltext{N}}}\Big[\e^{\beta t}\|\delta X^{i,N}_{\cdot\wedge t}\|_\infty^2\Big] & \le 24T\ell_\Lambda^2\bar\eps\mathrm{e}^{\beta T}CNR_N^2(1+C_X+N)+8\ell^2_\Lambda \bar\eps \E^{\P^{\smalltext{\hat\alpha}^{\tinytext{N}}\smalltext{,}\smalltext{N}}}\bigg[\int_0^t \sum_{j=1}^N\e^{\beta s}\|\delta Z^{j,j,N}_s\|^2  \diff s\bigg] \\
	  	\notag
		&\quad+24\ell_\Lambda^2 NR_N^2\bar\eps\E^{\P^{\smalltext{\hat\alpha}^{\tinytext{N}}\smalltext{,}\smalltext{N}}}\bigg[\int_0^t\sum_{j=1}^N\sum_{k=1}^N\e^{\beta s}\|\delta Z^{j,k,N}_s\|^2  \diff s\bigg]\\
		\label{eq:estim.sum.x}
		&\leq 24T\ell_\Lambda^2\bar\eps\mathrm{e}^{\beta T}CNR_N^2(1+C_X+N)+8\ell_\Lambda^2\bar\eps\big(1+3NR_N^2\big)\E^{\P^{\smalltext{\hat\alpha}^{\tinytext{N}}\smalltext{,}\smalltext{N}}}\bigg[\int_0^t\sum_{j=1}^N\sum_{k=1}^N\e^{\beta s}\|\delta Z^{j,k,N}_s\|^2  \diff s\bigg].
	\end{align}

	Now, we move on to estimate $\delta Y^{i,N} = \tilde Y^{i,N} - Y^{i,N}$.
	Applying It\^o's formula to $\e^{\beta t}(\delta Y^{i,N}_t)^2$ and using Young's inequality and Lipschitz-continuity of $f$ and $\Lambda$, we have for every $\eps>0$
	  \begin{align*}
	  	\e^{\beta t}(\delta Y^{i,N}_t)^2 +\sum_{j=1}^N\int_t^T\e^{\beta s}\|\delta Z^{i,j,N}_s\|^2&\le 2\ell_g^2\Big(\|\delta X^{i,N}\|_\infty^2 + \cW_2^2(L^N(\X^N), L^N(\widetilde \X^N)) \Big)\\
	  	&\quad + \int_t^T\e^{\beta s}\bigg(\big(3\ell_f^2\eps^{-1} - \beta\big)|\delta Y^{i,N}_s|^2 + \eps(1+4 \ell^2_\Lambda)\|\delta X^{i,N}_{\cdot\wedge s}\|_\infty^2 +4\eps\ell_\Lambda^2\big(\|\delta Z^{i,i,N}_s\|^2 + |\aleph^{i,N}_s|^2\big)\bigg)\diff s\\
	  	&\quad +\frac{2\eps(1+4\ell_\Lambda^2)}N\sum_{j=1}^N \int_t^T\e^{\beta s}\|\delta X^{j,N}_{\cdot\wedge s}\|_\infty^2\mathrm{d}s + \frac{4\ell_\Lambda^2\eps}{N}\sum_{j=1}^N\int_t^T\e^{\beta s}\big(\|\delta Z^{j,j,N}_s\|^2 + |\aleph^{j,N}_s|^2\big)\diff s\\
	  	&\quad  - \sum_{j=1}^N\int_t^T2\e^{\beta s}\delta Y^{j,N}_s\delta Z^{i,j,N}_s\cdot \diff W^{\hat\alpha^{\text{\fontsize{4}{4}\selectfont $N$}},j}_s.
	  \end{align*}
	Next we take conditional expectation on both sides (the stochastic integral disappears by the same arguments we have used before) and profit from \Cref{eq:estim.aleph} to obtain that for $\beta\geq 3\ell_f^2\eps^{-1}$
	\begin{align}
	\notag
		&\e^{\beta t}|\delta Y^{i,N}_t|^2 + \E^{\P^{\smalltext{\hat\alpha}^{\tinytext{N}}\smalltext{,}\smalltext{N}}}\bigg[\sum_{j=1}^N\int_t^T\e^{\beta s}\|\delta Z^{i,j,N}_s\|^2\diff s\bigg| \cF_{N,t} \bigg]\le 2\ell_g^2\E^{\P^{\smalltext{\hat\alpha}^{\tinytext{N}}\smalltext{,}\smalltext{N}}}\bigg[\|\delta X^{i,N}\|^2_\infty + \frac1N\sum_{j=1}^N\|\delta X^{j,N}\|^2_\infty \bigg| \cF_{N,t} \bigg]\\\notag
		&\quad + C \eps R_N^2 +\eps\E^{\P^{\smalltext{\hat\alpha}^{\tinytext{N}}\smalltext{,}\smalltext{N}}}\bigg[ \int_t^T\e^{\beta s}\bigg((1+4\ell_\Lambda^2) \|\delta X^{i,N}_{\cdot\wedge s}\|_\infty^2 +4\ell_\Lambda^2\|\delta Z^{i,i,N}_s\|^2+\frac{2(1+4\ell_\Lambda^2) }N\sum_{j=1}^N\|\delta X^{j,N}_{\cdot\wedge s}\|_\infty^2+\frac{4\ell_\Lambda^2}{N}\sum_{j=1}^N\|\delta Z^{j,j,N}_s\|^2 \bigg)\diff s\\
	  	& \quad +6\eps R_N^2\int_t^T\mathrm{e}^{\beta s}\bigg(\|X^{i,N}_{\cdot\wedge s}\|^2+\frac1N\sum_{j=1}^N\|X^{j,N}_{\cdot\wedge s}\|^2+N\sum_{j=1}^N\big(\|\delta Z^{i,j,N}_s\|^2+\|\widetilde Z^{i,j,N}_s\|^2\big)+\sum_{k=1}^N\sum_{j=1}^N\big(\|\delta Z^{k,j,N}_s\|^2+\|\widetilde Z^{k,j,N}_s\|^2\big)\bigg)\mathrm{d}s  \bigg| \cF_{N,t}\bigg]  \label{eq:estim.close.one}.
	\end{align}
	Let us sum on both sides over $i\in\{1,\dots,N\}$. We get that for some $C$ independent of $N$, we have
	\begin{align*}
		&\big( 1 - \eps C_1(1+NR_N^2) \big)\E^{\P^{\smalltext{\hat\alpha}^{\tinytext{N}}\smalltext{,}\smalltext{N}}}\bigg[\sum_{i=1}^N\sum_{j=1}^N\int_t^T\e^{\beta s}\|\delta Z^{i,j,N}_s\|^2\diff s\bigg| \cF_{N,t} \bigg]\le 4\ell_g^2\E^{\P^{\smalltext{\hat\alpha}^{\tinytext{N}}\smalltext{,}\smalltext{N}}}\bigg[\sum_{j=1}^N\|\delta X^{j,N}\|^2_\infty \bigg| \cF_{N,t} \bigg]\\\notag
		&\quad + C\eps \E^{\P^{\smalltext{\hat\alpha}^{\tinytext{N}}\smalltext{,}\smalltext{N}}}\bigg[\sum_{i=1}^N\int_t^T\e^{\beta s}\|\delta X^{i,N}_s\|^2\diff s+NR_N^2\int_t^T\mathrm{e}^{\beta s}\sum_{k=1}^N\sum_{j=1}^N\|\widetilde Z^{k,j,N}_s\|^2\mathrm{d}s +R_N^2\sum_{i=1}^N\int_t^T\mathrm{e}^{\beta s}\|X^{i,N}_{\cdot\wedge s}\|^2\mathrm{d}s\bigg|\Fc_{N,t}\bigg]+ C\eps NR_N^2
	\end{align*}
	with $C_1 \coloneqq6(\ell_\Lambda^2 + NR_N^2).$ In light of \Cref{eq:estim.sum.x} and \Cref{lem:H2.bounded.aux} we thus have for some $C>0$ not depending on neither $N$, $\bar\eps$ nor $\eps$
	\begin{align*}
	\notag
		&\big( 1 - (\eps+\bar\eps) C_1(1+NR_N^2) \big)\E^{\P^{\smalltext{\hat\alpha}^{\tinytext{N}}\smalltext{,}\smalltext{N}}}\bigg[\sum_{i=1}^N\sum_{j=1}^N\int_0^T\e^{\beta s}\|\delta Z^{i,j,N}_s\|^2\diff s\bigg| \cF_{N,t} \bigg]\\
		&\le \eps C\E^{\P^{\smalltext{\hat\alpha}^{\tinytext{N}}\smalltext{,}\smalltext{N}}}\bigg[\int_0^T\bigg( 
		NR_N^2\sum_{i=1}^N\sum_{j=1}^N\e^{\beta s}\|\delta Z^{i,j,N}_s\|^2+R_N^2\sum_{i=1}^N\mathrm{e}^{\beta s}\|X^{i,N}_{\cdot\wedge s}\|^2\bigg)\diff s\bigg|\Fc_{N,t}\bigg] +  C N^2R_N^2 
		+ C NR_N^2.
	\end{align*}
	Since the sequence $(N^2R^2_N)_{N\in\N^\smalltext{\star}}$ is bounded, this implies that
	\begin{align}
	\label{eq:estim.sumZ.aux.proof.close}
		\big( 1 - (\eps+\bar\eps) C_1(1+NR_N^2) \big)&\E^{\P^{\smalltext{\hat\alpha}^{\tinytext{N}}\smalltext{,}\smalltext{N}}}\bigg[\sum_{i=1}^N\sum_{j=1}^N\int_0^T\e^{\beta s}\|\delta Z^{i,j,N}_s\|^2\diff s\bigg| \cF_{N,t} \bigg]
		\le 
		 C +\eps CR_N^2\E^{\P^{\smalltext{\hat\alpha}^{\tinytext{N}}\smalltext{,}\smalltext{N}}}\bigg[\int_0^T\sum_{i=1}^N\mathrm{e}^{\beta s}\|X^{i,N}_{\cdot\wedge s}\|^2\diff s\bigg|\Fc_{N,t}\bigg] .
	\end{align}
	We thus obtain after choosing $\bar\eps$ and $\eps$ small enough (independently of $N$)
	\begin{align}
	\label{eq:bound.double.sum}
		\E^{\P^{\smalltext{\hat\alpha}^{\tinytext{N}}\smalltext{,}\smalltext{N}}}\bigg[\sum_{i=1}^N\sum_{j=1}^N\int_0^T\e^{\beta s}\|\delta Z^{i,j,N}_s\|^2\diff s\bigg| \cF_{N,t} \bigg]
		\le 
		 C\bigg(1+R_N^2\E^{\P^{\smalltext{\hat\alpha}^{\tinytext{N}}\smalltext{,}\smalltext{N}}}\bigg[\int_0^T\sum_{i=1}^N\mathrm{e}^{\beta s}\|X^{i,N}_{\cdot\wedge s}\|^2\diff s\bigg|\Fc_{N,t}\bigg]\bigg),\; \text{for all}\; N\in \N^\star,
	\end{align}
	and for a constant $C$ that does not depend on $N$. Taking expectations and using the fact that $X^{i,N}$ has moments of any order under $\P^{\hat\alpha^{\text{\fontsize{4}{4}\selectfont $N$}},N}$ which are uniformly bounded in $N$, we deduce that
\begin{align}
	\label{eq:bound.double.sum2}
		\E^{\P^{\smalltext{\hat\alpha}^{\tinytext{N}}\smalltext{,}\smalltext{N}}}\bigg[\sum_{i=1}^N\sum_{j=1}^N\int_0^T\e^{\beta s}\|\delta Z^{i,j,N}_s\|^2\diff s\bigg]
		\le 
		 C(1+NR_N^2),\; \text{for all}\; N\in \N^\star,
	\end{align}

	Coming back to \Cref{eq:estim.close.one} and using \Cref{eq:bound.double.sum2}, Lemma \ref{lem:H2.bounded.aux}, and \Cref{eq:estim.sum.x} this allows to continue the estimation
	\begin{align}
	\notag
		&|\delta Y^{i,N}_0|^2 + \big(1-(\eps+\bar\eps)C(1+NR_N^2)\big)\E^{\P^{\smalltext{\hat\alpha}^{\tinytext{N}}\smalltext{,}\smalltext{N}}}\bigg[\sum_{j=1}^N\int_0^T\e^{\beta s}\|\delta Z^{i,j,N}_s\|^2\diff s \bigg]\\
		\label{eq:last.estim.close.aux.proof}
		& \le \eps\E^{\P^{\smalltext{\hat\alpha}^{\tinytext{N}}\smalltext{,}\smalltext{N}}}\bigg[ \int_0^T\e^{\beta s}\bigg( (1+4\ell_\Lambda^2)\|\delta X^{i,N}_{\cdot\wedge s}\|_\infty^2 +\frac{2(1+4\ell_\Lambda^2)}N\sum_{j=1}^N\|\delta X^{j,N}_{\cdot\wedge s}\|_\infty^2+\frac{C}{N}\bigg)\diff s\bigg] +C\eps R_N^2(1+N)  +C R_N^2\\
		\notag
		&\leq CNR_N^2 +CR_N^2+\frac CN+C\eps\E^{\P^{\smalltext{\hat\alpha}^{\tinytext{N}}\smalltext{,}\smalltext{N}}}\bigg[\int_0^T\e^{\beta s}\|\delta Z^{i,i,N}_s\|^2\diff s \bigg],
	\end{align}	
	which yields the bound on $\delta Y^{i,N}$ and $\delta Z^{i,j}$ in \Cref{eq:seq.close}. The bound of $\delta X^{i,N}$ thus follows from \Cref{eq:estim.sum.x}, \Cref{eq:x.estim.close}, \Cref{eq:cond.Kb.proof}, and \Cref{eq:bound.double.sum2}.
\end{proof}

\subsubsection{Step 3: propagation of chaos}
The final step of the proof is a propagation of chaos result for the (backward) particle system $(\tX^{i,N}, \tY^{i,N}, \tZ^{i,j,N})_{(i,j)\in\{1,\dots,N\}^\smalltext{2}}$ given by \Cref{eq:auxiliary.1}.
\begin{proposition}[Propagation of chaos]
\label{prop:prob.chaos}
	Let the conditions of {\rm\Cref{thm:main.limit}} be satisfied.
	Recall the probability measure $\P^{\hat\alpha}$ given in {\rm\Cref{eq:true.measure}.}
	It holds
	\begin{equation*}
		\big|\tY^{i,N}_0 - Y^i_0 \big|^2 \le C\sup_{t\in[0,T] }\Big\{\E^{\P^{\smalltext{\hat\alpha}}}\Big[\cW_2^2\big(L^N(\hat{\bm{\alpha}}_t^N), \cL_{\hat\alpha}(\hat\alpha_t)\big)\Big] + \E^{\P^{\smalltext{\hat\alpha}}}\Big[\cW_2^2\big(L^N(\X^N_{\cdot\wedge t}),\cL_{\hat\alpha}(X_{\cdot\wedge t}) \big) \Big]\Big\},\;\text{\rm for all}\; N \in \N^\star,
	\end{equation*}
	where $C>0$ is a constant that does not depend on $N$. In particular, if the measure argument in $b$ and $f$ is state dependent, \emph{i.e.} $L^N(\X^N_{\cdot\wedge t}, \alpha_t)$ is replaced by $L^N(\X^N_t, \alpha_t)$, and $A\subseteq \R^k$ for some $k\in\N^\star$, then we have
	\begin{equation*}
		\big|\tY^{i,N}_0 - Y^i_0 \big|^2 \le C\big(r_{\text{\fontsize{5}{5}\selectfont$\np$}, k, q} + r_{\text{\fontsize{5}{5}\selectfont$\np$},m, q}\big),\; \text{\rm for all}\; N \in \N^\star,\; q>2.
	\end{equation*}
\end{proposition}
\begin{proof}
	The proof of this result is similar to that of \Cref{prop.seq.close}.
	We will present the argument for clarity.
	Consider the coupled FBSDE system
	\begin{equation}
		\begin{cases}
			\displaystyle \barX^{i,N}_t = X^i_0 + \int_0^t b_s\big(\barX^{i,N}_{\cdot\wedge s}, L^N(\overline\X^N_{\cdot\wedge s}, \overline \alpha^{N}_s), \overline \alpha^{i,N}_s\big)\diff s + \int_0^t\sigma_s(\barX^{i,N}_{\cdot\wedge s})\diff W^{\hat\alpha,i}_s,\; t\in[0,T],\; \P^{\hat\alpha}\text{--a.s.},\\[0.8em]
		\displaystyle	\barY^{i,N}_t = g(\barX^{i,N},L^N(\overline\X^N)) + \int_t^T f_s\big(\barX^{i,N}_{\cdot\wedge s}, L^N(\overline\X^N_{\cdot\wedge s}, \overline \alpha^{N}_s), \overline \alpha^{i,N}_s\big)\diff s - \int_t^T \sum_{j=1}^N\barZ^{i,j,N}_s\cdot \diff W^{\hat\alpha,j}_s,\; t\in[0,T],\; \P^{\hat\alpha}\text{--a.s.},\\[0.8em]
		\displaystyle	\overline \alpha^{i,N}_t\coloneqq  \Lambda_t\big(\barX^{i,N}_{\cdot\wedge t}, L^N(\overline\X^{N}_{\cdot\wedge t}),\barZ^{i,i,N}_t,0\big),\; \overline\alpha^N\coloneqq \big(\overline\alpha^{i,N}\big)_{i\in \{1,\dots,N\}}.
		\end{cases}
	\end{equation}
	This equation admits a square integrable solution by assumption.
	Observe that the probability $\P^{\hat\alpha}$ and the Brownian motions $(W^{\hat\alpha,1},\dots, W^{\hat\alpha,N})$ are fixed from \Cref{eq:true.measure}.
	Since $(\tX^{i,N}, \tY^{i,N}, \tZ^{i,j,N})$ satisfies the same equation on $(\Omega, \cF, \P^{\hat\alpha^{\text{\fontsize{4}{4}\selectfont $N$}},N})$ with the Brownian motions $(W^{\hat\alpha^{\text{\fontsize{4}{4}\selectfont $N$}}, 1},\dots, W^{\hat\alpha^{\text{\fontsize{4}{4}\selectfont $N$}}, N})$, it follows by Girsanov's theorem that 
	\[
	\P^{\hat\alpha^{\text{\fontsize{4}{4}\selectfont $N$}},N}\circ\big(\tX^{i,N}, \tY^{i,N}, \tZ^{i,j,N}\big)^{-1} = \P^{\hat\alpha }\circ \big(\barX^{i,N}, \barY^{i,N}, \barZ^{i,j,N}\big)^{-1}.
	\]
	Thus, since $\P^{\hat\alpha^{\text{\fontsize{4}{4}\selectfont $N$}},N}$ and $\P^{\hat\alpha}$ agree on $\cF_{N,0}$, it follows that $\tY^{i,N}_0 = \barY^{i,N}_0$ for all $i\in \{1,\dots,N\}$ and all $N\in \N^\star$.
	Hence, it suffices to estimate the rate of convergence of $(\barY^{i,N})_{N\in\N^\smalltext{\star}}$ to $Y^{i}_0$.
	Let us denote $\delta X^{i,N}\coloneqq  \barX^{i,N} - X^i $, $\delta Y^{i,N} \coloneqq  \barY^{i,N} - Y^i$ and $\delta Z^{i,j,N}\coloneqq  \barZ^{i,j,N} - Z^i\mathbf{1}_{\{i=j\}}$.
	The proof of the convergence of $\delta X^{i,N}$, $\delta Y^{i,N}$, and $\delta Z^{i,j,N}$ is essentially the same as the proof of \Cref{prop.seq.close}.
	Therefore, we give only the main steps of the proof to avoid repetitions.
	By exactly the same arguments used to obtain Equation \ref{eq:x.estim.close}, we have
	\begin{align}
		\notag (1-\eta)	\E^{\P^{\smalltext{\hat\alpha}}}\Big[\e^{\beta T}\|\delta X^{i,N}\|_\infty^2\Big]  & \le \E^{\P^{\hat\alpha}}\bigg[\int_0^T\e^{\beta s}\bigg( \beta + 2\ell_b^2\bar\eps^{-1} +  \ell_\sigma^2\big(1 + C_{\rm BDG}^2\eta^{-1}\big)+4\ell^2_\Lambda\bar\eps-2K_b\big)\|\delta X^{i,N}_{\cdot\wedge s}\|^2_\infty\mathrm{d}s\bigg]\\
	 &\quad +\bar\eps\E^{\P^{\smalltext{\hat\alpha}}}\bigg[\int_0^T\e^{\beta s}\frac1N\sum_{j=1}^N\Big\{(2 + 4\ell_\Lambda^2)\|\delta X^{j,N}_{\cdot\wedge s}\|^2_\infty + 4\ell_\Lambda^2\e^{\beta s}\|\delta Z^{j,j,N}_s\|^2 \Big\} + 4\ell_\Lambda^2\|\delta Z^{i,i,N}_s\|^2\diff s\bigg] + \bar\eps E^N
	\end{align}
	for every positive $\bar\eps$ and $\eta$, with $E^N\coloneqq  4\ell_\Lambda^2\E^{\P^{\smalltext{\hat\alpha}}}\big[\int_0^T\e^{\beta t}\big(\cW_2^2\big(L^N(\hat{\bm{\alpha}}_t^N), \cL_{\hat\alpha}(\hat\alpha_t)\big) + \cW_2^2\big(L^N(\X^N_{\cdot\wedge t}),\cL_{\hat\alpha}(X_{\cdot\wedge t}) \big) \big)\diff t\big].$ Thus, if $K_b$ satisfies \Cref{eq:cond.Kb.proof}, it then holds
	\begin{align}
	\label{eq:estim.X.chaos}
		\E^{\P^{\smalltext{\hat\alpha}}}\Big[\e^{\beta T}\|\delta X^{i,N}\|_\infty^2 \Big] &\le \frac{\bar\eps(2 + 4\ell_\Lambda^2)}{1-\eta}\E^{\P^{\smalltext{\hat\alpha}}}\bigg[\int_0^T\e^{\beta s} \frac1N\sum_{j=1}^N\big\{\|\delta X^{j,N}_{\cdot\wedge s}\|^2_\infty + \|\delta Z^{j,j,N}_s\|^2\big\} + \e^{\beta s}\|\delta Z^{i,i,N}_s\|^2 \diff s \bigg] + E^N,\\
	\label{eq:estim.sum.x.chaos}
		\frac1N\sum_{i=1}^N\E^{\P^{\smalltext{\hat\alpha}}}\Big[\e^{\beta T}\|\delta X^{i,N}\|_\infty^2 \Big] &\le \frac{\bar\eps(4\ell_\Lambda^2 +1)}{1 +\bar\eps(2+ \ell_\Lambda^2) - \eta}\E^{\P^{\smalltext{\hat\alpha}}}\bigg[\int_0^T\e^{\beta s}\frac1N\sum_{i=1}^N\|\delta Z^{i,i,N}_s\|^2  \diff s \bigg] + E^N.
	\end{align}
	Then, using similar arguments to the ones leading to \eqref{eq:estim.close.one}, we have
	\begin{align*}
		\notag
		&\e^{\beta t}|\delta Y^{i,N}_t|^2 + \E^{\P^{\smalltext{\hat\alpha}}}\bigg[\sum_{j=1}^N\int_t^T\e^{\beta s}\|\delta Z^{i,j,N}_s\|^2\diff s\bigg| \cF_{N,t} \bigg]\\\notag
		& \le 2\e^{\beta T}\ell_g^2\E^{\P^{\smalltext{\hat\alpha}}}\Big[\|\delta X^{i,N}\|_\infty^2 + \frac1N\sum_{j=1}^N\|\delta X^{j,N}\|^2_\infty + \cW_2^2(L^N(\X^N), \cL_{\hat\alpha}(X))  \Big| \cF_{N,t} \Big] + \E^{\P^{\smalltext{\hat\alpha}}}\bigg[ \int_t^T\e^{\beta s}\Big(\ell^2_f(1+4\ell_\Lambda^2)\eps^{-1} - \beta\Big) \|\delta Y^{i,N}_{s}\|^2 \diff s\bigg]\\
		&\quad +\eps\E^{\P^{\smalltext{\hat\alpha}}}\bigg[\int_t^T\e^{\beta s}\frac1N\sum_{j=1}^N\Big\{(1 + 12\ell_\Lambda^2)\|\delta X^{j,N}_{\cdot\wedge s}\|^2 + 4\ell^2_\Lambda\|\delta Z^{j,j,N}_s\|^2\Big\} +\e^{\beta s} (\|\delta X^{i,N}_{\cdot\wedge s}\|_\infty^2+ \|\delta Z^{i,i,N}_s\|^2 )\diff s   \bigg| \cF_{N,t}\bigg] + \eps E_N.
	\end{align*}
	Thus, for $\beta$ large enough, we have
		\begin{align*}
		\notag
		 \E^{\P^{\smalltext{\hat\alpha}}}\bigg[\e^{\beta t}|\delta Y^{i,N}_t|^2 + (1-\eps)\sum_{j=1}^N\int_t^T\e^{\beta s}\|\delta Z^{i,j,N}_s\|^2\diff s \bigg]
		& \le 2\e^{\beta T}\ell_g^2\E^{\P^{\smalltext{\hat\alpha}}}\bigg[\|\delta X^{i,N}\|_\infty^2 + \frac1N\sum_{j=1}^N\|\delta X^{j,N}\|^2_\infty  \bigg]+E^{1,N} \\
		&\quad +\eps(1 + 12\ell_\Lambda^2)\E^{\P^{\smalltext{\hat\alpha}}}\bigg[\int_t^T\bigg(\frac1N\sum_{j=1}^N\big(\|\delta X^{j,N}_{\cdot\wedge s}\|^2 + \|\delta Z^{j,j,N}_s\|^2\big) + \|\delta X^{i,N}_{\cdot\wedge s}\|_\infty^2\bigg) \diff s  \bigg],
	\end{align*}
	with $E^{1,N}\coloneqq   2\mathrm{e}^{\beta T}\ell_g^2\E^{\P^{\smalltext{\hat\alpha}}}\big[\cW_2^2(L^N(\X^N), \cL_{\hat\alpha}(X))\big] + 4\ell_\Lambda^2\eps\E^{\P^{\smalltext{\hat\alpha}}}\big[\int_0^T\e^{\beta t}\big(\cW_2^2\big(L^N(\hat{\bm{\alpha}}_t^N), \cL_{\hat\alpha}(\hat\alpha_t)\big) + \cW_2^2\big(L^N(\X^N_{\cdot\wedge t}),\cL_{\hat\alpha}(X_{\cdot\wedge t}) \big) \big)\diff t\big].$ Hence, using \Cref{eq:estim.X.chaos,eq:estim.sum.x.chaos}, we can find two positive constants $C_1$ and $C_2$ such that
	\begin{align}
		& \E^{\P^{\smalltext{\hat\alpha}}}\bigg[\e^{\beta t}|\delta Y^{i,N}_t|^2 + (1-\eps)\sum_{j=1}^N\int_t^T\e^{\beta s}\|\delta Z^{i,j,N}_s\|^2\diff s \bigg]
		\label{eq:prop.chaos.estim.yz} \le \bar\eps C_1\E^{\P^{\smalltext{\hat\alpha}}}\bigg[\int_t^T\frac1N\sum_{j=1}^N \|\delta Z^{j,j,N}_s\|^2 \diff s  \bigg] +  C_2(E^N + E^{1,N}).
	\end{align}
	Summing up on both sides, we obtain
	\begin{align}
		\notag
		& \E^{\P^{\smalltext{\hat\alpha}}}\bigg[\e^{\beta t}\sum_{i=1}^N|\delta Y^{i,N}_t|^2 + (1-\eps-\bar\eps C_1)\sum_{i=1}^N\sum_{j=1}^N\int_t^T\e^{\beta s}\|\delta Z^{i,j,N}_s\|^2\diff s \bigg] \le  C_2N(E^N + E^{1,N}).
	\end{align}
	Therefore, choosing $\bar\eps$ and $\eps$ small enough, it holds that
	\begin{equation*}
		\E^{\P^{\smalltext{\hat\alpha}}}\bigg[\sum_{i=1}^N\sum_{j=1}^N\int_t^T\e^{\beta s}\|\delta Z^{i,j,N}_s\|^2\diff s \bigg]
		 \le  C_3N(E^N + E^{1,N}),
	\end{equation*}
	for some constant $C_3>0$.
	Using this back in \Cref{eq:prop.chaos.estim.yz} allows to obtain, for $\eps<1$,
	\begin{align*}
		\E^{\P^{\smalltext{\hat\alpha}}}\bigg[\e^{\beta t}|\delta Y^{i,N}_t|^2 + \sum_{j=1}^N\int_t^T\e^{\beta s}\|\delta Z^{i,j,N}_s\|^2\diff s \bigg]
		 \le  C_4(E^N + E^{1,N}),
	\end{align*}
	for some constant $C_4>0$.
	This concludes the proof since $E^N + E^{1,N}$ is dominated by 
	\begin{equation*}
	\sup_{t\in[0,T] }\Big\{\E^{\P^{\smalltext{\hat\alpha}}}\Big[\cW_2^2\big(L^N(\hat{\bm{\alpha}}_t^N), \cL_{\hat\alpha}(\hat\alpha_t)\big)\Big] + \E^{\P^{\smalltext{\hat\alpha}}}\Big[\cW_2^2\big(L^N(\X^N_{\cdot\wedge t}),\cL_{\hat\alpha}(X_{\cdot\wedge t}) \big) \Big]\Big\}.
	\end{equation*}
\end{proof}

	\subsubsection{Step 4: Convergence of the Nash equilibria}
	Let us now prove \Cref{eq:conv.law.hatalpha}.
	By the respective representations \eqref{eq:bsde.char.mfg.proof.conv.ith} and \eqref{eq:rep.Nash.proof} of $\hat\alpha^{i,N}_t$ and $\hat\alpha^i_t$, we have
	\begin{align*}
    		&\int_0^T\cW_2^2\big(\P^{\hat\alpha^{\text{\fontsize{4}{4}\selectfont $N$}},N}\circ (\hat\alpha^{i,N}_t)^{-1}, \cL_{\hat\alpha}(\hat\alpha_t^i) \big)\diff t \\
		& = \int_0^T\cW_2^2\Big(\P^{\hat\alpha^{\text{\fontsize{4}{4}\selectfont $N$}},N}\circ\Lambda_t\big(X^{i}_{\cdot\wedge t}, Z^{i,i,N}_t, L^N(\X^N_{\cdot\wedge t}), \aleph^{i,N}_t \big)^{-1}, \cL_{\hat\alpha}\big(\Lambda_t(X^{i}_{\cdot\wedge t}, Z^{i}_t, \cL_{\hat\alpha}(X_{\cdot\wedge t}), 0) \big) \Big)\diff t	\\
		&\le \int_0^T\cW_2^2\Big(\P^{\hat\alpha^{\text{\fontsize{4}{4}\selectfont $N$}},N}\circ\Lambda_t\big(X^{i}_{\cdot\wedge t}, Z^{i,i,N}_t, L^N(\X^N_{\cdot\wedge t}), \aleph^{i,N}_t \big)^{-1}, \P^{\hat\alpha^{\text{\fontsize{4}{4}\selectfont $N$}},N}\circ\Lambda_t\big(\tX^{i}_{\cdot\wedge t}, \tZ^{i,i,N}_t, L^N(\widetilde{\X}_{\cdot\wedge t}), 0 \big)^{-1} \Big)\diff t	\\
		&\quad + \int_0^T\cW_2^2\Big(\P^{\hat\alpha^{\text{\fontsize{4}{4}\selectfont $N$}},N}\circ\Lambda_t\big(\tX^{i}_{\cdot\wedge t}, \tZ^{i,i,N}_t, L^N(\widetilde{\X}_{\cdot\wedge t}), 0 \big)^{-1}, \cL_{\hat\alpha}\big(\Lambda_t(X^{i}_{\cdot\wedge t}, Z^{i}_t, \cL_{\hat\alpha}(X_{\cdot\wedge t}), 0) \big) \Big)\diff t\\
		&\le \ell_\Lambda \E^{\P^{\smalltext{\hat\alpha}^{\tinytext{N}}\smalltext{,}\smalltext{N}}}\bigg[\int_0^T\bigg(\|X^{i}_{\cdot\wedge t} - \tX^{i,N}_{\cdot\wedge t}\|^2_\infty+ \|Z^{i,i,N}_t - \tZ^{i,i,N}_t\|^2+\frac1N\sum_{j=1}^N\|X^j_{\cdot\wedge t} - \tX^{j,N}_{\cdot\wedge t}\|^2_\infty+\|\aleph^{i,N}_t\|^2\bigg)\diff t\bigg]	\\
		&\quad + \int_0^T\cW_2^2\Big(\cL_{\hat\alpha}\big(\Lambda_t(\barX^{i,N}_{\cdot\wedge t}, \barZ^{i,i,N}_t, L^N(\overline{\X}_{\cdot\wedge t}), 0) \big), \cL_{\hat\alpha}\big(\Lambda_t(X^{i}_{\cdot\wedge t}, Z^{i}_t, \cL_{\hat\alpha}(X_{\cdot\wedge t}), 0) \big) \Big)\diff t	\\
		&\le \ell_\Lambda \E^{\P^{\smalltext{\hat\alpha}^{\tinytext{N}}\smalltext{,}\smalltext{N}}}\bigg[\int_0^T\bigg(\|X^{i}_{\cdot\wedge t} - \tX^{i,N}_{\cdot\wedge t}\|^2_\infty+ \|Z^{i,i,N}_t - \tZ^{i,i,N}_t\|^2+\frac1N\sum_{j=1}^N\|X^j_{\cdot\wedge t} - \tX^{j,N}_{\cdot\wedge t}\|^2_\infty+\|\aleph^{i,N}_t\|^2\bigg)\diff t\bigg]	\\
		&\quad + \ell_\Lambda \E^{\P^{\smalltext{\hat\alpha}}}\bigg[\int_0^T\bigg(\|X^{i}_{\cdot\wedge t} - \barX^{i,N}_{\cdot\wedge t}\|^2_\infty+ \|Z^{i,i,N}_t - \barZ^{i,i,N}_t\|^2+\frac1N\sum_{j=1}^N\|X^j_{\cdot\wedge t} - \barX^{j,N}_{\cdot\wedge t}\|^2_\infty+\cW_2^2(L^N(\X_{\cdot\wedge t}), \cL_{\hat\alpha}(X_{\cdot\wedge t}))\bigg)\diff t\bigg],
	\end{align*}
	where the first inequality is the triangular inequality, the second one follows by Lipschitz-continuity of $\Lambda$ and the fact that $\P^{\hat\alpha^{\text{\fontsize{4}{4}\selectfont $N$}},N}\circ(\tX^{i,N}, \tY^{i,N}, \tZ^{i,j,N})^{-1} = \P^{\hat\alpha }\circ (\barX^{i,N}, \barY^{i,N}, \barZ^{i,j,N})^{-1}$, and the third one uses again Lipschitz-continuity of $\Lambda$ and the triangular inequality.
	Now, by \Cref{prop.seq.close}, we have
	\begin{align*}
		\int_0^T\cW_2^2\big(\P^{\hat\alpha^{\text{\fontsize{4}{4}\selectfont $N$}},N}\circ({\hat\alpha}^{i,N}_t)^{-1}, \cL_{\hat\alpha}(\hat\alpha_t^i) \big)\diff t \le C\bigg(\frac1N  + NR^2_N + \gamma^N\bigg) + C\E^{\P^{\smalltext{\hat\alpha}^\tinytext{N}}}\bigg[\int_0^T\|\aleph^{i,N}_t\|^2\diff t \bigg].
	\end{align*}
	Finally, using \Cref{eq:estim.aleph} and \Cref{lem:H2.bounded.aux} yields
	\begin{align*}
		\int_0^T\cW_2^2\big(\P^{\hat\alpha^{\text{\fontsize{4}{4}\selectfont $N$}},N}\circ(\hat\alpha^{i,N}_t)^{-1}, \cL_{\hat\alpha}(\hat\alpha_t) \big)\diff t \le C\bigg(\frac1N + NR^2_N + \gamma^N\bigg), 
	\end{align*}
	which is \Cref{eq:conv.law.hatalpha}. When the law in the coefficients $b$ and $f$ is state dependent, the same computations as above yield
	\begin{align*}
		&|\widetilde Y^{i,N}_0 - Y_0^i|^2 + \int_0^T\cW_2^2\big(\P^{\hat\alpha^\smalltext{N},N}\circ(\hat\alpha^{i,N}_t), \cL_{\hat\alpha}(\hat\alpha_t) \big)\diff t\\\notag
		&\le C\bigg(\sup_{t \in [0,T]}\E^{\P^{\smalltext{\hat\alpha}}}\Big[  \cW_2^2\big(L^N(\X^N_{t}), \cL_{\hat\alpha}(X_{ t})\big)\Big]+ \sup_{t \in [0,T]}\E^{\P^{\smalltext{\hat\alpha}}}\Big[  \cW_2^2\big(L^N(\hat{\bm{\alpha}}^N_{t}), \cL_{\hat\alpha}(\hat\alpha_{t})\big)\Big]\bigg) \le C\big(r_{\text{\fontsize{5}{5}\selectfont$\np$}, k, q} + r_{\text{\fontsize{5}{5}\selectfont$\np$},m, q}\big),
	\end{align*}
	where the second inequality follows by \citeauthor*{fournier2015rate} \cite[Theorem 1]{fournier2015rate}, provided that the processes $\hat\alpha^i$ take values in $\R^k$. 
	This concludes the proof.

\subsection{Proof for the example}

\begin{proof}[Proof of \Cref{cor:price.impact}]
	It is easily checked that the only fixed-point $\hat a\in \mathcal{O}^\np$ of $\mathcal{H}^\np$ is given by $a^i\coloneqq  \frac{1}{1-\gamma( \x_t^i)/\np}z^i$, $i\in\{1,\dots,N\}$.
	Thus, the function $\Lambda$ is
	\begin{equation*}
		\Lambda_t\big(\x, \xi, z^i,\aleph^{i,N}(\x)\big) \coloneqq \frac{1}{1-\aleph^{i,N}(\x)}z^i, \; \text{with} \; \aleph^{i,N}(\x) = \frac{\gamma (\x_{ t}^i)}{N},\;\text{and}\; \Lambda_t(\x, \xi, z^i,0) = z^i .
	\end{equation*}
	Thus, {\rm \Cref{ass.Lambda.conv}}.$(i)$--$(iv)$ is satisfied.
	Note that the Hamiltonian $H$ takes the form
	\begin{equation*}
		H_t(\x,\xi,z) \coloneqq  \gamma( \x_{ t})\int_{\R} a\xi(\d a) +\frac{1}{2}|z|^2 - k(t,\x_{ t} ),
	\end{equation*}
	since $b$ and $\sigma$ are respectively given by $b_t(\x, \xi,a) = a$ and $\sigma_t(\x) = \textit{\rm Id}_{\R^d}$. Since the minimiser of the Hamiltonian of the mean-field game satisfies $\Lambda_t(\x, \xi, z) = z$, the generalised McKean--Vlasov BSDE \eqref{eq:bsde.char.mfg} takes the form
	\begin{equation}
	\label{eq:gen.MkV.example}
	 	Y_t = g(X_T) + \int_t^T\big(\gamma(X_{s})\E^{\P^{\smalltext{\hat\alpha}}}[Z_s]-c(Z_s)-k(s,X_{ s})\big)\diff s - \int_t^TZ_s\cdot\diff W^{\hat\alpha}_s,\; t\in[0,T],\; \P^{\hat\alpha}\text{\rm--a.s.} ,
	\end{equation} 
	with $\diff \P^{\hat\alpha}/\mathrm{d}\P = \cE\big( \int_0^\cdot Z_s\diff W_s\big)_T$.
	By \Cref{thm:Gen.MkV.BSDE}, this equation admits a unique solution.
	Thus, by \Cref{prop:char.mfe}, the mean-field game admits a unique solution.
	Moreover, in the present case the, if the functions $k$ and $\gamma$ are state-dependent then the FBSDE \eqref{eq:FBSDE} admits a unique square-integrable (global) solution, see \emph{e.g.} \citeauthor*{delarue2002existence} \cite[Theorem 2.6]{delarue2002existence}.
	Therefore, \Cref{ass.Lambda.conv} is satisfied. 
\end{proof}

\section{Existence and uniqueness results for McKean--Vlasov BSDEs}
\label{sec:exists.MKVBSDE}
This section compiles existence and uniqueness results for McKean--Vlasov equations.
In the ensuing subsection we investigate well-posedness of classical McKean--Vlasov equations with dependence in the law of the control.
Since we were not able to find such results in the literature, we provide a proof in \Cref{sec:appen}. 
This will serve us when investigating generalised McKean--Vlasov equations in the final subsection.
\subsection{Well-posedness of McKean--Vlasov BSDEs}
\label{sec:McKV-BSDE}

We give ourselves the maps
\[
F:[0,T]\times\Cc_m \times \R^\xdim\times\R^{\xdim\times \bmdim}\times \cP_2\big(\R^\xdim\times \R^{\xdim\times \bmdim}\big) \longrightarrow \R^\xdim,\; G:\Cc_m \longrightarrow\R^m,
\]
and consider the McKean--Vlasov BSDEs
\begin{equation}
\label{eq:MkV BSDE}
	\Yc_t = G(X) +\int_t^T F_s\big( X_{\cdot\wedge s},\Yc_s, \Zc_s, \cL(\Yc_s,\Zc_s)\big)\mathrm{d}s - \int_t^T\Zc_s\mathrm{d}W_s,\; t\in[0,T],\; \P \text{\rm--a.s.}
\end{equation}
on the probability space $(\Omega,\cF,\P)$.
We will consider the following assumption.
\begin{assumption}\label{assum:main}
	$(i)$ For any $(y,z,\xi)\in \mathbb{R}^m\times \mathbb{R}^{m\times d}\times \cP_2(\mathbb{R}^m\times \mathbb{R}^{\xdim\times \bmdim})$, $[0,T]\times\cC_m\ni (t,\bx)\longmapsto F_t(\bx,y,z,\xi)$ is $\F$-optional$;$ 

	\medskip
	$(ii)$ the functions $G$ and $F(\cdot, y,z,\xi)$ have polynomial growth, \emph{i.e.} for every $(t,y,z,\xi)\in [0,T]\times\mathbb{R}^m\times \mathbb{R}^{m\times d}\times \cP_2(\mathbb{R}^m\times \mathbb{R}^{m\times d})$
	\begin{align*}
		\|G(\bx)\|\le \ell_G\big(1 + \|\bx\|_\infty^{e}\big),\; \text{\rm and}\; \|F_t(\bx,0,0,\delta_{(0,0)})\|_\infty\le \ell_F\big(1 + \|\x\|_\infty^{e^\prime}\big),
	\end{align*}
	for some $(e,e^\prime)\in[1,\infty)^2$ and positive $\ell_F$, and $\ell_G;$

	\medskip
	$(iii)$ the function $F$ is Borel-measurable in all its arguments, and the following condition is satisfied:
	for any $(t,\bx)\in[0,T]\times\cC_m$, $(y,z,\xi)\longmapsto F_t(\bx,y,z,\xi)$ is uniformly $\ell_F$--Lipschitz-continuous for some $\ell_F>0$. That is, for any $(t,\bx,y,y^\prime,z,z^\prime,\xi,\xi^\prime)\in [0,T]\times\cC_m\times(\mathbb{R}^m)^2\times (\mathbb{R}^{m\times d})^2\times \big(\cP_2(\mathbb{R}^m\times \mathbb{R}^{m\times d})\big)^2$ 
	\begin{align*}
		\big\|F_t(\bx,y,z,\xi) - F_t(\bx,y^\prime,z^\prime,\xi^\prime)\big\|&\le \ell_F \big(\|y-y^\prime\| + \|z-z^\prime\|+ {\mathcal W}_2(\xi, \xi^\prime) \big).
	\end{align*}
\end{assumption}

\subsubsection{Liminary results}
We now turn our attention to the existence of McKean--Vlasov BSDEs.
\begin{proposition}
\label{prop:MckV BSDE}
	Let {\rm \Cref{assum:main}} hold.
	Then, 
	there exists a unique solution to {\rm \eqref{eq:MkV BSDE}} such that $\Yc\in\S^2(\RR^{m},\FF)$ and $\Zc\in\H^2(\RR^{m\times d},\FF)$.
\end{proposition}
\begin{remark}
\label{rem:F.moments}
	Recall that since $\sigma$ is bounded, $X$ has moments of every order under $\P$.
	Therefore, it follows by polynomial growth of $F(\x, 0,0,\delta_{(0,0)})$ that $F_t(X_{\cdot\wedge t},0,0,\delta_{(0,0)})$ also has moments of every order under $\P$.
	Similarly, by polynomial growth of $G$, the random variable $G(X)$ is $\P$--square-integrable.
	These will be used in the remainder of this section without further mention.
\end{remark}
\subsubsection{Proof of Proposition \ref{prop:MckV BSDE}}
\label{sec:appen}
We start with some estimates.
\begin{lemma}\label{lemma:estimates}
	Let $F^\prime$ and $G^\prime$ be maps satisfying the same assumptions as $F$ and $G$ in {\rm\Cref{assum:main}}, and let us assume that $(\Yc,\Zc)\in \H^2(\RR^m,\FF)\times \H^2(\RR^{m\times d},\FF)$ and that $(\Yc^\prime,\Zc^\prime)\in \H^2(\RR^m,\F)\times \H^2(\RR^{m\times d},\FF)$ solves {\rm \Cref{eq:MkV BSDE}} with $F^\prime$ instead of $F$, and $\xi^\prime$ instead of $\xi$. 
	Then $\Yc$ and $\Yc^\prime$ both belong to $\SS^2(\RR^m,\FF)$, and we have for any $\beta\in\RR$, any $\varepsilon>0$, and any $t\in[0,T]$
 	\begin{align*}
 		\EE\bigg[\mathrm{e}^{\beta t}\|\delta \Yc_t\|^2& +\int_t^T\mathrm{e}^{\beta s}\Big(\big(\beta-(\varepsilon^{-1}(1+2\ell_F^2)+4\ell)\big)\|\delta \Yc_s\|^2+(1-2\varepsilon)\|\delta \Zc_s\|^2\Big)\mathrm{d}s\bigg] \leq \varepsilon \EE\bigg[ \mathrm{e}^{\beta T}\|\delta G\|^2+\int_t^T\mathrm{e}^{\beta s}\|\delta F_s\|^2\mathrm{d}s\bigg],
 	\end{align*}
 	where we denoted for simplicity
	\[
		\delta G\coloneqq G-G^\prime,\; \delta \Yc_\cdot\coloneqq \Yc_\cdot-\Yc_\cdot^\prime,\; \delta \Zc_\cdot\coloneqq \Zc_\cdot-\Zc_\cdot^\prime,\; \delta F_\cdot\coloneqq  F_\cdot\big(X_{\cdot\wedge\cdot},\Yc^\prime_\cdot,\Zc^\prime_\cdot,\cL(\Yc^\prime_\cdot,\Zc^\prime_\cdot)\big)-F^\prime_\cdot\big(X_{\cdot\wedge\cdot},\Yc^\prime_\cdot,\Zc^\prime_\cdot,\cL(\Yc^\prime_\cdot,\Zc^\prime_\cdot)\big).
	\]
\end{lemma}

\begin{proof}
	We start by proving that $\Yc\in \S^2(\RR^m,\FF)$, the proof for $\Yc^\prime$ being the same. Notice first that we immediately have
	\[
 		\sup_{t\in[0,T]}\|\Yc_t\|^2\leq 3\bigg(\|G(X)\|^2+\bigg(\int_0^T\big\|F_s(X_{\cdot\wedge s},\Yc_s,\Zc_s,\cL(\Yc_s,\Zc_s)\big)\big\|\mathrm{d}s\bigg)^2+\sup_{t\in[0,T]}\bigg\|\int_t^T\Zc_s\mathrm{d}W_s\bigg\|^2\bigg).
	\]
	Now, by Doob's inequality, since $\Zc\in \H^2(\RR^{m\times d},\FF)$, we have
	\[
 		\EE\bigg[\sup_{t\in[0,T]}\bigg\|\int_t^T\Zc_s\mathrm{d}W_s\bigg\|^2\bigg]\leq 2\EE\bigg[\sup_{t\in[0,T]}\bigg\|\int_0^t\Zc_s\mathrm{d}W_s\bigg\|^2\bigg]+2\|\Zc\|^2_{\H^\smalltext{2}(\RR^{\smalltext{m}\smalltext{\times} \smalltext{d}},\FF)}\leq 10\|\Zc\|^2_{\H^\smalltext{2}(\RR^{\smalltext{m}\smalltext{\times} \smalltext{d}},\FF)},
	 \]
	so that
	\begin{align*}
	 	\|\Yc\|^2_{\SS^\smalltext{2}(\RR^\smalltext{m},\FF)}&\leq  3\EE\big[\|G(X)\|^2\big]+12\EE\bigg[\bigg(\int_0^T\|F_s(X_{\cdot\wedge s},0,0,\delta_{(0,0)})\|\mathrm{d}s\bigg)^2\bigg]+12\ell_F^2T\|\Yc\|^2_{\H^\smalltext{2}(\RR^\smalltext{m},\FF)}+6(2\ell_F^2T+5)\|\Zc\|^2_{\H^\smalltext{2}(\RR^{\smalltext{m}\smalltext{\times} \smalltext{d}},\FF)}\\
		&\quad+12\ell_F^2T\int_0^T\cW_2^2\big(\cL(\Yc_s,\Zc_s),\delta_{(0,0)}\big)\mathrm{d}s\\
 		&\leq 3\EE\big[\|G(X)\|^2\big]+12\EE\bigg[\bigg(\int_0^T\|F_s(X_{\cdot\wedge s},0,0,\delta_{(0,0)})\|\mathrm{d}s\bigg)^2\bigg]+24\ell_F^2T\|\Yc\|^2_{\H^\smalltext{2}(\RR^\smalltext{m},\FF)}+6(4\ell_F^2T+5)\|\Zc\|^2_{\H^\smalltext{2}(\RR^{\smalltext{m}\smalltext{\times} \smalltext{d}},\FF)}<\infty,
 	\end{align*}
	where we used the fact that
	\begin{equation}\label{eq:controlwasser}
		\cW_2^2\big(\cL(\Yc_s,\Zc_s),\delta_{(0,0)}\big)\leq \EE\big[\|\Yc_s\|^2+\|\Zc_s\|^2\big],\; \mathrm{d}s\otimes\mathrm{d}\PP\text{\rm--a.e. on}\; [0,T]\times\Omega. 
	\end{equation}
	We can now obtain the estimates. Let us apply It\^o's formula to $(\mathrm{e}^{\beta t}\|\Yc_t\|^2)_{t\in[0,T]}$. We obtain that for any $t\in[0,T]$
	\begin{align}\label{eq:mkvbsde}
		\nonumber \mathrm{e}^{\beta t}\|\delta \Yc_t\|^2+\int_t^T\mathrm{e}^{\beta s}\|\delta \Zc_s\|^2\mathrm{d}s&= \mathrm{e}^{\beta T}\|\delta G\|^2+2\int_t^T\mathrm{e}^{\beta s}\delta \Yc_s\cdot \big(F_s\big(X_{\cdot\wedge s},\Yc_s,\Zc_s,\cL(\Yc_s,\Zc_s)\big) -F^\prime_s\big(X_{\cdot\wedge s},\Yc^\prime_s,\Zc^\prime_s,\cL(\Yc^\prime_s,\Zc^\prime_s)\big)\big)\\
		&\quad-\beta\int_t^T\mathrm{e}^{\beta s}\|\delta \Yc_s\|^2\mathrm{d}s-2\int_t^T\mathrm{e}^{\beta s}\delta \Yc_s\cdot \delta \Zc_s\mathrm{d}W_s.
	\end{align}
	Notice now that, using the inequalities $2ab\leq \varepsilon a^2 +\varepsilon^{-1}b^2$, and $\sqrt{a^2+b^2}\leq |a|+|b|$, valid for any $(a,b,\varepsilon)\in\RR\times\RR\times (0,\infty)$
	\begin{align*}
 		&\ 2\bigg|\int_t^T\mathrm{e}^{\beta s}\delta \Yc_s\cdot \big(F_s\big(X_{\cdot\wedge s},\Yc_s,\Zc_s,\cL(\Yc_s,\Zc_s)\big) -F^\prime_s\big(X_{\cdot\wedge s},\Yc^\prime_s,\Zc^\prime_s,\cL(\Yc^\prime_s,\Zc^\prime_s)\big)\big)\bigg|\\
 		\leq&\ 2\int_t^T\mathrm{e}^{\beta s}\|\delta \Yc_s\| \Big(\|\delta F_s\|  +\ell\big(\|\delta \Yc_s\|+\|\delta \Zc_s\|+\cW_2\big(\cL(\Yc_s,\Zc_s),\cL(Y^\prime_s,Z^\prime_s)\big)\big)\Big)\mathrm{d}s\\
 		\leq &\ \big(\varepsilon^{-1}(1+2\ell_F^2)+2\ell_F\big)\int_t^T\mathrm{e}^{\beta s}\|\delta \Yc_s\|^2\mathrm{d}s+2\ell_F\int_t^T\mathrm{e}^{\beta s}\|\delta \Yc_s\|\EE\big[\|\delta \Yc_s\|\big]\mathrm{d}s+\varepsilon \int_t^T\mathrm{e}^{\beta s}\|\delta F_s\|^2\mathrm{d}s+\varepsilon\int_t^T\mathrm{e}^{\beta s}\|\delta \Zc_s\|^2\mathrm{d}s\\
 		&+\varepsilon\EE\bigg[\int_t^T\mathrm{e}^{\beta s}\|\delta \Zc_s\|^2\mathrm{d}s\bigg].
	\end{align*}
	Notice as well that we have by Burkholder--Davis--Gundy's inequality, that there is some $C>0$ such that
	\begin{align*}
 		\EE\bigg[\sup_{t\in[0,T]}\bigg\|\int_0^t\mathrm{e}^{\beta s}\delta \Yc_s\cdot \delta \Zc_s\mathrm{d}W_s\bigg\|\bigg]&\leq C\EE\bigg[\bigg|\int_0^t\mathrm{e}^{2\beta s}\|\delta \Yc_s\|^2\| \delta \Zc_s\|^2\mathrm{d}s\bigg|^{1/2}\bigg]\leq \mathrm{e}^{\beta T}\|\Yc\|_{\SS^\smalltext{2}(\RR^\smalltext{m},\FF)}\|\Zc\|_{\H^\smalltext{2}(\RR^{\smalltext{m}\smalltext{\times} \smalltext{d}},\FF)}<\infty,
	\end{align*}
	proving thus that $\big(\int_0^t\mathrm{e}^{\beta s}\delta \Yc_s\cdot \delta \Zc_s\mathrm{d}W_s\big)_{t\in[0,T]}$ is an $(\FF,\PP)$-martingale. Using these computations in \eqref{eq:mkvbsde} and taking expectations, we deduce the desired result.
\end{proof}

\begin{proof}[Proof of \Cref{prop:MckV BSDE}]
	For any $\beta\in \RR$, we define a new norm on $\Hc^2\coloneqq \H^2(\RR^m,\FF)\times \H^2(\RR^{m\times d},\FF)$ by
	 \[
		 \|(y,z)\|_{\Hc^\smalltext{2}_\smalltext{\beta}}^2\coloneqq \EE\bigg[\int_0^T\mathrm{e}^{\beta s}\| y_s\|^2\mathrm{d}s+\int_0^T\mathrm{e}^{\beta s}\|z_s\|^2\mathrm{d}s\bigg],\; (y,z)\in \Hc^2.
	 \]
	It is obvious that $\big(\Hc^2,\|\cdot\|_{\Hc^\smalltext{2}_\smalltext{\beta}}\big)$ is a Banach space, and that the norms $\big(\|\cdot\|_{\Hc^\smalltext{2}_\smalltext{\beta}}\big)_{\beta\in\RR}$ are all equivalent. 
	We now define a map $\Phi:\big(\Hc^2,\|\cdot\|_{\Hc^\smalltext{2}_\smalltext{\beta}}\big)\longrightarrow \big(\Hc^2,\|\cdot\|_{\Hc^\smalltext{2}_\smalltext{\beta}}\big)$ by
	\[
		\Phi(u,v)\coloneqq (U,V),\; (u,v)\in\Hc^2_\beta,
	\]
	where\footnote{To be perfectly rigorous, $U$ should be chosen as a continuous $\PP$-modification of the right-hand side, which exists since the right-hand side is the sum of a square-integrable $(\FF,\PP)$-martingale and the continuous and square-integrable quantity $-\int_0^tf_s(u_s,v_s,\cL(u_s,v_s))\mathrm{d}s$, and $\FF$ is the $\PP$-augmentation of a Brownian filtration.}
	\[
		U_t\coloneqq \EE\bigg[\xi+\int_t^TF_s\big(X_{\cdot\wedge s},u_s,v_s,\cL(u_s,v_s)\big)\mathrm{d}s\bigg|\cF_t\bigg],\; t\in[0,T],
		\]
	and $V$ is obtained through the martingale representation
	\[
		U_t+\int_0^tF_s\big(X_{\cdot\wedge s},u_s,v_s,\cL(u_s,v_s)\big)\mathrm{d}s=U_0+\int_0^tV_s\mathrm{d}W_s,\; t\in[0,T].
	\]
	Notice that the fact that $(U,V)\in\Hc^2$ is immediate, using \eqref{lemma:estimates} with $F^\prime$ and $\xi^\prime$ equal to $0$ and generator $F_\cdot(u_\cdot,v_\cdot,\cL(u_\cdot,v_\cdot))$ instead of $F$. 
	Moreover, using again \eqref{lemma:estimates}, we obtain that for any $(u,v,u^\prime,v^\prime)\in\Hc^2\times\Hc^2$, with images by $\Phi$ denoted by $(U,V,U^\prime,V^\prime)$, any $\beta\in\RR$, any $\varepsilon>0$, and any $t\in[0,T]$
	\begin{align*}
	 	&\ \EE\bigg[\mathrm{e}^{\beta t}|U_t- U^\prime_t|^2+\big(\beta-\big(\varepsilon^{-1}(1+2\ell_F^2)+4\ell_F\big)\big)\int_t^T\mathrm{e}^{\beta s}\|U_s-U^\prime_s\|^2\mathrm{d}s+(1-2\varepsilon)\int_t^T\mathrm{e}^{\beta s}\| V_s-V^\prime_s\|^2\mathrm{d}s\bigg]\\
	 	\leq&\ \varepsilon \EE\bigg[ \int_t^T\mathrm{e}^{\beta s}\big\|F_s\big(X_{\cdot\wedge s},u_s,v_s,\cL(u_s,v_s)\big)-F_s\big(X_{\cdot\wedge s},u^\prime_s,v^\prime_s,\cL(u^\prime_s,v^\prime_s)\big)\big\|^2\mathrm{d}s\bigg]\\
	 	\leq&\ 6\ell_F^2\varepsilon\EE\bigg[\int_t^T\mathrm{e}^{\beta s}\|u_s-u^\prime_s\|^2\mathrm{d}s+\int_t^T\mathrm{e}^{\beta s}\| v_s-v^\prime_s\|^2\mathrm{d}s\bigg].
	\end{align*}
	Taking $\beta>\varepsilon^{-1}(1+2\ell_F^2)+4\ell_F$ and $\varepsilon<1/2$, we deduce
	\[
	 	\EE\bigg[\int_0^T\mathrm{e}^{\beta s}\| V_s-V^\prime_s\|^2\mathrm{d}s\bigg]\leq \frac{6\ell_F^2\varepsilon}{1-2\varepsilon}\big\|(u-u^\prime,v-v^\prime)\big\|^2_{\Hc^\smalltext{2}_\smalltext{\beta}},
	 \; 
	 	\EE\big[\mathrm{e}^{\beta t}\|U_t- U^\prime_t\|^2\big]\leq 6\ell_F^2\varepsilon\EE\bigg[\int_t^T\mathrm{e}^{\beta s}\big(\|u_s-u^\prime_s\|^2+\| v_s-v^\prime_s\|^2\big)\mathrm{d}s\bigg].
	\]
	Integrating the second inequality on $[0,T]$ and using Fubini's theorem then leads to
	\[
	 	\EE\bigg[\int_0^T\mathrm{e}^{\beta s}\| U_s-U^\prime_s\|^2\mathrm{d}s\bigg]\leq 6\ell_F^2T\varepsilon\big\|(u-u^\prime,v-v^\prime)\big\|^2_{\Hc^\smalltext{2}_\smalltext{\beta}},
	\]
	so that overall
	\[
		\big\|(U-U^\prime,V-V^\prime)\big\|^2_{\Hc^\smalltext{2}_\smalltext{\beta}}\leq 6\ell_F \varepsilon\bigg(T+\frac1{1-2\varepsilon}\bigg)\big\|(u-u^\prime,v-v^\prime)\big\|^2_{\Hc^\smalltext{2}_\smalltext{\beta}}.
	\]
	We can therefore always choose $\varepsilon$ sufficiently small so that $6\ell_F \varepsilon\big(T+(1-2\varepsilon)^{-1}\big)<1$ and then $\beta$ sufficiently large so that $\beta>\varepsilon^{-1}(1+2\ell_F^2)+4\ell_F$, in which case $\Phi$ becomes a contraction on $(\Hc^2,\|\cdot\|_{\Hc^\smalltext{2}_\smalltext{\beta}})$, which therefore has a unique fixed-point, providing us with the required solution to the McKean--Vlasov BSDE.
\end{proof}

\subsection{Well-posedness of generalised McKean--Vlasov BSDEs}
\label{sec:gen.MckV.BSDE}

Let us now address the question of well-posedness of the (as far as we know) new kind of equation that was derived in \Cref{sec:MFG-cha} on characterisation of mean-field equilibria.
Since this equation seems to be central for the investigation of mean-field games with interaction through the control (at least in their weak formulation), we state and prove the results in a general setting.
Thus, we use the generic function $F$ and $G$ on an extended space:
\[
F:[0,T]\times\Cc_m \times \R^\xdim\times\R^{\xdim\times \bmdim}\times \cP_2\big(\Cc_m\times\R^\xdim\times \R^{\xdim\times \bmdim}\big) \longrightarrow \R^\xdim,\; G:\Cc_m\times \cP_2(\Cc_m) \longrightarrow\R^m,
\]
and we fix a function
\[
	B:[0,T]\times \cC_m\times \RR^{m\times \bmdim}\times \cP_2(\Cc_m\times\RR^m\times\RR^{m\times\bmdim}) \longrightarrow \RR^{\bmdim}.
\]
We are interested in the equation
\begin{equation}
\label{eq:Gen BSDE}
	\begin{cases}
		\displaystyle P_t  = G\big(X,\cL_{\overline\PP}(X)\big) + \int_t^TF_s\big(X_{\cdot\wedge s}, P_s,Q_s, \cL_{\overline\PP}(X_{\cdot\wedge s},P_s,Q_s)\big)\mathrm{d}s - \int_t^TQ_s\mathrm{d}\overline W_s,\; t\in[0,T],\; \overline \PP\text{\rm--a.s.},\\[1em]
	\displaystyle\; \cL_{\overline\PP}(X_{\cdot\wedge s},P_s,Q_s)\coloneqq  \overline\PP\circ (X_{\cdot\wedge s},P_s,Q_s)^{-1},\;	\frac{\mathrm{d}\overline\PP}{\mathrm{d}\PP} \coloneqq  \cE\bigg(\int_0^TB_s\big(X_{\cdot\wedge s},Q_s, \cL_{\overline\PP}(X_{\cdot\wedge s},P_s,Q_s)\big)\cdot \mathrm{d}W_s \bigg),\\
	\displaystyle \overline  W_\cdot\coloneqq  W_\cdot - \int_0^\cdot B_s\big(X_{\cdot\wedge s},Q_s, \cL_{\overline\PP}(X_{\cdot\wedge s},P_s,Q_s)\big)\mathrm{d}s.
	\end{cases}
\end{equation}
The distinct feature of this equation, compared to the standard McKean--Vlasov equation considered in \Cref{sec:McKV-BSDE} is that it `embeds' a fixed point problem in its formulation.
In fact, the driving Brownian motion $\overline W$ and the underlying probability measure $\overline \P$ are themselves unknown, or at least part of the solution. In this regard, the generalised McKean--Vlasov BSDE seems to be close in spirit to the notion of weak solutions to BSDEs developed by \citeauthor*{buckdahn2005weak} \cite{buckdahn2005weak}, and \citeauthor*{buckdahn2006backward} \cite{buckdahn2006backward,buckdahn2008continuity}.

\medskip
The proof of the ensuing result will make use of the Cameron--Martin space $\mathfrak{H}$ whose definition we recall
\[
\mathfrak{H}\coloneqq  \bigg\{h\in \mathcal{C}_d: \text{\rm $h$ is absolutely continuous, $h_0=0$ and}\; \int_0^T|\dot h(t)|^2\mathrm{d}t <\infty \bigg\}.
\]

As usual, the Cameron--Martin space is equipped with the norm $\|h\|_{\mathfrak{H}}^2\coloneqq  \int_0^T|\dot h(t)|^2\mathrm{d}t$. Notice also that we will sometimes abuse notations slightly and still say that a random process $h:[0,T]\times\Omega\longrightarrow \Cc_d$ belongs to $\mathfrak H$, provided that for $\P$--a.e. $\omega\in\Omega$, we have that $t\longmapsto h_t(\omega)$ belongs to $\mathfrak H$. We call such processes \emph{random shifts}.

\medskip
The main result of this section is given below, and its proof will complete the argument for the existence and uniqueness of a mean-field game equilibrium. We however first state our main assumptions for this section.

\begin{assumption}\label{assum:gen.MkV}
	$(i)$ The maps $\sigma$, $G$ and $B$  are respectively $\ell_\sigma$--, $\ell_G$--, and $\ell_B$--Lipschitz-continuous with $B$ bounded.
	That is for any $(t,\x,\x^\prime,y,y^\prime,z,z^\prime,\mu,\mu^\prime,\xi,\xi^\prime)\in [0,T]\times \cC_m^2\times(\R^m)^2\times (\RR^{m\times \bmdim})^2\times(\cP_2(\cC_m))^2\times \big(\cP_2(\cC_m\times\RR^m\times\RR^{m\times\bmdim})\big)^2$, there are some positive constants $\ell_B,$ $\ell_{B(X)}$, $\ell_\sigma,$ $\ell_{G(X)}$ and $\ell_{G(\mu)}$ such that
	\begin{align*}
		\big\|B_t(\bx,z,\xi) - B_t(\bx^\prime,z^\prime,\xi^\prime)\big\|\le \ell_B \Big(\|\bx - \bx^\prime\|_\infty + \|z-z^\prime\|+ {\mathcal W}_2\big((\xi^2,\xi^3), ((\xi^\prime)^2,(\xi^\prime)^3)\big) \Big)+\ell_{B(X)}{\mathcal W}_2\big(\xi^1,(\xi^\prime)^1\big),\\
		 \|\sigma_t(\x) - \sigma_t(\x^\prime)\|\le \ell_\sigma\|\x - \x^\prime\|_\infty,\;
		\big\|G(\bx, \mu) - G(\bx^\prime,\mu^\prime)\| \le \ell_{G(X)}\|\bx - \bx^\prime\|_\infty + \ell_{G(\mu)}\cW_2(\mu, \mu^\prime),\; \|B\|_\infty < \infty,
		 \end{align*}
	where $\xi^i$, $i\in\{1,2,3\}$, is the $i$-th marginal of $\xi\in\Pc_2(\Cc_m\times\R^m\times\R^{m\times d}).$
	In addition
{\color{black}	\begin{align*}
		(\bx - \bx^\prime)\cdot\big(\|B_t(\bx,z,\xi) - B_t(\bx^\prime,z,\xi)\big)\le -K_B\|\bx - \bx^\prime\|_\infty^2,
	\end{align*}
	for a constant $K_B\ge \ell_B^2 + (4C_{\rm BDG}+1)\ell_G^2;$}
	\medskip

	$(ii)$ $F$ satisfies one of the following conditions
		\begin{itemize}
		\item[$(iia)$]  $F$ is continuously differentiable in $(y,z)$, Lipschitz-continuous in $y$, and locally Lipschitz-continuous and of quadratic growth in $z$, in the sense that there is a constant $\ell_F>0$ and a linearly growing function $\ell^1_F:(\R^{m\times d})^2\longrightarrow (0,+\infty)$ such that for any $(t,\x,\x^\prime,y,y^\prime,z,z^\prime,\xi,\xi^\prime)\in [0,T]\times \cC_m^2\times(\R^m)^2\times (\RR^{m\times \bmdim})^2\times(\cP_2(\cC_\smalltext{m}\times \R^\smalltext{m}\times\R^{m\times d}))^2$ we have
		\begin{align*}
		\|F_t(\bx,y,z,\xi)\| \le \ell_F\bigg(1 + \|\x\|_\infty + \|y\| + \|z\|^2+\bigg(\int_{\Cc_\smalltext{\xdim\times} \RR^\smalltext{m}\times\R^{\smalltext{m}\smalltext{\times} \smalltext{d}}}\big(\|x\|_\infty^2 + \|y\|^2+\|\zeta\|^2\big)\xi(\diff x, \diff y,\d \zeta)\bigg)^{1/2} \bigg),\\
			\big\|F_t(\bx,y,z,\xi) - F_t(\bx^\prime,y^\prime,z^\prime,\xi^\prime)\big\|\le \ell_F^1( z, z^\prime) \|z-z^\prime\|  +\ell_F\big(  \|\bx - \bx^\prime\|_\infty +\|y-y^\prime\| + {\mathcal W}_2(\xi, \xi^\prime) \big),\; \ell^1_F(z,z^\prime)\leq \ell_F\big(1+\|z\|+\|z^\prime\|\big).
		\end{align*}
		Moreover, the random variables $G(X,\xi)$ and processes $F_t(X,y,z,\xi)$ and $B_t(X,z,\xi)$ are Malliavin differentiable with bounded Malliavin derivatives$;$

	\item[$(iib)$] $F$ is $\ell_F$--Lipschitz-continuous, \emph{i.e.} the function $\ell_F^1$ in $(iia)$ is constant, equal to $\ell_F;$
		\end{itemize}
	
	$(iii)$ we have $\sup_{t\in [0,T]} \|F_t(X_{\cdot\wedge t},0,0,\delta_{\{0,0,0\}})\|_{\L^\smalltext{\infty}(\R^\smalltext{m},\Fc_{\text{\fontsize{4}{4}\selectfont$T$}})}<\infty;$

\medskip
	$(iv)$ 
	For every probability measure $\Pi$ on $(\Omega,\F)$ and every independent $(\F,\Pi)$--Brownian motions $B$, the following forward--backward {\rm SDE} admits a unique solution $(\overline X, \overline Y, \overline Z) \in \S^2(\R^m,\F)\times \S^2(\R,\F)\times  \H^2(\R^d,\F)$
	 \begin{equation*}
 	\begin{cases}
 		\d \overline X_t = B_t(\overline X_{\cdot\wedge t}, z_t, \xi^1_t)\d t + \sigma_t(\overline X_{\cdot\wedge t})\d B_t,\\
 		\d \overline Y_t = -F_t(\overline X_t, \overline Y_t, \overline Z_t,\xi^1_t )\d t + \overline Z^i_t\d B_t,\\ 
 		\overline Y_T = G\big(\overline X, \xi^3),\; \overline X_0 = X_0,\; \Pi\text{\rm--a.s.}, 
 	\end{cases}
 \end{equation*}
 for every parameter $(z, \xi^1, \xi^2,\xi^3)\in \H^2(\R^d,\F)\times \cP_2(\cC,\R,\R^d)^2\times \cP_2(\cC);$
		
	$(v)$ $\sigma$ is state dependent, the function $G$ satisfies {\rm \Cref{ass.g.smooth}.$(i)$}
	and the function $\widetilde F$ satisfies $(ii)$, with $\widetilde F$ given by
	\begin{align}
	\label{eq:def.tilde.F}
		\notag
		\widetilde F_t(\x_{\cdot\wedge t}, y,z, \xi) &\coloneqq  F_t\big(\x_{\cdot\wedge t}, y+ G(\x_t, \xi^1_t),z+\partial_xG(\x_t,\xi^1_t)\sigma_t(\x_t), \Phi(\xi)\big) + \frac12\mathrm{Tr}\big[\partial_{xx}G(\x_t, \xi^1_t)\sigma_t(\x_t)\sigma_t^\top(\x_t)\big]\\
		&\quad + \frac12\int_{\mathbb{R}^\smalltext{\xdim}}\mathrm{Tr}\big[ \partial_a\partial_\mu G(\x_t, \xi^1_t)(a)\sigma_t(a)\sigma_t(a)^\top \big] \xi^1(\diff a) -  B_t\big(\x_t, z+\partial_xG(\x_t,\xi^1_t)\sigma_t(\x_t), \Phi(\xi)\big)\cdot \partial_xG(\x_t,\xi^1_t)\sigma_t(\x_t),
	\end{align}
	where $\xi^i$, $i\in\{1,2,3\}$, is the $i$-th marginal of $\xi\in \cP_2(\R^m\times \R^m\times \R^{m\times d})$, and where the map $\Phi:\cP_2(\R^m\times \R^m\times \R^{m\times d}) \longrightarrow \cP_2(\cC_m\times \R^m\times \R^{m\times d})$ is defined as being the unique measure on $\R^m\times \R^m\times \R^{m\times d}$ such that for any Borel sets $A_1\times A_2\times A_3 \in \R^m\times \R^m\times \R^{m\times d}$
	\begin{equation*}
		\Phi(\xi)(A_1\times A_2 \times A_3) = \xi^1(A_1)(\xi^1\otimes\xi^2)\big(f^{-1}_1(A_2)\big)(\xi^1\otimes \xi^3)\big(f^{-1}_2(A_3)\big),
	\end{equation*}
	where $f_1(x, p) \coloneqq  p + G(x, \xi_1),\; f_2(x, q) \coloneqq  q + \partial_xG(x, \xi_1)\sigma_t(x).$
\end{assumption}

We can now state our main result.
\begin{theorem}
\label{thm:Gen.MkV.BSDE}
	Let {\rm \Cref{assum:gen.MkV}}.$(i)$--$(iv)$ hold.
	There is $\Psi>0$ such that {\color{black}if $\|G\|_\infty\le \Psi$},
	then {\rm \Cref{eq:Gen BSDE}} admits a unique solution $(P,Q)\in\mathbb{H}^2(\R^m,\F) \times \H^{2}(\R^m,\F)$.
	Moreover, 
	 if in addition 
	{\rm\Cref{assum:gen.MkV}}.$(v)$ holds, then one can take $\Psi=+\infty$.
\end{theorem}

We start with some \emph{a priori} estimates for solutions to \Cref{eq:Gen BSDE}
\begin{lemma}
	\label{lem:a.priori.estimate.general}
	Let {\rm \Cref{assum:gen.MkV}}.$(i)$, $(iib)$ and $(iii)$ hold. Then, if $(P,Q)\in\H^2(\R^m,\F,\overline\P)\times\H^2(\R^{m\times d},\F,\overline\P)$ solves {\rm \Cref{eq:Gen BSDE}}, then we actually have $(P,Q)\in \mathbb{S}^\infty(\R^m,\F) \times \H^{2}_{\mathrm{BMO}}(\R^{m\times d},\F) $.
\end{lemma}

\begin{proof}
Fix some $\beta\geq 0$, and let $(\tau_n)_{n\in\N}$ be a sequence of $\F$--stopping times localising the local martingale $\int_0^\cdot\mathrm{e}^{\beta u}P_u\cdot Q_u\d \overline W_u$, and such that $P^{\tau_{\text{\fontsize{4}{4}\selectfont$n$}}}$ and $Z^{\tau_{\text{\fontsize{4}{4}\selectfont$n$}}}$ are bounded for any $n\in\N$. By It\^o's formula, we have for any $n\in\N$ large enough, and any $\tau\in\Tc(\F)$
 	\begin{align*}
 		&\mathrm{e}^{\beta \tau}\|P_\tau\|^2 +\int_\tau^{\tau_\smalltext{n}\wedge T}\mathrm{e}^{\beta u}\|Q_u\|^2\d u+\beta\int_\tau^{\tau_\smalltext{n}\wedge T}\mathrm{e}^{\beta u}\|P_u\|^2\d u\\
		&= \mathrm{e}^{\beta \tau_{\text{\fontsize{4}{4}\selectfont$n$}}\wedge T}\big\|P_{\tau_{\text{\fontsize{4}{4}\selectfont$n$}}\wedge T}\big\|^2 + 2\int_\tau^{\tau_{\text{\fontsize{4}{4}\selectfont$n$}}\wedge T}\mathrm{e}^{\beta u}P_u\cdot F_u\big(X_{\cdot\wedge u}, P_u, Q_u, \cL_{\overline\P}(X_{\cdot\wedge u},P_u,Q_u)\big) \d u - 2\int_\tau^{\tau_{\text{\fontsize{4}{4}\selectfont$n$}}\wedge T}\mathrm{e}^{\beta u}P_u\cdot Q_u\d \overline W_u\\
 		&\le  \mathrm{e}^{\beta \tau_{\text{\fontsize{4}{4}\selectfont$n$}}\wedge T}\big\|P_{\tau_{\text{\fontsize{4}{4}\selectfont$n$}}\wedge T}\big\|^2 + 2\int_\tau^{\tau_{\text{\fontsize{4}{4}\selectfont$n$}}\wedge T}\mathrm{e}^{\beta u}\|P_u\|\Big(\|F^0 \|_\infty+\ell_F\big(\|P_u\| + \|Q_u\| + \E^{\overline\P}\big[\|X_{\cdot\wedge u}\|_\infty^2+\|P_u\|^2+\|Q_u\|^2\big]^{\frac12} \big) \Big)\d u \\
		&\quad- 2\int_\tau^{\tau_{\text{\fontsize{4}{4}\selectfont$n$}}\wedge T}\mathrm{e}^{\beta u}P_u\cdot Q_u\d \overline W_u,
	\end{align*}
     	where we used the shorthand notation $\|F^0\|_\infty\coloneqq \sup_{t\in [0,T]} \|F_t(X_{\cdot\wedge t},0,0,\delta_{\{0,0\}})\|_{\L^\smalltext{\infty}(\R^\smalltext{m},\Fc_{\text{\fontsize{4}{4}\selectfont$T$}})}$. We thus deduce after using Young's inequality  that for any $\eps>0$
	\begin{align}\label{eq:estimatesinftybmo3}
		&\notag \mathrm{e}^{\beta \tau}\|P_\tau\|^2 +(1-2\eps)\int_\tau^{\tau_{\text{\fontsize{4}{4}\selectfont$n$}}\wedge T}\mathrm{e}^{\beta u}\|Q_u\|^2\d u+\bigg(\beta-\frac{1+2\ell_F^2}{\eps}-2\ell_F\bigg)\int_\tau^{\tau_{\text{\fontsize{4}{4}\selectfont$n$}}\wedge T}\mathrm{e}^{\beta u}\|P_u\|^2\d u\\
		&\notag\leq  \mathrm{e}^{\beta \tau_{\text{\fontsize{4}{4}\selectfont$n$}}\wedge T}\big\|P_{\tau_{\text{\fontsize{4}{4}\selectfont$n$}}\wedge T}\big\|^2+\eps\int_\tau^{\tau_{\text{\fontsize{4}{4}\selectfont$n$}}\wedge T}\mathrm{e}^{\beta u}\Big(\|F^0 \|_\infty^2 +  \E^{\overline\P}\big[\|X_{\cdot\wedge u}\|_\infty^2\big] + \E^{\overline\P}\big[\|P_u\|^2\big] + \E^{\overline\P}\big[\|Q_u\|^2\big]\Big)\d u-2\int_\tau^{\tau_{\text{\fontsize{4}{4}\selectfont$n$}}\wedge T}\mathrm{e}^{\beta u}P_u\cdot Q_u\d \overline W_u\\
		&\leq  \mathrm{e}^{\beta \tau_{\text{\fontsize{4}{4}\selectfont$n$}}\wedge T}\big\|P_{\tau_n\wedge T}\big\|^2+\frac{\eps }\beta\mathrm{e}^{\beta T}\|F^0 \|_\infty^2+\eps\E^{\overline \P}\bigg[\int_\tau^{\tau_{\text{\fontsize{4}{4}\selectfont$n$}}\wedge T}\mathrm{e}^{\beta u}\big(\|X_{\cdot\wedge u}\|_\infty^2+\|P_u\|^2+\|Q_u\|^2\big)\d u\bigg]-2\int_\tau^{\tau_{\text{\fontsize{4}{4}\selectfont$n$}}\wedge T}\mathrm{e}^{\beta u}P_u\cdot Q_u\d \overline W_u.
	\end{align}
	Taking conditional expectation in \Cref{eq:estimatesinftybmo3} under $\overline \P$, we deduce
	\begin{align}\label{eq:estimatesinftybmo4}
		&\notag \mathrm{e}^{\beta \tau}\|P_\tau\|^2 +\E^{\overline \P}\bigg[(1-2\eps)\int_\tau^{\tau_\smalltext{n}\wedge T}\mathrm{e}^{\beta u}\|Q_u\|^2\d u+\bigg(\beta-\frac{1+2\ell_F^2}{\eps}-2\ell_F\bigg)\int_\tau^{\tau_\smalltext{n}\wedge T}\mathrm{e}^{\beta u}\|P_u\|^2\d u\bigg|\Fc_\tau\bigg]\\
		\notag		&\leq   \E^{\overline \P}\big[\mathrm{e}^{\beta \tau_{\text{\fontsize{4}{4}\selectfont$n$}}\wedge T}\big\|P_{\tau_\smalltext{n}\wedge T}\big\|^2 \big|\cF_\tau\big]+\frac{\eps }\beta\mathrm{e}^{\beta T}\|F^0 \|_\infty^2+ \eps T\mathrm{e}^{\beta T}\E^{\overline\PP}\bigg[\sup_{t \in [0,T]}\|X_t\|^2\bigg] +\eps T\bigg\|\sup_{t\in[0,\tau_\smalltext{n}\wedge T]} \big|\mathrm{e}^{\frac{\beta}{2} t} P_{t}\big|^2\bigg\|_{\L^\smalltext{\infty}(\R,\Fc_{\text{\fontsize{4}{4}\selectfont$T$}})} \\
		&\quad+\eps\bigg\|\underset{\tau\in[0,\tau_{\text{\fontsize{4}{4}\selectfont$n$}}\wedge T]}{\rm essup}\; \E^{\overline \P}\bigg[\int_\tau^{\tau_{\text{\fontsize{4}{4}\selectfont$n$}}\wedge T}\mathrm{e}^{\beta s}\|Q_s\|^2\mathrm{d}s\bigg|\cF_\tau\bigg]\bigg\|_{\L^\smalltext{\infty}(\R,\Fc_{\text{\fontsize{4}{4}\selectfont$T$}})}.
	\end{align}
For $\eps$ small enough and $\beta$ large enough, we thus deduce that
\begin{align}\label{eq:estimatesinftybmo5}
\notag&(1-2\eps T)\bigg\|\sup_{t\in[0,\tau_{\text{\fontsize{4}{4}\selectfont$n$}}\wedge T]} \big|\mathrm{e}^{\frac{\beta}{2} t} P_{t}\big|^2\bigg\|_{\L^\smalltext{\infty}(\R,\Fc_{\text{\fontsize{4}{4}\selectfont$T$}})}+(1-4\eps)\bigg\|\underset{\tau\in[0,\tau_{\text{\fontsize{4}{4}\selectfont$n$}}\wedge T]}{\rm essup}\; \E^{\overline \P}\bigg[\int_\tau^{\tau_{\text{\fontsize{4}{4}\selectfont$n$}}\wedge T}\mathrm{e}^{\beta s}\|Q_s\|^2\mathrm{d}s\bigg|\cF_\tau\bigg]\bigg\|_{\L^\smalltext{\infty}(\R,\Fc_{\text{\fontsize{4}{4}\selectfont$T$}})}\\
&\leq  2\E^{\overline \P}\Big[\mathrm{e}^{\beta \tau_{\text{\fontsize{4}{4}\selectfont$n$}}\wedge T}\big\|P_{\tau_{\text{\fontsize{4}{4}\selectfont$n$}}\wedge T}\big\|^2\Big|\Fc_\tau\Big]+2\frac{\eps }\beta\mathrm{e}^{\beta T}\|F^0 \|_\infty^2+ 2\eps T\mathrm{e}^{\beta T}\E^{\overline\PP}\bigg[\sup_{t \in [0,T]}\|X_t\|^2\bigg].
\end{align}
Since $B$ is bounded and $\sigma$ of linear growth, we can easily show that
	\begin{equation}\label{eq:cpntrolxx}
		\E^{\overline\PP}\bigg[\sup_{t \in [0,T]}\|X_t\|^2\bigg]\le 
		4\e^{2\ell_\smalltext{\sigma}^\smalltext{2}T}\Big(\|X_0\|^2 + T^2\|B\|_\infty^2 + 2T\ell_\sigma^2\|\sigma_\cdot(0)\|^2_\infty \Big)\eqqcolon  C_X.
	\end{equation}
	Hence, we can expectations under $\overline \P$ in \Cref{eq:estimatesinftybmo5}, then use Fatou's lemma, the dominated convergence theorem and continuity of $P$, to let $n$ got to $+\infty$ and deduce
	\[
		(1-2\eps T)\big\|\mathrm{e}^{\beta \cdot}\delta P_{\cdot}\big\|^2_{\S^\smalltext{\infty}(\R^\smalltext{m},\F)} +(1-4\eps)\|Q\|^2_{\H^{\smalltext{2}\smalltext{,}\smalltext{\beta}}_{\text{\fontsize{4}{4}\selectfont$\mathrm{BMO}$}}(\R^{\smalltext{m}\smalltext{\times} \smalltext{d}},\F,\overline\P)}\leq 2\mathrm{e}^{\beta T}\|G\|^2_{\L^\smalltext{\infty}(\R^\smalltext{m},\Fc_{\text{\fontsize{4}{4}\selectfont$T$}})}+2\frac{\eps }\beta\mathrm{e}^{\beta T}\|F^0 \|_\infty^2+2\eps T\mathrm{e}^{\beta T}C_X.
	\]
	
It then suffices to recall, see for instance \cite[Lemma A.1]{herdegen2021equilibrium}, that the norms on $\H^{2,\beta}_{\mathrm{BMO}}(\R^{m\times d},\F,\overline\P)$ and $\H^{2,\beta}_{\mathrm{BMO}}(\R^{m\times d},\F)$ are equivalent since $B$ is bounded.
\end{proof}

\begin{proof}[Proof of \Cref{thm:Gen.MkV.BSDE}]
\emph{Step 1: reduction to Lipschitz-continuous generator $F$.}

	\medskip
	In this first step, observe that if {\rm\Cref{assum:gen.MkV}}.$(iia)$ is satisfied, then any solution $(P,Q)$ of the BSDE \eqref{eq:Gen BSDE} will have a bounded $Q$.
	In fact, by \cite{ankirchner2007classical}, $(P,Q)$ is Malliavin differentiable and $Q$ is the trace of the Malliavin derivative of $P$, \emph{i.e.} (denoting by $D_sP$, $D_sQ$ the Malliavin derivative in the direction of $\overline W$ of $P$ and $Q$ respectively) we have $Q_t = D_tP_t$, $\mathrm{d}t\otimes \mathrm{d}\overline\P$--a.e. Moreover, for any $s\in[0,T]$, $(D_sP, D_sQ)$ satisfies
	\begin{equation}
	\label{eq:Mall.deriv.equation}
		D_sP_t = D_sG + \int_t^T\big(D_sF_u(X_{\cdot\wedge u}, P_u,Q_u, \cL_{\overline\P}(X_{\cdot\wedge u}, P_u, Q_u)) + \partial_yF_uD_sP_u + \partial_zF_uD_sQ_u\big)\mathrm{d}u - \int_t^TD_sQ_u\mathrm{d}\overline W_u,\; t\in[0,T].
	\end{equation}
	Applying It\^o's formula to $\|D_sP_t\|^2$ and then Young's inequality yields, for every $\tau\in\Tc(\F)$ and every $i\in\{1,\dots,N\}$
	\begin{align*}
		&\|D_s^iP^n_\tau\|^2 + \E^{\hat\P}\bigg[\int_\tau^T\|D_s^iQ^n_u\|^2\d u \bigg|\cF_\tau \bigg]\\
		& \le \E^{\hat\P}\bigg[\|D_s^iG^n\|^2 +\int_\tau^T \bigg(\|D_s^iF_u^n(X_{\cdot\wedge u}, P_u,Q_u, \cL_{\overline\P}(X_{\cdot\wedge u}, P_u, Q_u))\|^2 + \Big(1 + 2 \|\partial_yF^n\|^2 \Big)\|D_s^iP_u^n\|^2\bigg)\d u\bigg |\cF_\tau \bigg],
	\end{align*}
	where $D^i$ is the derivative in the direction of the $i$-th coordinate of the Brownian motion $\overline W$, the superscript $n\in \{1,\dots,m\}$ means the $n$-th coordinate of the corresponding vector, and $\diff \hat\P = \cE\big(\int_0^\cdot \partial_zF_u\diff \overline W_u\big)_T\diff \overline\P$.
	Therefore, applying Gronwall's inequality and choosing $\eps<1$, say $\varepsilon = 1/2$, yields
	\begin{align}
	\label{eq:bound.DY.bmo.DZ}
		\|D_s^iP^n\|^2_{\S^\smalltext{\infty}(\R,\F)} + \|D_s^iQ^n\|^2_{\H^\smalltext{2}_{\text{\fontsize{4}{4}\selectfont BMO}}(\R^{ \smalltext{d}},\F,\hat\P)}\le 2\mathrm{e}^{(1 + 2\ell^\smalltext{2}_\smalltext{F})T}\Big(\|DG^i\|_\infty^2 + T\|DF^n\|_\infty^2\Big)\eqqcolon  \ell_z.
	\end{align}
In particular, $Q$ is bounded by $L_z \coloneqq  \sqrt{m\ell_z}$.
Therefore, $(P,Q)$ also solves \rm\Cref{eq:Gen BSDE} with $F$ therein replaced by 
	\begin{equation*}
		F^\prime_t(\x, y, z, \xi) \coloneqq  \begin{cases}
		F_t(\x, y, z, \xi),\; \text{if}\; \|z\|\le \ell_z,\\
	F_t\big(\x, y, \ell_q\frac{z}{\|z\|}, \xi\big),\; \text{if}\; \|z\|> \ell_z,
	\end{cases}
	\end{equation*}
	which now is Lipschitz-continuous with respect to the variables $(\x, y, z, \xi)$, uniformly in $t$. Hence, we can assume without loss of generality that $F$ is Lipschitz-continuous, so that in the rest of the proof we assume that \rm\Cref{assum:gen.MkV}.$(iib)$ is satisfied.
		
\medskip

	\emph{Step 2: introduction of the solution mapping $\Phi$.} Let us define the space of flows of probability measures $\mathfrak P^{m\times d}$, consisting of all Borel-measurable maps $[0,T]\ni t\longmapsto \xi_t\in \Pc_2(\Cc_m\times\R^m\times\R^{m\times d})$ such that $\int_0^T\int_{\Cc_\smalltext{m}\times\R^\smalltext{m}\times\R^{\smalltext{m}\smalltext{\times} \smalltext{d}}}\big(\|\bx\|_\infty^2+\|y\|^2+\|z\|^2\big)\xi_t(\d\bx,\d y,\d z)\d t<\infty$, which we equip with the distance
	\begin{equation*}
 		\cW_{2,\beta,[0,T]}(\xi, \xi^\prime) \coloneqq  \bigg( \int_0^T\mathrm{e}^{\beta t}\cW_2^2(\xi_t, \xi^\prime_t)\d t \bigg)^{1/2},
 	\end{equation*} 
 	defined for any $\beta >0$. 
 	Since $(\Pc_2(\Cc_m\times\R^m\times\R^{m\times d}), \Wc_2)$ is a complete metric space, it is easily verified that the distance $\cW_{2,\beta,[0,T]}(\cdot,\cdot)$ makes $\mathfrak P^{m\times \bmdim}$ a complete metric space as well. Further define for any $\beta>0$ the space $\H^{2,\beta}_{\mathrm{BMO}}(\R^{m\times d},\F)$ as the space of processes $Z \in \H^{2}_{\mathrm{BMO}}(\R^{m\times d},\F)$ with the norm $\|z\|_{\H^{\smalltext{2}\smalltext{,}\smalltext{\beta}}_{\text{\fontsize{4}{4}\selectfont$\mathrm{BMO}$}}(\R^{\smalltext{m}\smalltext{\times} \smalltext{d}},\F)}\coloneqq  \|\mathrm{e}^{\beta/2 \cdot}z_\cdot\|_{\H^\smalltext{2}_{\text{\fontsize{4}{4}\selectfont$\mathrm{BMO}$}}(\R^{\smalltext{m}\smalltext{\times} \smalltext{d}},\F)}$.

	\medskip
	For the rest of the proof we fix some $\beta>0$ which will be specified below. Let now $(y,z,\xi) \in \mathbb{S}^\infty(\R^m,\F) \times \H^{2,\beta}_{\mathrm{BMO}}(\R^{m\times d},\F) \times \mathfrak P^{m\times \bmdim}$ be given, and denote by $\PP^{z,\xi}$ the probability measure with density
	\begin{equation*}
		\frac{\mathrm{d}\PP^{z,\xi}}{\mathrm{d}\PP} = \cE\bigg(\int_0^TB_u\big(X_{\cdot\wedge u},z_u, \xi_u\big)\cdot \mathrm{d}W_u \bigg).
	\end{equation*}
Consider then the following (standard) McKean--Vlasov BSDE
	\begin{equation}
	\label{eq:standard.MkV}
		Y_t = G\big(X, \cL_{\PP^{z,\xi}}(X)\big) + \int_t^TF_u\big(X_{\cdot\wedge u},Y_u,Z_u, \cL_{\PP^{z,\xi}}(X_{\cdot\wedge u},Y_u,Z_u)\big)\mathrm{d}u - \int_t^TZ_u\mathrm{d}W^{z,\xi}_u,\; t\in[0,T],\;  \PP^{z,\xi}\text{\rm--a.s.},
	\end{equation}
where $\cL_{\PP^{z, \xi}}(X_{\cdot\wedge u},Y_u,Z_u) \coloneqq  \PP^{z,\xi}\circ (X_{\cdot\wedge u},Y_u,Z_u)^{-1},$ and $W^{z,\xi} \coloneqq  W - \int_0^{\cdot} B_u(X_{\cdot\wedge u},z_u, \xi_u)\mathrm{d}u.$ By \Cref{prop:MckV BSDE}, the BSDE \eqref{eq:standard.MkV} admits a unique solution $(Y, Z) \in \SS^2(\RR^m,\FF, \P^{z, \xi})\times \H^2(\RR^{m\times d},\FF, \P^{z, \xi})$.	

	\medskip
	We now denote by $\Phi$ the functional mapping which associates to any $(y,z,\xi)\in  \mathbb{H}^2(\R^m,\F) \times \H^{2,\beta}_{\mathrm{BMO}}(\R^{m\times d},\F) \times \mathfrak P^{m\times \bmdim}$ the triplet $\big(Y, Z, (\cL_{\PP^{z,\xi}}(X_{\cdot\wedge t},Y_t,Z_t))_{t\in[0,T]}\big)$, where $(Y,Z)$ solves the McKean--Vlasov BSDE \eqref{eq:standard.MkV}.
	Our goal is to show that $\Phi$ admits a unique fixed-point in $ \mathbb{S}^\infty(\R^m,\F) \times \H^{2{\color{black},\beta}}_{\mathrm{BMO}}(\R^{m\times d},\F) \times \mathfrak P^{m\times \bmdim}$.

\medskip

	\emph{Step 3: the solution mapping $\Phi$ is well-defined.}

	\medskip
We will show that for every $(y,z,\xi) \in  \mathbb{S}^\infty(\R^m,\F) \times \H^{2,\beta}_{\mathrm{BMO}}(\R^{m\times d},\F) \times \mathfrak P^{m\times \bmdim}$, it holds that $\big(Y, Z, (\cL_{\PP^{z,\xi}}(X_{\cdot\wedge t},Y_t,Z_t))_{t\in[0,T]} \big) \in  \mathbb{S}^\infty(\R^m,\F) \times \H^{2,\beta}_{\mathrm{BMO}}(\R^{m\times d},\F) \times \mathfrak P^{m\times \bmdim}$.
	We already know that $Z$ is bounded in the BMO norm, we will show that $Y$ is bounded as well and derive a finer bound for the BMO norm of $Z$ that will be used later in the proof.
	In fact, by It\^o's formula, we have for any $\beta\geq 0$, $\tau\in\Tc(\F)$, and any $\eps>0$
 	\begin{align*}
 		&\mathrm{e}^{\beta \tau}\|Y_\tau\|^2 +\int_\tau^T\mathrm{e}^{\beta u}\|Z_u\|^2\d u+\beta\int_\tau^T\mathrm{e}^{\beta u}\|Y_u\|^2\d u\\
		&= \mathrm{e}^{\beta T}\big\|G(X, \cL_{\PP^{z,\xi}}(X))\big\|^2 + 2\int_\tau^T\mathrm{e}^{\beta u}\Big(Y_u\cdot F_u\big(X_{\cdot\wedge u}, Y_u, Z_u, \cL_{\P^{z,\xi}}(X_{\cdot\wedge u},Y_u,Z_u)\big) + Y_u\cdot Z_uB_u(X_{\cdot \wedge u},z_u,\xi_u)\Big)\d u\\
		&\quad - 2\int_\tau^T\mathrm{e}^{\beta u}Y_u\cdot Z_u\d W_u\\
 		&\le \mathrm{e}^{\beta T}\big\|G(X, \cL_{\PP^{z,\xi}}(X))\big\|^2 + 2\int_\tau^T\mathrm{e}^{\beta u}\|Y_u\|\Big(\|F^0 \|_\infty+\|B\|_\infty\|Z_u\|+\ell_F\big(\|Y_u\| + \|Z_u\| + \E^{\P^{z,\xi}}\big[\|X_{\cdot\wedge u}\|_\infty^2+\|Y_u\|^2+\|Z_u\|^2\big]^{\frac12} \big) \Big)\d u \\
		&\quad- 2\int_\tau^T\mathrm{e}^{\beta u}Y_u\cdot Z_u\d W_u,
	\end{align*}
 	where we used the shorthand notation $\|F^0\|_\infty\coloneqq \sup_{t\in [0,T]} \|F_t(X_{\cdot\wedge t},0,0,\delta_{\{0,0\}})\|_{\L^\infty(\R^m,\Fc_T)}$. 
 	We thus deduce after using Young's inequality and \citeauthor*{herdegen2021equilibrium} \cite[Lemma A.1]{herdegen2021equilibrium} that for any $\eps>0$
	\begin{align}\label{eq:estimatesinftybmo}
		&\notag \mathrm{e}^{\beta \tau}\|Y_\tau\|^2 +(1-2\eps)\int_\tau^T\mathrm{e}^{\beta u}\|Z_u\|^2\d u+\bigg(\beta-\frac{1+\|B\|_\infty^2+2\ell_F^2}{\eps}-2\ell_F\bigg)\int_\tau^T\mathrm{e}^{\beta u}\|Y_u\|^2\d u\\
		&\notag\leq \mathrm{e}^{\beta T}\big\|G(X, \cL_{\PP^{z,\xi}}(X))\big\|^2+\eps\int_\tau^T\mathrm{e}^{\beta u}\Big(\|F^0 \|_\infty^2 +  \E^{\P^{z,\xi}}\big[\|X_{\cdot\wedge u}\|_\infty^2\big] + \E^{\P^{z,\xi}}\big[\|Y_u\|^2\big] + \E^{\P^{z,\xi}}\big[\|Z_u\|^2\big]\Big)\d u-2\int_\tau^T\mathrm{e}^{\beta u}Y_u\cdot Z_u\d W_u\\
		&\leq  \mathrm{e}^{\beta T}\|G\|^2_{\L^\smalltext{\infty}(\R^\smalltext{m},\Fc_{\text{\fontsize{4}{4}\selectfont$T$}})}+\frac{\eps }\beta\mathrm{e}^{\beta T}\|F^0 \|_\infty^2+\eps\E^{\P^{z,\xi}}\bigg[\int_0^T\mathrm{e}^{\beta u}\big(\|X_{\cdot\wedge u}\|_\infty^2+\|Y_u\|^2+\|Z_u\|^2\big)\d u\bigg]-2\int_\tau^T\mathrm{e}^{\beta u}Y_u\cdot Z_u\d W_u\notag\\
	\notag	&\leq\mathrm{e}^{\beta T}\|G\|^2_{\L^\smalltext{\infty}(\R^\smalltext{m},\Fc_{\text{\fontsize{4}{4}\selectfont$T$}})}+\frac{\eps }\beta\mathrm{e}^{\beta T}\|F^0 \|_\infty^2 + \eps T\mathrm{e}^{\beta T}\E^{\PP^{z,\xi}}\bigg[\sup_{t \in [0,T]}\|X_t\|^2\bigg]+\eps T \big\|\mathrm{e}^{\frac{\beta }{2}\cdot} Y_{\cdot}\big\|^2_{\S^\smalltext{\infty}(\R^\smalltext{m},\F)} +8\eps\big(1+\|B\|_\infty\big)^2\|Z\|^2_{\H^{\smalltext{2}\smalltext{,}\smalltext{\beta}}_{\text{\fontsize{4}{4}\selectfont$\mathrm{BMO}$}}(\R^{\smalltext{m}\smalltext{\times} \smalltext{d}},\F)}\\
		&\quad -2\int_\tau^T\mathrm{e}^{\beta u}Y_u\cdot Z_u\d W_u.
	\end{align}
	Taking conditional expectation in \Cref{eq:estimatesinftybmo}, we deduce
	\begin{align}\label{eq:estimatesinftybmo2}
		&\notag \mathrm{e}^{\beta \tau}\|Y_\tau\|^2 +(1-2\eps)\E\bigg[\int_\tau^T\mathrm{e}^{\beta u}\|Z_u\|^2\d u+\bigg(\beta-\frac{1+\|B\|_\infty^2+2\ell_F^2}{\eps}-2\ell_F\bigg)\int_\tau^T\mathrm{e}^{\beta u}\|Y_u\|^2\d u\bigg|\Fc_\tau\bigg]\\
		&\leq  \mathrm{e}^{\beta T}\|G\|^2_{\L^\smalltext{\infty}(\R^\smalltext{m},\Fc_{\text{\fontsize{4}{4}\selectfont$T$}})}+\frac{\eps }\beta\mathrm{e}^{\beta T}\|F^0 \|_\infty^2+ \eps T\mathrm{e}^{\beta T}\E^{\PP^{z,\xi}}\bigg[\sup_{t \in [0,T]}\|X_t\|^2\bigg] +\eps T \big\|\mathrm{e}^{\frac{\beta}{2} \cdot} Y_{\cdot}\big\|^2_{\S^\smalltext{\infty}(\R^\smalltext{m},\F)} +8\eps\big(1+\|B\|_\infty\big)^2\|Z\|^2_{\H^{\smalltext{2}\smalltext{,}\smalltext{\beta}}_{\text{\fontsize{4}{4}\selectfont$\mathrm{BMO}$}}(\R^{\smalltext{m}\smalltext{\times} \smalltext{d}},\F)}.
	\end{align}
	Since $B$ is bounded and $\sigma$ of linear growth, we can easily show that
	\begin{equation*}
		\E^{\PP^{z,\xi}}\bigg[\sup_{t \in [0,T]}\|X_t\|^2\bigg]\le 
		C_X,
	\end{equation*}
	with $C_X$ given in \Cref{eq:cpntrolxx}. We therefore obtain that for $\beta$ large enough
	\begin{align*}
		(1-2\eps T)\big\|\mathrm{e}^{\beta \cdot}\delta Y_{\cdot}\big\|^2_{\S^\smalltext{\infty}(\R^\smalltext{m},\F)} & +\big(1-2\eps-16\eps(1+\|B\|_\infty)^2\big)\|Z\|^2_{\H^{\smalltext{2}\smalltext{,}\smalltext{\beta}}_{\text{\fontsize{4}{4}\selectfont$\mathrm{BMO}$}}(\R^{\smalltext{m}\smalltext{\times} \smalltext{d}},\F)}\leq 2\mathrm{e}^{\beta T}\|G\|^2_{\L^\smalltext{\infty}(\R^\smalltext{m},\Fc_{\text{\fontsize{4}{4}\selectfont$T$}})}+\frac{2\eps }\beta\mathrm{e}^{\beta T}\|F^0 \|_\infty^2 + 2\eps T\mathrm{e}^{\beta T}C_X.
	\end{align*}
	This shows that $Y$ is bounded and $Z$ belongs to the space $\H^{2,\beta}_{\mathrm{BMO}}(\R^{m\times d},\F)$.
	In addition, we have for $\eps$ small enough
	\begin{align}
	\label{eq:bound.C.tgf}
		\|Z\|^2_{\H^{\smalltext{2}\smalltext{,}\smalltext{\beta}}_{\text{\fontsize{4}{4}\selectfont$\mathrm{BMO}$}}(\R^{\smalltext{m}\smalltext{\times} \smalltext{d}},\F)}\leq \frac{2\mathrm{e}^{\beta T}}{\ell_\eps}\bigg(\|G\|^2_{\L^\smalltext{\infty}(\R^\smalltext{m},\Fc_{\text{\fontsize{4}{4}\selectfont$T$}})}+\frac{\eps }\beta\|F^0 \|_\infty^2 + \eps TC_X\bigg) \eqqcolon   C^\eps_{F,G,T},
	\end{align} 
	where $
		\ell_\eps\coloneqq  1-2\eps-16\eps(1+\|B\|_\infty)^2.$ It remains to check that $(\cL_{\PP^{z,\xi}}(X_{\cdot\wedge t},Y_t,Z_t))_{t\in[0,T]}\in\mathfrak P^{m\times d}$. 
	But this is true since \Cref{eq:cpntrolxx} holds, $(Y, Z) \in \SS^2(\RR^m,\FF, \P^{z, \xi})\times \H^2(\RR^{m\times d},\FF, \P^{z, \xi})$, and using inequalities similar to \eqref{eq:controlwasser}.
	
	\medskip

	\emph{Step 4: the mapping $\Phi$ is a contraction for $G$ small enough.}

	\medskip

	Fix for $i\in\{1,2\}$, $(y^i, z^i,\xi^i) \in \mathbb{H}^2(\R^m,\F) \times \mathbb{H}^{2,\beta}_{\mathrm{BMO}}(\R^{m\times d},\F) \times \mathfrak P^{m\times \bmdim}$, set $\Phi(y^i, z^i, \xi^i) \eqqcolon   \big(Y^i, Z^i, (\cL_{\PP^{z^i, \xi^i}}(X^i_{\cdot\wedge t},Y^i_t,Z^i_t))_{t\in[0,T]} \big)$, and let $\delta Y\coloneqq  Y^1 - Y^2$, $\delta Z = Z^1 - Z^2$, $\delta y\coloneqq  y^1 - y^2$, $\delta z = z^1 - z^2$ and $\delta G\coloneqq  G\big(X, \cL_{\P^{z^{\text{\fontsize{4}{4}\selectfont$1$}}, \xi^{\text{\fontsize{4}{4}\selectfont$1$}}}}(X)\big)- G\big(X, \cL_{\P^{z^{\text{\fontsize{4}{4}\selectfont$2$}}, \xi^{\text{\fontsize{4}{4}\selectfont$2$}}}}(X)\big)$. It follows by Girsanov's theorem that these processes satisfy 
	\begin{align*}
		\delta Y_t& = \delta G+ \int_t^T\Big(F_u\big( X_{\cdot\wedge u}, Y^1_u, Z^1_u,\cL_{\P^{\smalltext{z}^{\tinytext{1}}\smalltext{,} \smalltext{\xi}^{\tinytext{1}}}}(X_{\cdot\wedge u},Y^1_u, Z^1_u) \big) - F_u\big( X_{\cdot\wedge u}, Y^2_u, Z^2_u,\cL_{\P^{\smalltext{z}^{\tinytext{2}}\smalltext{,} \smalltext{\xi}^{\tinytext{2}}}}(X_{\cdot\wedge u},Y^2_u, Z^2_u) \big) \Big)\d u\\
		&\quad
		+ \int_t^TZ^1_u \big(B_u(X_{\cdot\wedge u},z_u^1, \xi_u^1) - B_u(X_{\cdot\wedge u},z_u^2, \xi_u^2)\big)\d u  +\int_t^T B_u(X_{\cdot\wedge u},z_u^2, \xi_u^2)\big(Z^1_u - Z^2_u\big)\d u - \int_t^T\big(Z^2_u - Z^1_u\big)\d W_u.
	\end{align*}
	Let $(\beta,\eps)\in(0,\infty)^2$ be fixed and apply It\^o's formula to $\big(\mathrm{e}^{\beta t}\|\delta Y_t\|^2\big)_{t\in[0,T]}$ to get
	\begin{align*}
		\mathrm{e}^{\beta t} \|\delta Y_t\|^2  &= {\mathrm{e}^{\beta T}}\|\delta G\|^2 + \int_t^T2\mathrm{e}^{\beta u}\delta Y_u\cdot \Big(F_u\big(X_{\cdot\wedge u},  Y^1_u, Z^1_u, \cL_{\P^{\smalltext{z}^{\tinytext{1}}\smalltext{,} \smalltext{\xi}^{\tinytext{1}}}}(X_{\cdot\wedge u},Y^1_u,Z^1_u)\big) - F_u\big(X_{\cdot\wedge u},Y^2_u, Z^2_u, \cL_{\P^{\smalltext{z}^{\tinytext{2}}\smalltext{,} \smalltext{\xi}^{\tinytext{2}}}}(X_{\cdot\wedge u},Y^2_u,Z^2_u)\big) \Big)\mathrm{d}u\\
		&\quad +\int_t^T2\mathrm{e}^{\beta u}\delta Y_u\cdot\Big(Z^1_u \big(B_u(X_{\cdot\wedge u},z_u^1, \xi_u^1) - B_u(X_{\cdot\wedge u},z_u^2, \xi_u^2)\big) +  B_u(X_{\cdot\wedge u},z_u^2, \xi_u^2)(Z^1_u - Z^2_u) \Big)\d u\\
		&\quad -\int_t^T \beta \mathrm{e}^{\beta u}\|\delta Y_u\|^2\mathrm{d}u -  \int_t^T\mathrm{e}^{\beta u}\|\delta Z_u\|^2\mathrm{d}u- \int_t^T2\mathrm{e}^{\beta u}\delta Y_u\cdot \delta Z_u \mathrm{d}W_u.
	\end{align*}
	
	Since $\delta Y$ is bounded and $\delta Z$ is in $\H^{2}_{\mathrm{BMO}}(\R^m,\F)$, it follows that $\int_0^t2\mathrm{e}^{\beta u}\delta Y_u\cdot \delta Z_u \mathrm{d}W_u$ is a true $(\F,\P)$-martingale.
	Thus, taking conditional expectation on both sides, using Lipschitz-continuity of $F$ and $G$, and applying Young's inequality, we have for any $\varepsilon>0$ and any $\tau\in\Tc(\F)$
	\begin{align}
	\notag
		& \mathrm{e}^{\beta \tau}\|\delta Y_\tau\|^2  + (1-2\eps)\E\bigg[\int_\tau^T\mathrm{e}^{\beta u}\|\delta Z_u\|^2\mathrm{d}u \bigg| \mathcal{F}_\tau\bigg]  + \bigg(\beta -\frac{\|B\|^2_\infty}{\eps}-2\ell_F\bigg(1+\frac{2\ell_F}{\eps}\bigg) \bigg)\E\bigg[\int_\tau^T\mathrm{e}^{\beta u}\|\delta Y_u\|^2\mathrm{d}u\bigg| \mathcal{F}_\tau \bigg]\\\notag
		&\leq {\mathrm{e}^{\beta T}}\ell_{G(\mu)}^2\cW_2^2\big( \cL_{\P^{\smalltext{z}^{\tinytext{1}}\smalltext{,} \smalltext{\xi}^{\tinytext{1}}}}(X), \cL_{\P^{\smalltext{z}^{\tinytext{2}}\smalltext{,} \smalltext{\xi}^{\tinytext{2}}}}(X) \big) + \eps\int_\tau^T\mathrm{e}^{\beta u} \cW_2^2\big( \cL_{\P^{\smalltext{z}^{\tinytext{1}}\smalltext{,} \smalltext{\xi}^{\tinytext{1}}}}(X_{\cdot\wedge u},Y^1_u,Z^1_u), \cL_{\P^{\smalltext{z}^{\tinytext{2}}\smalltext{,} \smalltext{\xi}^{\tinytext{2}}}}(X_{\cdot\wedge u},Y^2_u,Z^2_u) \big) \mathrm{d}u\\
		&\quad + \E\bigg[\frac{1}{\eps}\int_\tau^T\mathrm{e}^{\beta u}\|\delta Y_u Z^1_u\|^2\d u + 2\eps\ell_B^2\int_\tau^T\mathrm{e}^{\beta u} \big(\cW_2^2(\xi^1_u, \xi^2_u) + \|\delta z_u\|^2\big)\d u\bigg|\cF_\tau\bigg].
		\label{eq:exits.first.estim}
	\end{align}
	Let us take a closer look at the terms $\cW_2\big(\cL_{\P^{\smalltext{z}^{\tinytext{1}}\smalltext{,} \smalltext{\xi}^{\tinytext{1}}}}(X_{\cdot\wedge u},Y^1_u,Z^1_u), \cL_{\P^{\smalltext{z}^{\tinytext{2}}\smalltext{,} \smalltext{\xi}^{\tinytext{2}}}}(X_{\cdot \wedge u},Y^2_u,Z^2_u) \big)$ and $\cW_2^2\big(\cL_{\P^{\smalltext{z}^{\tinytext{1}}\smalltext{,} \smalltext{\xi}^{\tinytext{1}}}}(X),\cL_{\P^{\smalltext{z}^{\tinytext{2}}\smalltext{,} \smalltext{\xi}^{\tinytext{2}}}}(X)\big)$ on the right-hand side of \eqref{eq:exits.first.estim}.
	By Kantorovich's duality (see \emph{e.g.} \citeauthor*{villani2009optimal} \cite[Theorem 5.10]{villani2009optimal}), it holds for Lebesgue--a.e. $t\in[0,T]$
	\begin{align*}
		&\cW_2^2\big( \cL_{\P^{\smalltext{z}^{\tinytext{1}}\smalltext{,} \smalltext{\xi}^{\tinytext{1}}}}(X_{\cdot\wedge t},Y^1_t,Z^1_t),\cL_{\P^{\smalltext{z}^{\tinytext{2}}\smalltext{,} \smalltext{\xi}^{\tinytext{2}}}}(X_{\cdot \wedge t},Y^2_t,Z^2_t) \big) \\
		& = \sup_{(f, g)\in\Mc}\bigg\{\int_{\Cc_\smalltext{m}\times\R^{\smalltext{m}}\times \R^{\smalltext{m}\smalltext{\times}\smalltext{ d}}} f(x)\d\cL_{\P^{\smalltext{z}^{\tinytext{1}}\smalltext{,} \smalltext{\xi}^{\tinytext{1}}}}(X_{\cdot\wedge t},Y^1_t,Z^1_t)(x) + \int_{\Cc_\smalltext{m}\times\R^{\smalltext{m}}\times\R^{\smalltext{m}\smalltext{\times} \smalltext{d}}} g(x)\d\cL_{\P^{\smalltext{z}^{\tinytext{2}}\smalltext{,} \smalltext{\xi}^{\tinytext{2}}}}(X_{\cdot \wedge t},Y^2_t,Z^2_t)(x)\bigg\}\\
		&= \sup_{(f,g)\in\Mc}\Big\{ \EE^{\PP^{\smalltext{z}^{\tinytext{1}}\smalltext{,} \smalltext{\xi}^{\tinytext{1}}}}\big[f(X_{\cdot\wedge t}, Y^1_t,Z^1_t)\big] + \EE^{\PP^{\smalltext{z}^{\tinytext{2}}\smalltext{,} \smalltext{\xi}^{\tinytext{2}}}}\big[g(X_{\cdot\wedge t},Y^2_t,Z^{2}_t)\big] \Big\}.
	\end{align*}
	where the supremum is over the set $\Mc$ of all bounded continuous functions $f$ and $g$ from $\Cc_m\times \R^{m}\times\R^{m\times d}$ to $\R$, such that $f(\x,y,z) + g(\x^\prime,y^\prime,z^\prime) \le \|\x-\x^\prime\|_\infty^2+\|y-y^\prime\|^2+\|z-z^\prime\|^2$, for all $(\x,y,z,\x^\prime,y^\prime,z^\prime) \in \big(\Cc_m\times\R^{m}\times \R^{m\times d}\big)^2$.
	
	\medskip
	{\color{black}
Observe that $\P^{z^\smalltext{1}, \xi^\smalltext{1}}\circ(X, Y^1, Z^1)^{-1} = \P\circ (\bar X^1, \bar Y^1, \bar Z^1)^{-1} $ and $\P^{z^\smalltext{2}, \xi^\smalltext{2}}\circ(X, Y^2, Z^2)^{-1} = \P\circ (\bar X^2, \bar Y^2, \bar Z^2)^{-1} $
 where for $i\in\{1,2\}$, the processes $(\bar X^i, \bar Y^i, \bar Z^i)$ solve the FBSDE
 \begin{equation*}
 	\begin{cases}
 		\d \bar X^i_t = B_t(\bar X^i_{\cdot\wedge t}, z^i_t, \xi^i_t)\d t + \sigma_t(\bar X^{i}_{\cdot\wedge t})\d W_t,\\
 		\d \bar Y^i_t = -F_t(\bar X^i_t, \bar Y^i_t, \bar Z^i_t, \cL_{\P^{\smalltext{z}^\tinytext{i}\smalltext{,} \smalltext{\xi}^\tinytext{i}}}(X_{\cdot\wedge t}, Y_t, Z_t))\d t + \bar Z^i_t\d W_t,\\
 		\bar Y^i_T = G\big(\bar X^i, \cL_{\P^{\smalltext{z}^\tinytext{i}\smalltext{,} \smalltext{\xi}^\tinytext{i}}}(X)),\; \bar X^i_0 = X_0,\; \P\text{\rm--a.s.}
 	\end{cases}
 \end{equation*}
 This is a standard---non McKean--Vlasov---FBSDE which admits a unique solution by assumption.
 Therefore, we have 
	\begin{align}
	\label{eq:estim.newspace.time}
		\cW_2^2\big( \cL_{\PP^{\smalltext{z}^{\tinytext{1}}\smalltext{,} \smalltext{\xi}^{\tinytext{1}}}}(X_{\cdot\wedge t},Y^1_t,Z^1_t)&, \cL_{\PP^{\smalltext{z}^{\tinytext{2}}\smalltext{,} \smalltext{\xi}^{\tinytext{2}}}}(X_{\cdot\wedge t},Y^2_t,Z^2_t) \big) 
		\le  \EE^{\PP}\big[\|\bar X_{\cdot\wedge t}^1 - \bar X_{\cdot\wedge t}^2\|_\infty^2+ \|\bar Y^1_t - \bar Y^2_t\|^2+ \|\bar Z^1_t - \bar Z^2_t\|^2\big] .
	\end{align}
	Similarly, we have
	\begin{align}
	\label{eq:estim.newspace}
		\cW_2^2\big( \cL_{\PP^{\smalltext{z}^{\tinytext{1}}\smalltext{,} \smalltext{\xi}^{\tinytext{1}}}}(X), \cL_{\PP^{\smalltext{z}^{\tinytext{2}}\smalltext{,} \smalltext{\xi}^{\tinytext{2}}}}(X) \big) 
		\le  \EE^{\PP}\big[\|\bar X^1 - \bar X^2\|_\infty^2\big] .
	\end{align}
To estimate the differences on the right-hand sides of \eqref{eq:estim.newspace.time} and \eqref{eq:estim.newspace}, we first apply It\^o's formula to $\|\delta \bar X_t\|^2 = \|\bar X^1_t - \bar X^2_t\|^2$ and using Young's inequality and Lipschitz-continuity of $B$ and $\sigma$, we obtain for any $\varepsilon>0$ and any $t\in[0,T]$
\begin{align*}
	 \|\delta \bar X_t\|^2 & \leq \bigg(\frac{\ell^2_B}{\varepsilon} + \ell_\sigma^2 - K_B\bigg)\int_0^t\|\delta\bar X_{\cdot\wedge s}\|^2\mathrm{d}s + \varepsilon\int_0^t\big( \|\delta z_s\|^2 + \cW_2^2(\xi^1_s, \xi^2_s)\big)\d s + 2\int_0^t\delta \bar X_{\cdot\wedge s}\cdot\big(\sigma_s(\bar X^1_{\cdot\wedge s}) - \sigma_s(\bar X^2_{\cdot\wedge s}))\d W_s.
\end{align*}
Next, using Burkholder--Davis--Gundy's inequality, if $\varepsilon<1$ it follows that
\begin{align*}
	\E^\P\big[\|\delta \bar X\|_\infty^2\big] & \le \frac{1}{1-\varepsilon}\E^\P\bigg[\int_0^T\Big(\frac{\ell_B^2 + 4\ell_\sigma^2C_{\rm BDG}}{\varepsilon} + \ell_\sigma^2 - K_B\Big)\|\delta \bar X_{\cdot\wedge t}\|^2_\infty \d t \bigg] + \frac{\varepsilon}{1-\varepsilon}\E^\P\bigg[\int_0^T\big(\|\delta z_t\|^2 + \cW_2^2(\xi^1_t, \xi^2_t)\big)\d t \bigg].
\end{align*}
Since 
\begin{equation*}
	K_B\ge \ell_B^2 + (4C_{\rm BDG}+1)\ell_\sigma^2= \inf_{0<\varepsilon<1}\bigg\{\frac{\ell_B^2 + 4\ell_\sigma^2C_{\rm BDG}}{\varepsilon} + \ell_\sigma^2 \bigg\},
\end{equation*}
it follows that
\begin{align}
\label{eq:estim.bar.X}
	\E^\P\big[\|\delta \bar X\|_\infty^2\big] & \le  \frac{\varepsilon}{1-\varepsilon}\E^\P\bigg[\int_0^T\big(\|\delta z_t\|^2 + \cW_2^2(\xi^1_t, \xi^2_t)\big)\d t \bigg].
\end{align}
For the backward equation, applying It\^o's formula to $\e^{\beta t}\|\delta \bar Y_t\|^2 \coloneqq \e^{\beta t}\|\bar Y^1_t - \bar Y^2_t\|^2$ and using Lipschitz-continuity of $G$ and $F$ and Young's inequality, for every $\eta>0$ we have
\begin{align*}
	\E^\P\bigg[\e^{\beta t}\|\delta \bar Y_t\|^2 + (1 - \eta)\int_t^T\e^{\beta s}\|\delta \bar Z_s\|^2\d s\bigg] &\le 2\e^{\beta T}\ell_G^2\E^\P\Big[\|\delta \bar X\|_\infty^2 + \cW_2^2(\cL_{\P^{\smalltext{z}^\tinytext{1}\smalltext{,}\smalltext{\xi}^\tinytext{1}}}(X), \cL_{\P^{z^2,\xi^2}}(X))\Big] + \E^\P\bigg[\int_t^T\Big(\frac{\ell_B^2}{\eta} + \eta - \beta\Big)\|\delta \bar Y_s\|^2\d s \bigg] \\
	& + \eta \E^\P\bigg[\int_t^T\Big(\|\delta \bar X_{\cdot\wedge s}\|_\infty^2 + \cW_2^2\big( \cL_{\P^{\smalltext{z}^\tinytext{1}\smalltext{,}\smalltext{\xi}^\tinytext{1}}}(X_{\cdot\wedge s}, Y^1_s, Z^1_s), \cL_{\P^{\smalltext{z}^\tinytext{2}\smalltext{,}\smalltext{\xi}^\tinytext{2}}}(X_{\cdot\wedge s}, Y^2_s, Z^2_s)\big)\Big)\d s \bigg],
\end{align*}
where we put $\|\delta \bar Z_t\| \coloneqq \|\bar Z^1_t - \bar Z^2_t\|$.
Using \eqref{eq:estim.newspace.time} and \eqref{eq:estim.newspace} we have
\begin{align*}
	\E^\P\bigg[\e^{\beta t}\|\delta \bar Y_t\|^2 &+ (1 - 2\eta)\int_t^T\e^{\beta s}\|\delta \bar Z_s\|^2\d s\bigg] \le 4\e^{\beta T}\ell_G^2\E^\P\big[\|\delta \bar X\|_\infty^2 \big] + \E^\P\bigg[\int_t^T\Big(\frac{\ell_B^2}{\eta} + 2\eta - \beta\Big)\|\delta \bar Y_s\|^2\d s \bigg]
	 + \eta \E^\P\bigg[\int_t^T2\|\delta \bar X_{\cdot\wedge s}\|_\infty^2 \d s \bigg].
\end{align*}
If $\eta<1/2$ and $\beta$ large enough, by \eqref{eq:estim.bar.X} we then have
\begin{align*}
	\E^\P\bigg[\e^{\beta t}\|\delta \bar Y_t\|^2 + (1 - 2\eta)\int_t^T\e^{\beta s}\|\delta \bar Z_s\|^2\d s\bigg] &\le \big( 4\e^{\beta T}\ell_G^2 + 2T\eta\big)\E^\P\big[\|\delta \bar X\|_\infty^2 \big]\\
	&\le \frac{\varepsilon}{1-\varepsilon}\big( 4\e^{\beta T}\ell_G^2 + 2T\eta\big)\E^\P\bigg[\int_0^T\big(\|\delta z_t\|^2 + \cW_2^2(\xi^1_t, \xi^2_t)\big)\d t \bigg],
\end{align*}
with
\begin{equation*}
	C_{\varepsilon,1} \coloneqq \frac{1}{(1-\varepsilon)(1-2\eta)}\big( 4\e^{\beta T}\ell_G^2 + 2T\eta\big).
\end{equation*}
Thus, it follows from \eqref{eq:estim.newspace.time} and \eqref{eq:estim.newspace} that there is a constant $C>0$ such that
\begin{align}
\label{eq:estimwasser}
	\cW_2^2\big( \cL_{\PP^{\smalltext{z}^{\tinytext{1}}\smalltext{,} \smalltext{\xi}^{\tinytext{1}}}}(X), \cL_{\PP^{\smalltext{z}^{\tinytext{2}}\smalltext{,} \smalltext{\xi}^{\tinytext{2}}}}(X) \big) +\int_0^T\cW_2^2\big( \cL_{\PP^{\smalltext{z}^{\tinytext{1}}\smalltext{,} \smalltext{\xi}^{\tinytext{1}}}}(X_{\cdot\wedge t},Y^1_t,Z^1_t)&, \cL_{\PP^{\smalltext{z}^{\tinytext{2}}\smalltext{,} \smalltext{\xi}^{\tinytext{2}}}}(X_{\cdot\wedge t},Y^2_t,Z^2_t) \big) \d t \le \varepsilon C_{\varepsilon,1}\E^\P\bigg[\int_0^T\big(\|\delta z_t\|^2 + \cW_2^2(\xi^1_t, \xi^2_t)\big)\d t \bigg].
\end{align}

Thus, coming back to \eqref{eq:exits.first.estim}, we continue the estimation as
}
	\begin{align*}
	\notag
		& \mathrm{e}^{\beta \tau}\|\delta Y_\tau\|^2  + (1-2\eps)\E^\P\bigg[\int_\tau^T\mathrm{e}^{\beta u}\|\delta Z_u\|^2\mathrm{d}u \bigg| \mathcal{F}_\tau\bigg]  + \bigg(\beta -\frac{\|B\|_\infty^2}\eps-2\ell_F\bigg(1+\frac{2\ell_F}{\eps}\bigg) \bigg)\E^\P\bigg[\int_\tau^T\mathrm{e}^{\beta u}\|\delta Y_u\|^2\mathrm{d}u\bigg| \mathcal{F}_\tau \bigg]\\
		\notag	&\le \varepsilon C_{\varepsilon,1}\big( {\mathrm{e}^{\beta T}}\ell_{G(\mu)}^2 + 1\big)\E^\P\bigg[\int_0^T\mathrm{e}^{\beta u}\big(\cW_2^2(\xi^1_u, \xi^2_u)+  \|\delta z_u\|^2 \big)\mathrm{d}u\bigg] + \frac{1}{\eps}\E^\P\bigg[\int_\tau^T\mathrm{e}^{\beta u}\|\delta Y_u Z^1_u\|^2\d u \bigg|\cF_\tau\bigg]\\ 
		\notag	&\le \varepsilon C_{\eps,2}\bigg(\|\delta z\|^2_{\H^{\smalltext{2}\smalltext{,}\smalltext{\beta}}_{\text{\fontsize{4}{4}\selectfont$\mathrm{BMO}$}}(\R^{\smalltext{m}\smalltext{\times} \smalltext{d}},\F)} + \int_0^T\mathrm{e}^{\beta u} \cW_2^2(\xi^1_u, \xi^2_u)\d u \bigg) + \frac1\eps \big\|\mathrm{e}^{\frac{\beta}{2} \cdot}\delta Y_{\cdot}\big\|^2_{\S^\smalltext{\infty}(\R^\smalltext{m},\F)}C^\eta_{T,G,F},
	\end{align*}
	with $C_{\eps,2} \coloneqq C_{\varepsilon,1}\big( {\mathrm{e}^{\beta T}}\ell_{G(\mu)}^2 + 1\big)$,	and where we used that 
	\begin{equation*}
		\E^\P\bigg[\int_\tau^T\|Z_u^1\|^2\d u\bigg|\Fc_\tau\bigg]\le \| Z\|^2_{\H^{\smalltext{2}}_{\text{\fontsize{4}{4}\selectfont$\mathrm{BMO}$}}(\R^{\smalltext{m}\smalltext{\times} \smalltext{d}},\F)}\le C^\eta_{F,G,T},
	\end{equation*}
	with the constant $C^\eta_{F,G,T}\coloneqq  \frac{2\mathrm{e}^{\beta T}}{\ell_\eta}\big(\|G\|^2_{\L^\smalltext{\infty}(\R^\smalltext{m},\Fc_{\text{\fontsize{4}{4}\selectfont$T$}})}+\frac{\eta }\beta\|F^0 \|_\infty^2 + \eta T C_X\big)$ defined in {\rm \Cref{eq:bound.C.tgf}} with $C_X$ given by {\rm \Cref{eq:cpntrolxx}}.
	
	\medskip
	If the bound $\|G\|^2_{\L^\smalltext{\infty}(\R^\smalltext{m},\Fc_{\text{\fontsize{4}{4}\selectfont$T$}})}$ is small enough, we can find $\eta$ such that $C^\eta_{F, G, T}<\frac\eps2$.
	Choosing $\beta>\frac{\|B\|_\infty^2}\eps+2\ell_F\big(1 + \frac{2\ell_F}{\eps}\big)$ and since $\tau\in\Tc(\F)$ was arbitrary, and we have $\|\cdot\|_{\H^\smalltext{2}(\R^{\smalltext{m}\smalltext{\times} \smalltext{d}},\F)} \le \|\cdot\|_{\H^\smalltext{2}_{\text{\fontsize{4}{4}\selectfont$\mathrm{BMO}$}}(\R^{\smalltext{m}\smalltext{\times} \smalltext{d}},\F)}$, this implies in particular $
		\|\delta Z\|^2_{\H^{\smalltext{2}\smalltext{,}\smalltext{\beta}}_{\text{\fontsize{4}{4}\selectfont$\mathrm{BMO}$}}(\R^{\smalltext{m}\smalltext{\times} \smalltext{d}},\F)}  	
		\le  \frac{\varepsilon C_{\eps,2}}{1-2\eps}\big(\|\delta z\|^2_{\H^{\smalltext{2}\smalltext{,}\smalltext{\beta}}_{\text{\fontsize{4}{4}\selectfont$\mathrm{BMO}$}}(\R^{\smalltext{m}\smalltext{\times} \smalltext{d}},\F)} + \cW^2_{2,\beta,[0,T]}(\xi, \xi^\prime)\big),
$
as well as
$
		\|\mathrm{e}^{\frac{\beta}{2} \cdot}\delta Y_\cdot\|^2_{\S^\smalltext{\infty}(\R^\smalltext{m},\F)} \leq 
		2\varepsilon C_{\eps,2}\big(\|\delta z\|^2_{\H^{\smalltext{2}\smalltext{,}\smalltext{\beta}}_{\text{\fontsize{4}{4}\selectfont$\mathrm{BMO}$}}(\R^{\smalltext{m}\smalltext{\times} \smalltext{d}},\F)} + \cW^2_{2,\beta,[0,T]}(\xi, \xi^\prime)\big).
$

\medskip
	It remains to control $\int_0^T\mathrm{e}^{\beta u} \cW_2^2\big( \cL_{\PP^{\smalltext{z}^{\tinytext{1}}\smalltext{,} \smalltext{\xi}^{\tinytext{1}}}}(X_{\cdot\wedge u}, Y^1_u,Z^1_u), \cL_{\PP^{\smalltext{z}^{\tinytext{2}}\smalltext{,} \smalltext{\xi}^{\tinytext{2}}}}(X_{\cdot\wedge u},Y^2_u,Z^2_u) \big)\d u$, which is already done in \Cref{eq:estimwasser}. 
	Thus, choosing $\varepsilon>0$ small enough, it follows that 
	the mapping $\Phi$ is a contraction, and thus admits a unique fixed-point in $ \mathbb{S}^\infty(\R^m,\F) \times \H^{2{\color{black},\beta}}_{\mathrm{BMO}}(\R^{m\times d},\F) \times \mathfrak P^{m\times \bmdim}$.

\medskip
	\emph{Step 5: case of smooth terminal conditions.}
	In this last step we further assume that $G$ depends on the terminal value of $X$ and its law.
	Then, applying It\^o's formula to $G(X_t, \cL_{\overline \P}(X_t))$, see \emph{e.g.} \citeauthor*{carmona2013probabilistic} \cite[Theorem 5.104]{carmona2013probabilistic}, we have
\begin{align*}
	G(X_T, \cL_{\overline \P}(X_T)) 
	&= G(X_t, \cL_{\overline \P}(X_t)) + \int_t^T\partial_xG(X_u,\cL_{\overline \P}(X_u))\sigma_u(X_{ u})\diff \overline W_u + \frac12\int_t^T\mathrm{Tr}\big[\partial_{xx}G(X_u,\cL_{\overline \P}(X_u))\sigma_u(X_{ u})\sigma_u(X_{ u})^\top\big]\diff u\\
	&\quad + \frac12\int_t^T\int_{\mathbb{R}^\smalltext{\xdim}}\mathrm{Tr}\big[ \partial_a\partial_\mu G(X_u, \cL_{\overline \P}(X_u))(a)\sigma_u(a)\sigma_u(a)^\top \big] \cL_{\overline \P}(X_t)(\diff a)\diff t \\
	&\quad- \int_t^T B_u(X_{\cdot\wedge u}, Q_u, \cL_{\overline \PP}(X_u,P_u, Q_u))\cdot \partial_xG(X_u,\cL_{\overline \P}(X_u))\sigma_u(X_{ u})\diff u.
\end{align*}
Plugging this back into the equation to be solved, \emph{i.e.} \rm\Cref{eq:Gen BSDE}, we have
\begin{align*}
	\displaystyle P_t &- G(X_t,\cL_{\overline \P}(X_t))  = \int_t^T\bigg(F_u\big(X_{\cdot\wedge u}, P_u,Q_u, \cL_{\overline\PP}(X_u,P_u,Q_u)\big) + \frac12\mathrm{Tr}\big[\partial_{xx}G(X_u,\cL_{\overline \P}(X_u))\sigma_u(X_{ u})\sigma_u(X_{ u})^\top\big]\bigg)\diff u\\
	& + \frac12\int_t^T\bigg(\int_{\mathbb{R}^\smalltext{\xdim}}\mathrm{Tr}\big[ \partial_a\partial_\mu G(X_u, \cL_{\overline \P}(X_u))(a)\sigma_u(a)\sigma_u(a)^\top \big] \cL_{\overline \P}(X_t)(\diff a) -  B_u(X_{\cdot\wedge u}, Q_u, \cL_{\overline \PP}(P_u, Q_u))\cdot \partial_xG(X_u,\cL_{\overline \P}(X_u))\sigma_u(X_{ u})\bigg)\diff u\\
	& - \int_t^T\big(Q_u-\partial_xG(X_u,\cL_{\overline \P}(X_u))\sigma_u(X_{ u})\big)\diff \overline W_u ,\; t\in[0,T],\; \overline \PP\text{\rm--a.s.}
\end{align*}
Let us put $\widetilde Q_t\coloneqq  Q_t-\partial_xG(X_t,\cL_{\overline \P}(X_t))\sigma_t(X_{ t})$ and $\widetilde P_t\coloneqq  P_t - G(X_t,\cL_{\overline \P}(X_t))$.
Then, $(P, Q) \in \cap_{p\ge0} \mathbb{S}^p(\R^m,\F) \times \H^{2}_{\mathrm{BMO}}(\R^m,\F)$ solves \rm\Cref{eq:Gen BSDE} if and only if $(\widetilde P, \widetilde Q) \in \mathbb{S}^\infty(\R^m,\F) \times \H^{2}_{\mathrm{BMO}}(\R^m,\F)$ solves \rm\Cref{eq:Gen BSDE} with $F$ replaced by $\widetilde F$ defined in \rm\Cref{eq:def.tilde.F} and $G=0$.
Note that the integrability property of $P$ follows from the fact that $\widetilde P\in \mathbb{S}^\infty(\R^m,\F)$, $G$ has linear growth and $X$ has every moments.
By \emph{Step 4}, $(\widetilde P, \widetilde Q)$ is the unique solution in $\mathbb{S}^\infty(\R^m,\F) \times \H^{2}_{\mathrm{BMO}}(\R^m,\F)$ of the BSDE with terminal condition zero and generator $\widetilde F$, hence $(P,Q)$ is the unique solution in $\cap_{p\ge0} \mathbb{S}^p(\R^m,\F) \times \H^{2}_{\mathrm{BMO}}(\R^m,\F)$ of the BSDE with terminal condition $G$ and generator $F$.
\medskip

\emph{Step 6: uniqueness in $\mathbb{H}^2(\R^m,\F) \times \H^{2}(\R^m,\F)$.}

\medskip

If $(P,Q) \in \mathbb{H}^2(\R^m,\F) \times \H^{2}(\R^m,\F)$, then, we have by boundedness of $B$ that $(P,Q) \in \mathbb{S}^\infty(\R^m,\F) \times \H^{2}_{\mathrm{BMO}}(\R^m,\F)$, see \Cref{lem:a.priori.estimate.general}.
Thus if $\|G\|_{\mathbb{L}^\smalltext{2}(\mathbb{R}^\smalltext{m},\cF_T)}$ and $\ell_G$ are small enough, then uniqueness in $\mathbb{H}^2(\R^m,\F) \times \H^{2}(\R^m,\F)$ follows from \emph{Step 4}.

\medskip
Let us now assume that $G$ satisfies \Cref{assum:gen.MkV}.$(v)$.
Given any two solutions $(P^1, Q^1)$, and  $(P^2, Q^2)$ in $\mathbb{H}^2(\R^m,\F) \times \H^{2}(\R^m,\F)$, defining \[
\widetilde Q^i_t\coloneqq  Q^i_t-\partial_xG(X_t,\cL_{\overline \P}(X_t))\sigma_t(X_{ t}),\; \widetilde P^i_t\coloneqq  P^i_t - G(X_t,\cL_{\overline \P}(X_t)),\;  i\in\{1,2\},\; t\in[0,T],
\] it follows that $(\widetilde P^1, \widetilde Q^1)$ and $(\widetilde P^2, \widetilde Q^2)$ are in $\mathbb{H}^2(\R^m,\F) \times \H^{2}(\R^m,\F)$ and solve the BSDE with terminal condition zero and generator $\widetilde F$. 
By the above argument, this equation admits a unique solution in $\mathbb{H}^2(\R^m,\F) \times \H^{2}(\R^m,\F)$, showing that $\widetilde P^1 = \widetilde P^2$ and $\widetilde Q^1 = \widetilde Q^2$ $\diff t\otimes \P$--a.e., showing uniqueness.
This concludes the proof.
\end{proof}

{\footnotesize
\bibliography{bibliographyDylan}}

\end{document}